\documentclass[12pt,a4paper,notitlepage]{article}
\usepackage{fullpage, amsfonts, amsthm, amsmath, setspace, graphicx, amssymb, float, colonequals, hyperref}

\newtheorem{theorem}{Theorem}[section]
\newtheorem{lemma}[theorem]{Lemma}

\newtheorem{corollary}[theorem]{Corollary}
\newtheorem{case}{Case}[theorem]
\newtheorem{subcase}{Sub-case}[case]

\numberwithin{subcase}{case}
\numberwithin{subsubcase}{subcase}
\numberwithin{claim}{theorem}

\newenvironment{definition}[1][Definition:]{\begin{trivlist}
\item[\hskip \labelsep {\bfseries #1}]}{\end{trivlist}}

\newenvironment{example}[1][Example:]{\begin{trivlist}
\item[\hskip \labelsep {\bfseries #1}]}{\end{trivlist}}

\newenvironment{note}[1][Note:]{\begin{trivlist}
\item[\hskip \labelsep {\bfseries #1}]}{\end{trivlist}}

\newenvironment{theorem-mannum}[2][Theorem:]{\begin{trivlist}
\item[\hskip \labelsep {\bfseries #1}\hskip \labelsep {\bfseries #2}]\em}{\end{trivlist}}

\newenvironment{corollary-mannum}[2][Corollary:]{\begin{trivlist}
\item[\hskip \labelsep {\bfseries #1}\hskip \labelsep {\bfseries #2}]\em}{\end{trivlist}}

\newcommand{\ds}{\displaystyle}
\title{\bf Level curve configurations and conformal equivalence of meromorphic functions.}
\author{Trevor Richards\thanks{Email: richardst@wlu.edu}\\\vspace{6pt} {\em Department of Mathematics, Washington and Lee University}\\ {\em Lexington, United States}}

\begin{document}

\maketitle
\begin{abstract}
Let $f=B_1/B_2$ be a ratio of finite Blaschke products having no critical points having no critical points on $\partial\mathbb{D}$.  Then $f$ has finitely many critical level curves (level curves containing critical points of $f$) in the disk, and the non-critical level curves of $f$ interpolate smoothly between the critical level curves.  Thus, to understand the geometry of all the level curves of $f$, one need only understand the configuration of the finitely many critical level curves of $f$.  In this paper we show that in fact such a function $f$ is determined not just geometrically but conformally by the configuration of its critical level curves.  That is, if $f_1$ and $f_2$ have the same configuration of critical level curves, then there is a conformal map $\phi$ such that $f_1\equiv f_2\circ\phi$.  We then use this to show that every configuration of critical level curves which could come from an analytic function is instantiated by a polynomial.  We also include a new proof of a theorem of B\^{o}cher (which is an extension of the Gauss--Lucas theorem to rational functions) using level curves.
\medskip

\end{abstract}
\section{HISTORY AND OVERVIEW}%

A great deal of work has been done on the geometry of level curves of an analytic (or meromorphic) function, especially concerning on the one hand issues such as convexity, star-shapeness, arc-length, and area (see for example~\cite{EHP,G,HW,P}), and on the other hand the relationship between functions which share a level curve (see for example~\cite{C,H,V}).  Inquiries of the latter sort culminated with the level curve structure theorem of Stephenson~\cite{St} in 1986, which implied many of the earlier results.  A nice summation of this may be found in~\cite{StSu} in which Stephenson and Sundberg also give a general result for inner functions sharing a level curve.

Here we highlight two other areas where analysis of level curves may be applied, and on which the present work bears.   We then give a brief overview of the results found in this paper.

\subsection{LEVEL CURVES AND THE FINGERPRINT OF A SHAPE}

The "fingerprint" $k$ which a smooth simple closed curve in $\mathbb{C}$ imposes on the unit circle $\mathbb{T}$ was introduced by A. A. Kirillov~\cite{K1, K2} in 1987, and is defined as follows.  Let $\Gamma$ be a smooth simple closed curve in $\mathbb{C}$, with bounded face $\Omega_-$ and unbounded face $\Omega_+$.  Let $\phi_-,\phi_+$ denote Riemann maps from $\mathbb{D},\mathbb{D}_+$ to $\Omega_-,\Omega_+$ respectively (here $\mathbb{D}_+$ is defined as $\mathbb{C}\setminus\text{cl}(\mathbb{D})$).  With certain normalizations on the Reimann maps, we define the fingerprint $k$ of $\Gamma$ by $k\colonequals{\phi_+}^{-1}\circ\phi_-$.  Since $\Gamma$ is smooth it is easy to show that $k$ is a diffeomorphism from $\mathbb{T}$ to $\mathbb{T}$.  Moreover, if $\hat{\Gamma}$ is the image of $\Gamma$ under an affine transformation $f(z)=az+b$, with corresponding fingerprint $\hat{k}$, then $k=\hat{k}\circ\psi$ for some automorphism $\psi:\mathbb{D}\to\mathbb{D}$.  Therefore we may define a function $\mathcal{F}$ which maps smooth simple closed curves (modulo composition with affine transformation) to the corresponding diffeomorphism of $\mathbb{T}$ which is its fingerprint (modulo pre-composition with an automorphism of $\mathbb{D}$).  (Note: this and more background may be found in~\cite{EKS}.)  Kirillov proved the following theorem~\cite{K1, K2}.

\begin{theorem}\label{theorem:Kirillov theorem.}
$\mathcal{F}$ is a bijection.
\end{theorem}

If we restrict our attention to smooth curves which arise as level curves of polynomials, a similar result may be obtained.  One first shows that if $\Gamma$ is a proper polynomial lemniscate (ie $\Gamma=\{z:|p(z)|=1\}$ for some $n$-degree polynomial $p$ such that $\{z:|p(z)|=1\}$ is smooth and connected) then the corresponding fingerprint is of  the form $k=B^{\frac{1}{n}}$ for some $n$-degree Blaschke product $B$.  If we let $\mathcal{F}_p$ denote the function $\mathcal{F}$ viewed as having as its domain the smooth simple closed curves which arise as proper polynomial lemniscates (modulo composition with an affine transformation), and having as its codomain the diffeomorphisms of $\mathbb{T}$ consisting of $n^{\text{th}}$ roots of $n$-degree Blaschke products (modulo pre-composition with an automorphism of $\mathbb{D}$), then one may prove the following theorem.

\begin{theorem}\label{theorem:Kirillov polynomial theorem.}
$\mathcal{F}_p$ is a bijection.
\end{theorem}

This result was stated and proved by Ebenfelt, Khavinson, and Shapiro in~\cite{EKS}.  The injectivity of $\mathcal{F}_p$ is an immediate consequence of Theorem~\ref{theorem:Kirillov theorem.}.  The surjectivity claim in Theorem~\ref{theorem:Kirillov polynomial theorem.} is equivalent to the following Corollary~\ref{cor: Polynomial equivalence.} in the present paper, which follows from our work in Section~\ref{sect: Pi is surjective: The generic case.}.

\begin{corollary-mannum}{\ref{cor: Polynomial equivalence.}}
For every finite Blaschke product $B$ of degree $n$, there is some $n$ degree polynomial $p$ such that the set $G\colonequals\{z:|p(z)|<1\}$ is connected, and some conformal map $\phi:\mathbb{D}\to{G}$ such that $B\equiv p\circ\phi$ on $\mathbb{D}$.
\end{corollary-mannum}

\subsection{LEVEL CURVES AND GREEN'S FUNCTIONS}%

Let $f$ be a meromorphic function with simply connected domain $D$.  Let $\Lambda$ be a bounded component of $\{z\in D:|f(z)|=1\}$ which does not intersect $\partial D$ and does not contain a zero of $f'$.  Then $\Lambda$ is an analytic simple closed curve in $D$.  Let $G$ denote the bounded face of $\Lambda$.  Since $\partial G$ is analytic, we may find a conformal map $\phi:\mathbb{D}\to G$ which extends analytically to $\partial\mathbb{D}$.  If we now pull $f$ back to $\mathbb{D}$ by $\widetilde{f}=f\circ\phi$, we obtain a non-constant function $\widetilde{f}$ which is meromorphic on the closure of $\mathbb{D}$ and has modulus $1$ on $\partial\mathbb{D}$, that is, a ratio of finite Blaschke products $\widetilde{f}=\widetilde{f_1}/\widetilde{f_2}$.  Decompose $\widetilde{f_1}$ and $\widetilde{f_2}$ into their component degree $1$ Blaschke products to obtain

\[
\widetilde{f}(z)=\ds\dfrac{\prod_{i=1}^M\widetilde{A_i}(z)}{\prod_{i=1}^N\widetilde{B_i}(z)}.
\]

We will now push each of these component degree~$1$ Blaschke products forward to $G$, and write $A_i=\widetilde{A_i}\circ\phi^{-1}$ and $B_i=\widetilde{B_i}\circ\phi^{-1}$.  Each $A_i$ and $B_i$ has a single zero in $G$ and has modulus $1$ on $\partial G$, and it is not hard to see that each $\ln|A_i|$ and $\ln|B_i|$ is a Green's function for $G$, with poles at the zero of $A_i$ and $B_i$ respectively.  Therefore $\ln|f|$ is an integer linear combination of Green's functions of $G$

\[
\ln|f|=\ds\sum_{i=1}^M\ln|A_i|-\sum_{i=1}^N\ln|B_i|.
\]

Conversely, if $\ds\sum_{i=1}^Pc_ig_i(z)$ is any integer linear compination of Green's functions of $G$, via the inverse process of that just described we obtain that $\ds\sum_{i=1}^Pc_ig_i(z)=\ln|h|$ for some function $h$ which is meromorphic on the closure of $G$, and which has $\partial G$ as a level curve.  Moreover, for any meromorphic function $h$, the critical points and level curves of $h$ are identical to the critical points and level curves of the harmonic function (with isolated singularities) $\ln|h|$.  Therefore the study in this paper of function elements $(f,G)$ (that is, the function $f$ with domain $G$) applies also to integer linear combinations of Green's functions of the region $G$.  (One may see this correspondence at work in~\cite{W1}, in which Walsh translates many results involving critical points of polynomials into results for Green's functions.)

\subsection{OVERVIEW}%

Our main goal in this paper is to explore the ways in which the configuration of level curves of a meromorphic function characterizes that function modulo conformal equivalence.

We begin in Section~\ref{sect: Preliminary results.} with several preliminary results on the bounded level curves of a meromorphic function.  In particular, we consider the level curves of a meromorphic function $f$ with domain $G$ such that

\begin{itemize}
\item
$G$ is open, bounded, and simply connected.
\item
$f$ may be extended to a meromorphic function on $cl(G)$.
\item
$|f|=1$ on $\partial G$.
\item
$f'\neq0$ on $\partial G$.
\end{itemize}

(Note that $|f|=1$ and $f'\neq0$ on $\partial G$ together imply that $\partial G$ is smooth.  Also, if $G$ is finitely connected then all the results of Section~\ref{sect: Preliminary results.} still hold, however, for simplicity's sake, we assume that $G$ is simply connected.)

If one pulls such a function back to the disk via a conformal map, one obtains a ratio of finite Blaschke products of different degrees.  Therefore we call the pair $(f,G)$ a generalized finite Blaschke ratio.

In Section~\ref{sect:Possible level curve configs.} we build up the notion of a level curve configuration for a generalized finite Blaschke ratio, and construct a set, which we will call $PC$, whose members will represent the possible configurations of the critical level curves of $(f,G)$ (that is, the level curves of $f$ in $G$ which contain critical points of $f$).  We then define a function $\Pi$ which maps $(f,G)$ to the corresponding configuration in $PC$.

Section~\ref{sect:Pi respects sim.} contains the following result, which shows that the data preserved in the critical level curve configuration of a generalized finite Blaschke ratio determines the function up to conformal equivalence.

\begin{theorem-mannum}{\ref{thm:Conformal equivalence iff Pi equivalence.}}
For two generalized finite Blaschke ratios $(f_1,G_1)$ and $(f_2,G_2)$, $(f_1,G_1)\sim(f_2,G_2)$ if and only if $\Pi(f_1,G_1)=\Pi(f_2,G_2)$.
\end{theorem-mannum}

Here $(f_1,G_1)\sim(f_2,G_2)$ means that there is some conformal map $\phi:G_1\to G_2$ such that $f_1\equiv f_2\circ\phi$ on $G_1$ (clearly this $\sim$, which we call "conformal equivalence" is an equivalence relation on the set of generalized finite Blaschke ratios).  Theorem~\ref{thm:Conformal equivalence iff Pi equivalence.} implies that if we view $\Pi$ as having for its domain the set of equivalence classes of generalized finite Blaschke ratios modulo conformal equivalence then, first, $\Pi$ is well defined, and second, $\Pi$ is injective.  This result is similar in some respects to the way in which the dynamical properties of a postcritically finite polynomial are preserved in the corresponding Hubbard tree \cite{Poi}.

In Sections~\ref{sect: Pi is surjective: The generic case.}~and~\ref{sect: Pi is surjective: The general case.} we show that, in a limited sense, $\Pi$ is surjective.  That is, we define a subset $PC_a\subset{PC}$ of configurations which naturally correspond to the level curve configurations of analytic functions.  If we view $\Pi$ as having for its domain the equivalence classes of generalized finite Blaschke ratios $(f,G)$ with analytic $f$, and having codomain $PC_a$, then $\Pi$ is surjective.   From this we will deduce Corollary~\ref{cor: Polynomial equivalence.} mentioned above.

\section{PRELIMINARY RESULTS}\label{sect: Preliminary results.}

We begin with a discussion of the bounded level curves of a meromorphic function.  We give brief justifications for these results, as they follow from elementary properties of meromorphic functions.  Throughout this section, we let $(f,G)$ denote some fixed generalized finite Blaschke ratio.

Let $\Lambda$ denote a level curve of $f$ in $G$ (that is, a component of the set $\{z\in G:|f(z)|=\epsilon\}$ for some constant $\epsilon>0$).  $\Lambda$ is a finite connected graph, whose vertices are points of non-injectivity of $f$, namely zeros of $f'$.  Several things may be said about which graphs may appear as critical level curves $\Lambda$.  If $z$ is a critical point of multipicity $n$ of $f$, the ramification of $f$ at $z$ is of degree $n+1$, and thus the level curve of $f$ containing $z$ has $2(n+1)$ edges meeting at $z$ (throughout the paper we will count an edge twice if both its ends are at $z$).  Thus there are evenly many edges incident to each vertex of $\Lambda$.  It can easily be shown that this fact implies that each edge of $\Lambda$ is incident to two distinct faces of $\Lambda$.  In Figure~\ref{fig:AdmissibleGraphs} we have several graphs which might appear as critical level curves of $f$ in $G$, and Figure~\ref{fig:NonAdmissibleGraphs} shows several graphs which may not appear as critical level curves of $f$ in $G$ (all modulo homeomorphism).

\begin{figure}[H]
\begin{minipage}[b]{0.45\linewidth}
\centering
	\includegraphics[width=\textwidth]{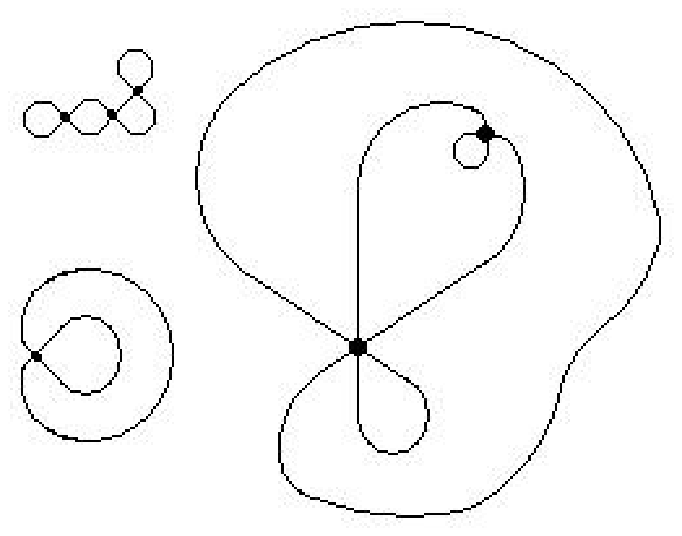}
	\caption{Admissible Graphs}
	\label{fig:AdmissibleGraphs}
\end{minipage}
\hspace{0.5cm}
\begin{minipage}[b]{0.45\linewidth}
\centering
	\includegraphics[width=\textwidth]{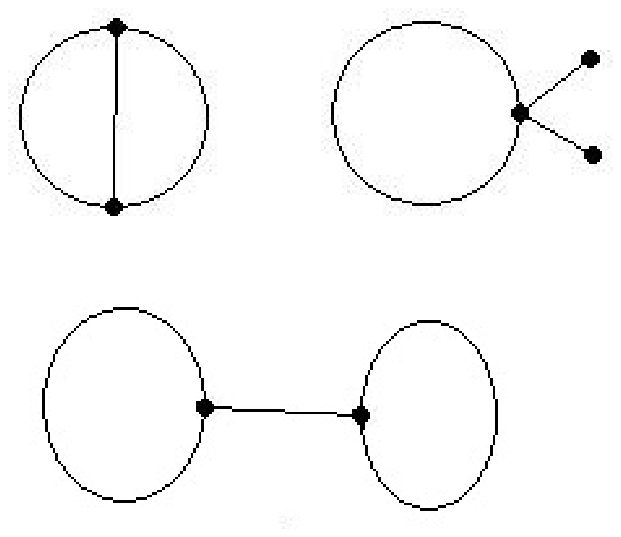}
	\caption{Non Admissible Graphs}
	\label{fig:NonAdmissibleGraphs}
\end{minipage}
\end{figure}

Here are several concrete examples as well.

\begin{example}
Let $f(z)=z^n$ for some $n\in\mathbb{Z}\setminus\{0\}$, and let $\epsilon>0$ be given.  Then the set $\{z\in\mathbb{C}:|f(z)|=\epsilon\}$ is the circle centered at $0$ with radius $\epsilon^{\frac{1}{n}}$.
\end{example}

\begin{example}
Let $f(z)=z^5-1$.  In Figure~\ref{fig:p(z)=z^5-1levelcurves} below we see the level sets $\{z\in\mathbb{C}:|f(z)|=\epsilon\}$, for $\epsilon=0.5,1,\text{ and }1.5$.
\end{example}

\begin{figure}[H]
	\centering
		\includegraphics[width=.5\textwidth]{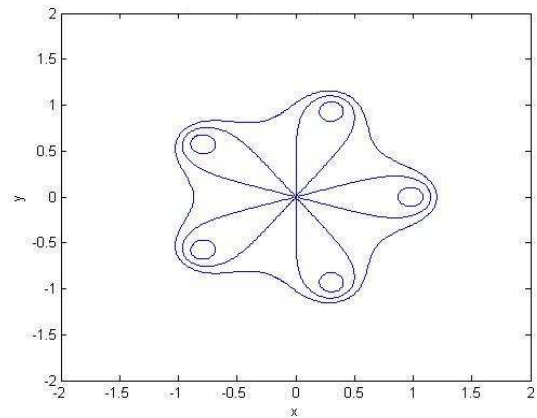}
	\caption{$f(z)=z^5-1$}
	\label{fig:p(z)=z^5-1levelcurves}
\end{figure}

We may use the facts about level curves mentioned above to give a new proof of a theorem of B\^{o}cher~\cite[pg. 97]{W2}.  The part which we prove here is an extension of the Gauss--Lucas theorem to certain rational functions.  B\^{o}cher's proof (and the normal proof of the Gauss--Lucas theorem) is analytic, making use of logarithmic differentiation.  What follows appears to be the first geometric proof of these results.  First a definition.

\begin{definition}
For $w\in{G}$, define $\Lambda_w$ to be the component of $\{z\in G:|f(z)|=|f(w)|\}$ which contains $w$.
\end{definition}

\begin{theorem}[B\^{o}cher's Theorem]\label{thm: Bocher's theorem.}
Let $T_1,T_2$ be two circles in the Riemann sphere $\bar{\mathbb{C}}$.  For $i,j\in\{1,2\}$ with $i\neq j$, let $D_i$ denote the face of $T_i$ which does not contain $T_j$.  If $R$ is a degree $n\geq1$ rational function with all of its zeros in $D_1$ and all of its poles in $D_2$, then all of the critical points of $R$ are contained in $D_1\cup D_2$.
\end{theorem}

Note that in the case where $R$ is a polynomial this becomes the content of the Gauss--Lucas theorem.

\begin{proof}
Let us suppose by way of contradiction that there is some critical point not in either $D_1$ or $D_2$.  By pre-composing with an appropriate M\"obius transformation, we may assume that this critical point is at the origin and that $T_1$ and $T_2$ are bounded with their centers on the negative real axis and positive real axis respectively, and that neither $T_1$ nor $T_2$ are within $1$ of the origin.  Since $R$ is at least $2$-to-$1$ in a neighborhood of $0$, there are at least $4$ edges of $\Lambda_0$ (the level curve of $R$ containing $0$) intersecting at $0$.  Therefore there is some horizontal line segment $L_c=\{x+ic:x\in[-1,1]\}$ which intersects $\Lambda_0$ in at least two distinct points.  Let $u_1=x_1+ic$ and $u_2=x_2+ic$ denote these two points (labelled so that $x_1<x_2$).

\begin{figure}[H]
	\centering
		\includegraphics[width=.8\textwidth]{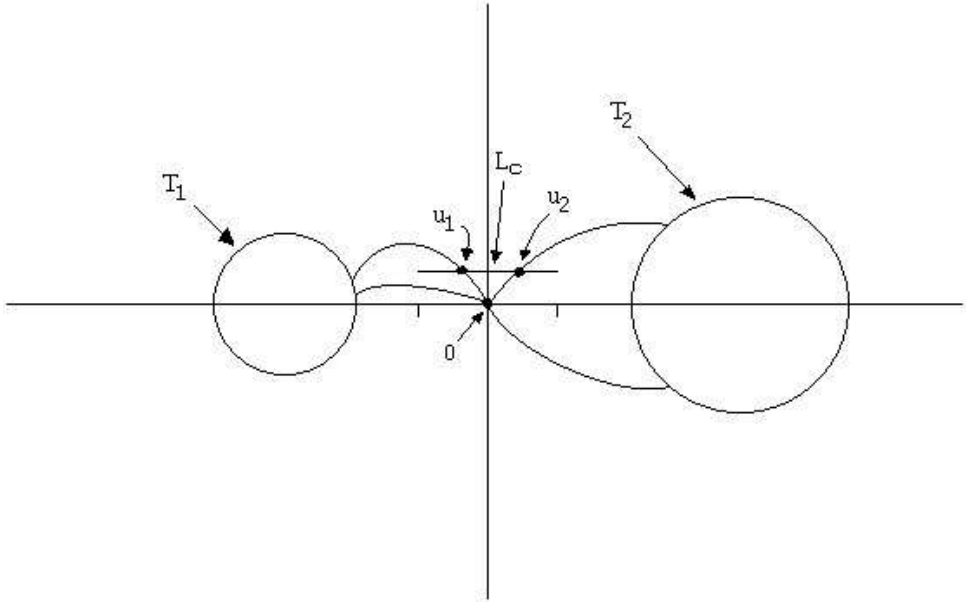}
	\caption{B\^{o}cher's Theorem}
	\label{fig:FigureForGaussTheorem}
\end{figure}

Let $z_1,\ldots,z_n\in D_1$ denote the zeros of $R$ and let $w_1,\ldots,w_n\in D_2$ denote the poles of $R$.  Then for each zero $z_i$ and pole $w_i$ of $R$, $|z_i-u_1|<|z_i-u_2|$, and $|w_i-u_1|>|w_i-u_2|$.  Therefore

\[
|R(u_1)|=\ds\dfrac{\prod_{i=1}^n|z_i-u_1|}{\prod_{i=1}^n|w_i-u_1|}<\dfrac{\prod_{i=1}^n|z_i-u_2|}{\prod_{i=1}^n|w_i-u_2|}=|R(u_2)|.
\]

This a contradiction because $u_1$ and $u_2$ are in the same level curve of $R$.  Thus we conclude that all critical points of $R$ are contained in $D_1\cup D_2$.
\end{proof}

Theorem~\ref{thm: Two level curves imply crit. level curve.} states that if any two level curves $\Lambda_1$ and $\Lambda_2$ of $f$ in $G$ are exterior to each other, then there is a critical level curve of $f$ in $G$ which "separates" the two.  That is, there is a critical level curve $\Lambda$ of $f$ in $G$ such that $\Lambda_1$ and $\Lambda_2$ are contained in different bounded faces of $\Lambda$.

\begin{theorem}\label{thm: Two level curves imply crit. level curve.}
Let each of $\Lambda_1$ and $\Lambda_2$ be level curves of $f$ contained in $G$.  Let $F_1$ denote the unbounded face of $\Lambda_1$ and $F_2$ the unbounded face of $\Lambda_2$.  If $\Lambda_1\subset{F_2}$, and $\Lambda_2\subset{F_1}$, then there is some $w\in{G}$ which lies in $F_1\cap{F_2}$, such that $f'(w)=0$ and $\Lambda_1$ and $\Lambda_2$ are contained in different bounded faces of $\Lambda_w$.
\end{theorem}

\begin{proof}
$G\cap F_1\cap F_2$ is open and connected, and it follows that we can find a path $\gamma:[0,1]\to{G}$ such that $\gamma(0)\in{\Lambda_1}$ and $\gamma(1)\in{\Lambda_2}$, and for all $r\in(0,1)$, $\gamma(r)\in{G}\cap{F_1\cap{F_2}}$.  Define $A\subset(0,1)$ be the set such that $r\in{A}$ if and only if $\Lambda_1$ is contained in one of the bounded faces of $\Lambda_{\gamma(r)}$ and $\Lambda_2$ is contained in the unbounded face of $\Lambda_{\gamma(r)}$.  Clearly if $L$ is any level curve of $f$ in $G$ such that $\Lambda_1$ and $\Lambda_2$ are in different faces of $L$, then $L$ intersects the path $\gamma$.  Since the level sets $\{z\in G:|f(z)|=\epsilon\}$ vary continuously as $\epsilon$ varies, it follows that there is a non-critical level curve of $f$ in $G$ which contains $\Lambda_1$ in its bounded face and $\Lambda_2$ in its unbounded face, and another non-critical level curve of $f$ in $G$ which contains $\Lambda_2$ in its bounded face and $\Lambda_1$ in its unbounded face.  Thus if we define $r_1\colonequals\sup(r\in(0,1):r\in{A})$, we have $r_1\in(0,1)$.  By repeated use of the continuity of the level sets of $f$, one may show that $\Lambda_1$ and $\Lambda_2$ are contained in different bounded faces of $\Lambda_{\gamma(r_1)}$.  The fact that $\Lambda_{\gamma(r_1)}$ has multiple bounded faces implies that $\Lambda_{\gamma(r_1)}$ has a vertex, and is thus a critical level curve of $f$.  The point $w$ in the statement of the theorem is any vertex in $\Lambda_{\gamma(r_1)}$.

\end{proof}

Theorem~\ref{thm: Two level curves imply crit. level curve.} gives a clear picture of the general structure of the level curves of $(f,G)$.  Since any two mutually exterior level curves of $f$ in $G$ are separated by a critical level curve of $f$ in $G$, it follows that if we remove the finitely many critical level curves of $f$ from $G$, along with each zero and pole of $f$ in $G$, then each component of the remaining set will be conformally equivalent to an annulus.  In Theorem~\ref{f Conformally Equiv. to a Power of z.} we show further that on each component of the remaining set, $f$ is conformally equivalent to the function $z\mapsto z^n$ for some $n$.  This may be thought of as an extension of the well known fact that if $f_1$ is meromorphic and has a zero (or pole) at $z_1$, then there is a neighborhood $G_1$ of $z_1$, and a conformal map $\phi_1$ from $G_1$ to a disk centered at zero such that $f_1={\phi_1}^{n_1}$ on $G_1$, where $n_1$ is the multiplicity of $z_1$ as a zero (or pole) of $f_1$.  First a definition.

\begin{definition}
Define 

\[
\mathcal{B}\colonequals\{z\in{G}:f'(z)=0\text{ or }f(z)=0\text{ or }f(z)=\infty\},
\]

and define 

\[
\Lambda_{\mathcal{B}}\colonequals\displaystyle\bigcup_{z\in\mathcal{B}}\Lambda_z.
\]

\end{definition}

\begin{theorem}\label{f Conformally Equiv. to a Power of z.}
Let $D$ be a component of $G\setminus\Lambda_{\mathcal{B}}$.  Then the following hold.

\begin{enumerate}
\item
$D$ is conformally equivalent to some annulus $A$ centered at the origin.

\item
Let $E_1$ denote the inner boundary of $D$, and let $E_2$ denote the outer boundary of $D$.  Then there is some $\epsilon_1,\epsilon_2\in[0,\infty]$ such that $\epsilon_1\neq\epsilon_2$, and $|f|\equiv\epsilon_1$ on $E_1$, and $|f|\equiv\epsilon_2$ on $E_2$.

\item\label{item:Description of phi.}  Let $i_1,i_2\in\{1,2\}$ be chosen so that $\epsilon_{i_1}<\epsilon_{i_2}$.  Then there is some $N\in\{1,2,\ldots\}$ such that $A=ann(0;{\epsilon_{i_1}}^{\frac{1}{N}},{\epsilon_{i_2}}^\frac{1}{N})$, and some conformal mapping $\phi:D\to{ann(0;{\epsilon_{i_1}}^{\frac{1}{N}},{\epsilon_{i_2}}^{\frac{1}{N}})}$ such that $f\equiv\phi^M$ on $D$, where $M=\pm{N}$.

\item
The conformal map $\phi$ described in Item~\ref{item:Description of phi.} extends continuously to $E_2$ and to all points in $E_1$ which are not zeros of $f'$.  If $z\in{E_1}$ is a zero of $f'$, and $\gamma:[0,1]\to{G}$ is a path such that $\gamma([0,1))\subset{D}$, and $\gamma(1)=z$, then $\displaystyle\lim_{r\to1^-}\phi(\gamma(r))$ exists.
\end{enumerate}
\end{theorem}

\begin{proof}
Let $D$ be some component of $G\setminus\Lambda_{\mathcal{B}}$.  Since $D$ does not contain a zero or pole of $f$, the maximum modulus theorem implies that $D^c$ must have at least one bounded component.  Suppose by way of contradiction that $D^c$ has two distinct bounded components.  The boundary of each of these bounded components is a level curve of $f$ in $G$, thus by Theorem~\ref{thm: Two level curves imply crit. level curve.} we may conclude that $D$ contains a zero of $f'$, which is a contradiction because all zeros of $f'$ in $G$ are contained in $\Lambda_{\mathcal{B}}$.  We conclude that $D^c$ has exactly one bounded component, and therefore $D$ is conformally equivalent to an annulus (see for example~\cite{C2}).

Let $E_1$ denote the interior boundary of $D$ and let $E_2$ denote the exterior boundary of $D$.  Each component of the boundary of $D$ is contained in a level curve of $f$ or in $\partial{G}$.  Therefore we may define $\epsilon_1\in[0,\infty]$ to be the value of $|f|$ on $E_1$ and $\epsilon_2\in[0,\infty]$ to be the value of $|f|$ on $E_2$.  By the maximum modulus theorem, since $D$ does not contain a zero or pole of $f$ and $D\subset{G}$, we may conclude that $\epsilon_1\neq\epsilon_2$.  (Assume throughout that $\epsilon_1<\epsilon_2$, otherwise make the appropriate minor changes.)  Similar reasoning implies that there are no two distinct level curves of $f$ in $D$ on which $|f|$ takes the same value.

For any $z\in D$, $\Lambda_z$ is a closed path in $D$, and by the maximum modulus principle the bounded face of $\Lambda_z$ must contain either a zero or pole of $f$, so $\Lambda_z$ must wind around the bounded component of $D^c$.  On the other hand, since $\Lambda_z$ is simple, $\Lambda_z$ winds only once around the bounded component of $D^c$.  Finally, by the argument principle, since $D$ does not contain any zero or pole of $f$ there is some $N\in\mathbb{Z}\setminus\{0\}$ independent of $z$ such that the change in $\arg(f)$ as $\Lambda_z$ is traversed in the positive direction is exactly $2\pi{N}$, and this is true of the boundaries $E_1$ and $E_2$ of $D$ as well.

Therefore if we set $\phi$ to be any $N^{th}$ root of $f$, $\phi$ is analytic on $D$ and injective on any given level curve of $f$ in $D$.  Since there are not two distinct level curves of $f$ in $D$ on which $|f|$ (and therefore $|\phi|$) takes the same value, it follows that $\phi$ is in fact injective on all of $D$.

Since the change in $\arg(f)$ along any level curve of $f$ in $D$ is $2\pi N$, it follows that if $\gamma$ is any closed path in $D$, the change in $\arg(f)$ along $\gamma$ is $k2\pi N$, where $k$ is the number of times $\gamma$ winds around $E_1$.

This fact, in conjunction with the Monodromy theorem, implies that $\phi$ can be extended continuously to all points in $\partial D$, except possibly to the critical points of $f$ in $E_1$.

Finally, if $z$ is any critical point of $f$ in $E_1$, and $\gamma:[0,1]\to G$ is a path with $\gamma([0,1))\subset D$ and $\gamma(1)=z$, then $\ds\lim_{r\to1^-}\phi(\gamma(r))$ exists merely because $f$ is non-zero on $\gamma([0,1))$, so the analytic continuation of the $N^{th}$ root along the path $f\circ\gamma$ exists.
\end{proof}
\section{THE POSSIBLE LEVEL CURVE CONFIGURATIONS OF A MEROMORPHIC FUNCTION}\label{sect:Possible level curve configs.}%

Our goal in this section is to classify the possible configurations of the critical level curves of a generalized finite Blaschke ratio.  We begin by defining an equivalence relation on the set of generalized finite Blaschke ratios.

\begin{definition}
If $(f_1,G_1)$ and $(f_2,G_2)$ are generalized finite Blaschke ratios, and there is some conformal map $\phi:G_1\to{G_2}$ such that $f_1\equiv{f_2}\circ\phi$ on $G_1$, then we say that $(f_1,G_1)$ and $(f_2,G_2)$ are conformally equivalent, and we write $(f_1,G_1)\sim(f_2,G_2)$.
\end{definition}

It is easy to see that $\sim$ is an equivalence relation on the collection of all generalized finite Blaschke ratios, and we make the following definition.

\begin{definition}
Let $H'$ denote the set of all generalized finite Blaschke ratios, and define $H\colonequals{H}'/\sim$.  Let $H'_a\subset{H'}$ denote the set of all generalized finite Blaschke ratios $(f,G)$ such that $f$ is analytic on $G$, and define $H_a\colonequals{H'_a}/\sim$.  We call $(f,G)\in H'_a$ a generalized finite Blaschke product.
\end{definition}

In Section~\ref{sect:Pi respects sim.} we will show that two generalized finite Blaschke ratios are in the same member of $H$ if and only if they have the same level curve structure.  In order to rigorously define the configuration of critical level curves of $(f,G)$, in this section we will define a set $PC$ (for "Possible Level Curve Configurations") whose members will parameterize the possible level curve configurations of a generalized finite Blaschke ratio.  We begin with a definition.

\begin{definition}
A connected finite graph $\Gamma$ embedded in $\mathbb{C}$ is said to be of meromorphic level curve type if there are evenly many, and more than two, edges incident to each vertex of $\Gamma$.  If in addition each edge of $\Gamma$ is incident to the unbounded face of $\Gamma$, we say that $\Gamma$ is of analytic level curve type.
\end{definition}

In Section~\ref{sect: Preliminary results.} we showed that the any level curve of a generalized finite Blaschke ratio would have the property which defines a meromorphic level curve type graph.  If $\Lambda$ is a level curve of a generalized finite Blaschke product, the open mapping theorem and maximum modulus theorem together imply that each edge of $\Lambda$ is incident to the unbounded face of $\Lambda$.  We will use graphs of meromorphic and analytic level curve type to construct our set $PC$.  Throughout we will view these graphs as embedded in $\mathbb{C}$ because we wish to keep track of the orientation of the faces and edges of the graphs about the vertices.  That is, two finite graphs $\Gamma_1$ and $\Gamma_2$ embedded in $\mathbb{C}$ will be considered equivalent if and only if there is an orientation-preserving homeomorphism from $\mathbb{C}$ to $\mathbb{C}$ which maps $\Gamma_1$ to $\Gamma_2$.  Thus, for example, we will not consider the two graphs in Figure~\ref{fig: Non-equivalent graphs.} equivalent to each other.

\begin{figure}[H]
\centering
	\includegraphics[width=0.4\textwidth]{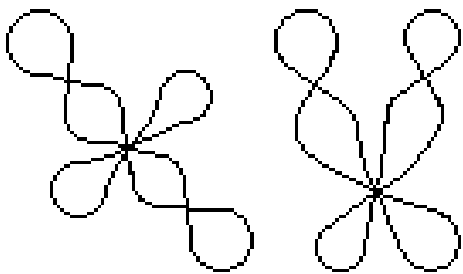}
           \caption{Non-equivalent Graphs.}
	\label{fig: Non-equivalent graphs.}

\end{figure}

We begin by defining a set $P$ whose members will represent the zeros, poles, and critical level curves of a generalized finite Blaschke ratio, along with certain auxiliary data to be defined.  As we describe the features and auxiliary data ascribed to a member of $P$ (and later of $PC$) we will parenthetically remark on the features of a level curve of a generalized finite Blaschke ratio $(f,G)$ which those features and auxiliary data are meant to represent.

There are two types of members of $P$, namely those meant to represent zeros and poles of $(f,G)$ (which we will call "single-point members of $P$") and those meant to represent critical level curves of $(f,G)$ (which we will call "graph members of $P$").  We will begin by describing the single-point members of $P$.

A single-point member $\langle w\rangle_P$ of $P$ consists of the graph consisting of a single vertex $w$ with no edges, to which we add the following pieces of auxiliary data.

\begin{itemize}
\item
We define $H(\langle{w}\rangle_P)$ to be a value in $\{0,\infty\}$ (depending on whether $\langle{w}\rangle_P$ will represent a zero of $f$ or a pole of $f$).

\item
We define $Z(\langle{w}\rangle_P)\colonequals n$ for some $n\in\mathbb{Z}\setminus\{0\}$, positive if $H(\langle w\rangle_P)=0$, negative if $H(\langle w\rangle_P)=\infty$.  (This represents the multiplicity of the point being represented as a zero or pole of $f$.)

\end{itemize}

The resulting object we denote $\langle{w}\rangle_P$.

A graph member $\langle\lambda\rangle_P$ of $P$ consists of a meromorphic level curve type graph $\lambda$, to which we add the following pieces of auxiliary data.

\begin{itemize}
\item
We define $H(\langle\lambda\rangle_P)=\epsilon$ for some value $\epsilon\in(0,\infty)$.  (This represents the value of $|f|$ on $\lambda$.)

\item
If $D$ is a bounded face of $\lambda$, we associate to $D$ an integer $z(D)\in\mathbb{Z}\setminus\{0\}$.  (This represents the number of zeros of $f$ in $D$ minus the number of poles of $f$ in $D$.)  If $D_1,D_2,\ldots,D_k$ denote the bounded faces of $\lambda$, we define $Z(\langle\lambda\rangle_P)\colonequals\displaystyle\sum_{i=1}^kz(D_i)$.  The assignment of $z(D_1),\ldots,z(D_k)$ must be done in such a way that $Z(\langle\lambda\rangle_P)\neq0$ and if $D_1$ and $D_2$ are bounded faces of $\lambda$ which share a common edge, then $z(D_1)$ and $z(D_2)$ are not both positive or both negative.  (This is the case for level curves of $f$ in view of the open mapping theorem.)

\item
We distinguish a finite number of points in $\lambda$ in such a manner that for each bounded face $D$ of $\lambda$, there are $|z(D)|$ distinct distinguished points in $\partial{D}$ (to represent the points in $\lambda$ at which $f$ takes non-negative real values).

\item
If $x\in\lambda$ is a vertex of $\lambda$, we designate a value $a(x)\in[0,2\pi)$.  (This will represent the argument of $f$ at $x$.)  We require that this assignment follows the following rules.
	\begin{itemize}
	\item
For a vertex $x$ of $\lambda$, $a(x)=0$ if and only if $x$ is a distinguished point of $\lambda$.

	\item
If $D$ is a bounded face of $\lambda$, and $z(D)>0$, and $x_1,x_2$ are distinct vertices of $\lambda$ in $\partial{D}$ such that $a(x_1)\geq{a(x_2)}$, then there is some distinguished point $z\in\partial{D}$ such that $x_1,z,x_2$ is written in increasing order as they appear in $\partial{D}$.  (This reflects the fact that if $\lambda$ is a level curve of $f$, and $D$ contains more zeros of $f$ than poles of $f$, then the argument of $f$ is increasing as $\partial{D}$ is traversed with positive orientation.)

	\item
If $D$ is a bounded face of $\lambda$, and $z(D)<0$, and $x_1,x_2$ are distinct vertices of $\lambda$ in $\partial{D}$ such that $a(x_1)\geq{a(x_2)}$, then there is some distinguished point $z\in\partial{D}$ such that $x_2,z,x_1$ is written in increasing order as they appear in $\partial{D}$.  (This reflects the fact that if $\lambda$ is a level curve of $f$, and $D$ contains more poles of $f$ than zeros of $f$, then the argument of $f$ is decreasing as $\partial{D}$ is traversed with positive orientation.)
	\end{itemize}
\end{itemize}

The resulting object, with the above auxiliary data, we denote $\langle\lambda\rangle_P$, and we define $P$ to be the set of all such $\langle\lambda\rangle_P$ and $\langle{w}\rangle_P$.  We also define $P_a\subset{P}$ by $\langle{w}\rangle_P\in{P_a}$ if and only if $Z(\langle{w}\rangle_P)>0$, and $\langle\lambda\rangle_P\in{P_a}$ if and only if $z(D)>0$ for each bounded face $D$ of $\lambda$.  (This $P_a$ is the collection of members of $P$ which may represent the zeros and critical level curves of an analytic function $f$.)

Throughout this paper, $\langle{w}\rangle_P$ will be used to refer to single-point members of $P$, $\langle\lambda\rangle_P$ will be used for graph members of $P$, and $\langle\xi\rangle_P$ will be used when we do not wish to distinguish between the two types of members of $P$.

Each member of $PC$ consists of a collection of members of $P$ arranged in different ways with respect to each other.  There are two aspects of this.  First, which graphs are in which bounded faces of which other graphs, and second, the rotational orientation of each graph with respect to the others.

We do this recursively.  To initialize our recursive construction, a level $0$ member of $PC$ will be a single-point member $\langle w\rangle_P$ of $P$ viewed as a member of $PC$, with no additional data (now written $\langle w\rangle_{PC}$).

Let $n>0$ be given.  Choose some graph member $\langle\lambda\rangle_P$ of $P$.  We will now construct a level $n$ member $\langle\lambda\rangle_{PC}$ of $PC$ as follows.  Let $D$ be a face of $\lambda$.  We have two steps.

\begin{enumerate}
\item
We choose some level $k<n$ member $\langle\xi_D\rangle_{PC}$ of $PC$ and assign it to $D$.  This assignment must satisfy the following restrictions.

\begin{itemize}
\item
$Z(\langle\xi_D\rangle_P)=z(D)$ (this represents the fact that if $\lambda$ is a level curve of $f$ in $G$, then all zeros and poles of $f$ in $D$ are contained in the bounded faces of some critical level curve $\xi_D$ of $f$ in $D$).

\item
If $z(D)>0$, then $H(\langle\xi_D\rangle_P)<H(\langle\lambda\rangle_P)$, and if $z(D)<0$, then $H(\langle\xi_D\rangle_P)>H(\langle\lambda\rangle_P)$ (this follows for level curves of meromorphic functions in view of the maximum modulus theorem).

\item
At least one of the $\langle\xi_D\rangle_{PC}$'s is a level $n-1$ member of $PC$ (to ensure that $\langle\lambda\rangle_{PC}$ was not constructed at any earlier recursive level).
\end{itemize}

This determines recursively which graphs lie in which bounded faces of which other graphs.

\item
We choose a map $g_D$ (which we will call a "gradient map") from the distinguished points of $\lambda$ in $\partial D$ to the distinguished points in $\xi_D$.  (In the context of level curves of a meromorphic function $f$, $g_D(y)=x$ means that $y$ and $x$ are connected by a gradient line of $f$).  This map must satisfy the following restriction.

\begin{itemize}
\item
$g_D$ must preserve the orientation of the distinguished points.  That is, if $y^{(1)},\ldots,y^{(z(D))}$ are the distinguished points of $\lambda$ in $\partial D$ listed in order of their appearance about $\partial D$, then the order in which the critical points in $\xi_D$ appear around $\xi_D$ is exactly $g_D(y^{(1)}),\ldots,g_D(y^{(z(D))})$ (this represents the fact that if $\lambda$ and $\xi_D$ are level curves of a meromorphic function $f$, then the gradient lines of $f$ in $D$ cannot cross since $D$ contains no critical points of $f$).
\end{itemize}
\end{enumerate}

We let $\langle\lambda\rangle_{PC}$ denote the resulting object.  The collection of all such level $0$ $\langle w\rangle_{PC}$ and level $n>0$ $\langle\lambda\rangle_{PC}$ we call $PC$, and we call $PC$ the set of possible level curve configurations.  We define $PC_a\subset{PC}$ to be the collection of members of $PC$ which are constructed entirely using members of $P_a$.  That is, $\langle\lambda\rangle_{PC}\in{PC_a}$ if and only if $\langle\lambda\rangle_P\in{P_a}$, and each member of $PC$ which is assigned to a bounded face of $\lambda$ is in $PC_a$.

We adopt the same convention of $w$, $\lambda$ or $\xi$ for members of $PC$ as we did for members of $P$, namely that level $0$ members of $PC$ we denote by $\langle{w}\rangle_{PC}$, and level $n>0$ members of $PC$ we denote by $\langle\lambda\rangle_{PC}$.  If we do not wish to specify the level of a member of $PC$ we will denote it by $\langle\xi\rangle_{PC}$.

\begin{example}
Following is a visual example of how a member of $PC$ is constructed.  We begin in Figure~\ref{fig:MemberofPCStep1} with a member $\langle\lambda\rangle_P$ of $P$ which consists of a graph $\lambda$ (here with vertex $z$ and bounded faces $D_1$ and $D_2$) along with auxiliary data, such as $H(\langle\lambda\rangle_P)=\frac{1}{2}$, $a(z)=\frac{\pi}{4}$, $z(D_1)=4$ and $z(D_2)=2$, and the marked distinguished points in $\lambda$.

\begin{figure}[H]
\centering
	\includegraphics[width=0.75\textwidth]{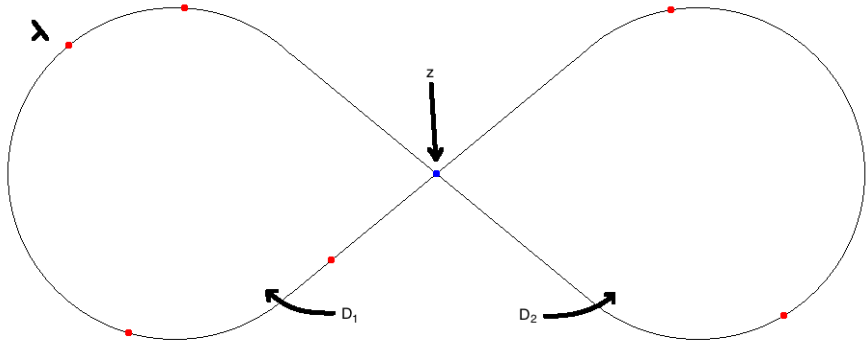}
           \caption{Member of $P$}
	\label{fig:MemberofPCStep1}

\end{figure}

Then, in Figure~\ref{fig:MemberofPCStep2}, we assign a member of $PC$ to each face of $\lambda$.  (Dashed lines represent the action of the gradient maps.)

\begin{figure}[H]
\centering
	\includegraphics[width=0.75\textwidth]{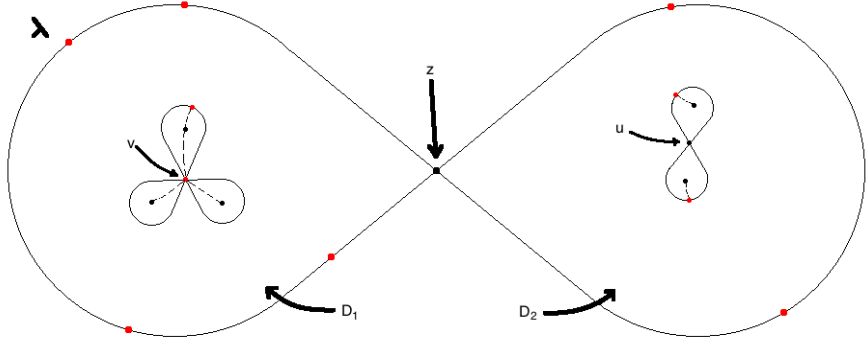}
           \caption{Assignment of members of $PC$ to the $D_1$ and $D_2$}
	\label{fig:MemberofPCStep2}

\end{figure}

Finally, in Figure~\ref{fig:MemberofPCStep3}, we designate a gradient map from the distinguished points in $\lambda$ to the distinguished points in the assigned members of $PC$.

\begin{figure}[H]
\centering
	\includegraphics[width=0.75\textwidth]{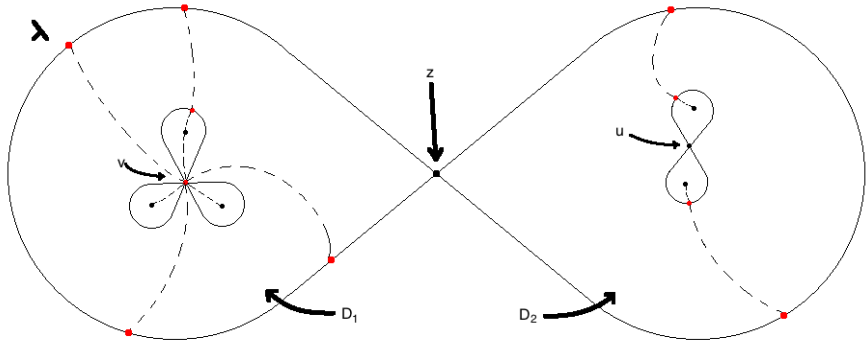}
           \caption{Designation of Gradient Maps}
	\label{fig:MemberofPCStep3}

\end{figure}

The resulting object we denote $\langle\lambda\rangle_{PC}$.

\end{example}

We now define $\Pi:H'\to{PC}$ by defining, for $(f,G)\in H'$, $\Pi(f,G)$ to be the member of $PC$ which may be constructed from the critical level curves, zeros, and poles of $f$ in $G$.  In the next Section~\ref{sect:Pi respects sim.}, we will show that $\Pi$ takes the same value on conformally equivalent members of $H'$, and therefore we may view $\Pi$ as acting on the set of equivalence classes $H$.  We will show conversely that if $\Pi$ takes the same value on two members of $H'$, then they are conformally equivalent, which implies that $\Pi$ acting on $H$ is injective.  It is easy to show that $\Pi(H_a)\subset{PC_a}$.  The major goal of Sections~\ref{sect: Pi is surjective: The generic case.}~and~\ref{sect: Pi is surjective: The general case.} is to show that $\Pi:H_a\to{PC_a}$ is also surjective.
\section{$\Pi$ RESPECTS CONFORMAL EQUIVALENCE}\label{sect:Pi respects sim.}%

Our goal in this section is to show that conformal equivalence of two generalized finite Blaschke ratios may be determined entirely by the data which is preserved by the map $\Pi$.  That is, we have the following theorem.

\begin{theorem}\label{thm:Conformal equivalence iff Pi equivalence.}
For two generalized finite Blaschke ratios $(f_1,G_1)$ and $(f_2,G_2)$, $(f_1,G_1)\sim(f_2,G_2)$ if and only if $\Pi(f_1,G_1)=\Pi(f_2,G_2)$.
\end{theorem}

\begin{proof}
The forward implication (if $(f_1,G_1)\sim(f_2,G_2)$ then $\Pi(f_1,G_1)=\Pi(f_2,G_2)$) follows straightforwardly from the definition of $\sim$ and the definition of $\Pi$.  Let $\phi:G_1\to G_2$ be a conformal map such that $f_1$ factors as $f_2$ composed with $\phi$.  Then if $\lambda$ is a level curve of $f_1$ in $G_1$, $\phi(\lambda)$ is a level curve of $f_2$ in $G_2$.  If $w$ is a zero or pole or critical point of $f_1$ in $G_1$, then $\phi(w)$ is a zero or pole or critical point respectively of $f_2$ in $G_2$ with the same multiplicity, and $\phi$ carries gradient lines of $f_1$ to gradient lines of $f_2$.  It follows immediately that the construction of $\Pi(f_1,G_1)$ proceeds in exactly the same manner as the construction of $\Pi(f_2,G_2)$, so we proceed to the more difficult backward implication.

Assume that $\Pi(f_1,G_1)=\Pi(f_2,G_2)$, and let $\langle\lambda\rangle_{PC}$ denote this common member of $PC$.  For $i\in\{1,2\}$, define $\mathcal{B}_i\colonequals\{z\in{G_i}:{f_i}'(z)=0\text{ or }f_i(z)=0\text{ or }f_i(z)=\infty\}$.  Define $\Lambda_{\mathcal{B}_i}\subset{G_i}$ by $\Lambda_{\mathcal{B}_i}\colonequals\displaystyle\bigcup_{z\in\mathcal{B}_i}\Lambda_z$.  Let $\langle\xi\rangle_P$ be some member of $P$ used in the construction of $\langle\lambda\rangle_{PC}$.  For $i\in\{1,2\}$, let $\xi_i$ denote the level curve of $f_i$ which gives rise to $\xi$.  Since $\xi_1$ and $\xi_2$ are the same when viewed as members of $P$, there is an orientation preserving homeomorphism $\phi:\xi_1\to\xi_2$.  Furthermore, if $E_1$ is some edge in $\xi_1$, and $E_2$ is the corresponding edge in $\xi_2$, then $E_1$ and $E_2$ contain the same number of distinguished points.  Therefore the change in $\arg(f_1)$ along $E_1$ is the same as the change in $\arg(f_2)$ along $E_2$.  By "reparameterizing" $\phi$ (ie changing the homeomorphism from $E_1$ to $E_2$), we may assume that $f_1\equiv f_2\circ\phi$ on $E_1$, and thus on all of $\xi_1$, and thus on all of $\Lambda_{\mathcal{B}_1}$.

We now wish to extend $\phi$ to $G_1\setminus\Lambda_{\mathcal{B}_1}$ while maintaining the property $f_1\equiv f_2\circ\phi$.  Let $F_1$ be a component of $G_1\setminus\Lambda_{\mathcal{B}_1}$ and let $F_2$ denote the corresponding component of $G_2\setminus\Lambda_{\mathcal{B}_2}$ (that is, the component of $G_2\setminus\Lambda_{\mathcal{B}_2}$ whose boundary is $\phi(\partial F_1)$).  By Theorem~\ref{f Conformally Equiv. to a Power of z.}, $F_1$ and $F_2$ are homeomorphic to annuli.  Let $L_{i,1}$ and $L_{e,1}$ denote the interior and exterior components of $\partial F_1$ respectively (and thus $L_{i,2}$ and $L_{e,2}$ are the interior and exterior components of $\partial F_2$ respectively).
  Since $\Pi(f_1,G_1)=\Pi(f_2,G_2)$, it follows that the number of zeros $f_1$ minus the number of poles of $f_1$ contained in the bounded component of ${F_1}^c$ is equal to the number of zeros of $f_2$ minus the number of poles of $f_2$ in the bounded component of ${F_2}^c$.  Let $N$ denote this common number.  Let $\epsilon_i,\epsilon_e$ denote the magnitude that $|f_1|$ takes on the interior and exterior components of $\partial F_1$ respectively (and thus, by the definition of $\Pi$, $|f_2|$ equals $\epsilon_i$ and $\epsilon_e$ on the interior and exterior components of $\partial F_2$ respectively).  By inspecting the proof of Theorem~\ref{f Conformally Equiv. to a Power of z.}, we see that the functions $\phi_1\colonequals{f_1}^{1/N}$ and $\phi_2\colonequals{f_2}^{1/N}$ conformally map $F_1$ and $F_2$ respectively to the annulus $ann(0;{\epsilon_i}^{1/N},{\epsilon_e}^{1/N})$ (for any choice of the $1/N^\text{th}$ root).  We define $\phi:F_1\to F_2$ by $\phi\colonequals{\phi_2}^{-1}\circ\phi_1$.  We now need to specify which $1/N^{th}$ roots will be used in the definition of $\phi_1$ and $\phi_2$ in order to ensure that $\phi$ defined as above on $F_1$ extends continuously to our map $\phi$ which is already defined on $\partial F_1$.  To this end, select some portion of a gradient line in $F_1$, one of whose endpoints is in $L_{i,1}$ and the other of whose endpoints is in $L_{e,1}$, and on which $\arg(f_1)=0$.  Let $\gamma_1$ denote this path, and let $\gamma_2$ denote the corresponding portion of gradient line in $F_2$.  We make the choice of $1/N^{\text{th}}$ roots so that we have $\arg(\phi_j)=0$ on $\gamma_j$ for $j=1,2$.

These normalizations on $\phi_1$ and $\phi_2$ imply that if $\gamma:[0,1]\to cl(F_1)$ is a path such that $\gamma([0,1))\subset F_1$ and $\gamma(1)\in\partial F_1$, then if we compose this path with $\phi$, we obtain $\ds\lim_{r\to1^-}\phi(\gamma(r))=\phi(\gamma(1))$ (where $\phi(\gamma(1))$ was defined earlier, when we defined the action of $\phi$ on $\Lambda_{\mathcal{B}_1}$).  That is, $\phi$ extends continuously to $\partial F_1$.

Since $\phi_1$ and $\phi_2$ are conformal maps with the same range, $\phi\colonequals{\phi_2}^{-1}\circ\phi_1$ is clearly a conformal map from $F_1$ to $F_2$.  Moreover, $f_1={\phi_1}^N$ on $F_1$ and $f_2={\phi_2}^N$ on $F_2$, thus for $z\in F_1$,

\[
f_2(\phi(z))=\left[\phi_2\left({\phi_2}^{-1}(\phi_1(z))\right)\right]^N=f_1(z).
\]

Extending $\phi$ as described above to each component of $G_1\setminus\Lambda_{\mathcal{B}_1}$, we conclude that $\phi$ is a continuous bijection from $G_1$ to $G_2$, analytic on $G_1\setminus\Lambda_{\mathcal{B}_1}$, and $f_1\equiv f_2\circ\phi$ on $G_1$.  If $z\in\Lambda_{\mathcal{B}_1}\setminus\mathcal{B}_1$, $\Lambda_{\mathcal{B}_1}$ is locally smooth at $z$, so by the Schwartz reflection principle $\phi$ is analytic at $z$.  Finally, $\mathcal{B}_1$ consists of finitely many isolated points, so for $z\in\mathcal{B}_1$, $\phi$ is analytic in a punctured neighborhood of $z$, continuous at $z$, thus analytic at $z$.  Thus $\phi$ is the desired conformal map.
\end{proof}

As mentioned at the end of the last section, the forward implication of Theorem~\ref{thm:Conformal equivalence iff Pi equivalence.} (that $(f_1,G_1)\sim(f_2,G_2)\Rightarrow\Pi(f_1,G_1)=\Pi(f_2,G_2)$) gives us that $\Pi$ acting on the equivalence classes of generalized finite Blaschke ratios ($\Pi:H\to PC$) is well defined.  The backwards implication of Theorem~\ref{thm:Conformal equivalence iff Pi equivalence.} (that $\Pi(f_1,G_1)=\Pi(f_2,G_2)\Rightarrow(f_1,G_1)\sim(f_2,G_2)$) gives us that $\Pi:H\to{PC}$ is injective.
\section{$\Pi:H_a\to PC_a$ IS A BIJECTION: THE GENERIC CASE}\label{sect: Pi is surjective: The generic case.}%

It is easy to see from the definition of $\Pi$ that $\Pi(H_a)\subset{PC_a}$ (one need only use the maximum modulus theorem).  We now begin to prove the following theorem.

\begin{theorem-mannum}{\ref{thm: Pi:H_a to PC_a is bijective.}}
$\Pi:H_a\to PC_a$ is a bijection.
\end{theorem-mannum}

Since $\Pi:H\to PC$ is injective (as shown in the last section), certainly $\Pi:H_a\to PC_a$ is injective as well.  It remains to show that $\Pi(H_a)=PC_a$.  In the course of our proof we will in fact show a stronger result, that we may consider only the action of $\Pi$ on the members of $H_a$ which have representatives $(f,G)$ such that $f$ is a polynomial, and $\Pi$ maps these members of $H_a$ surjectively onto $PC_a$.  This will also immediately imply Corollary~\ref{cor: Polynomial equivalence.}, that every equivalence class of generalized finite Blaschke products contains a polynomial function.  First several definitions.

\begin{definition}
For $G\subset\mathbb{C}$ an open simply connected set, and $f:G\to\mathbb{C}$ analytic on $G$, define $G_f\colonequals\{z\in{G}:|f(z)|<1\}$.
\end{definition}

\begin{definition}
Let ${H_p}'$ be the set of all generalized finite Blaschke ratios $(f,G)$ such that $(f,G)\sim(p|_{G_p},G_p)$ for some $p\in\mathbb{C}[z]$.  Henceforth we will write $(p,G_p)$ for $(p|_{G_p},G_p)$.  We also define $H_p\colonequals{H_p}'/\sim$.
\end{definition} 

Since $H_p\subset{H_a}$, if we can show that $\Pi(H_p)=PC_a$, then we are done.  That is, we wish to show that for any $\langle\Lambda\rangle_{PC}\in{PC}$, there is a polynomial $p\in\mathbb{C}[z]$ such that $\Pi(p,G_p)=\langle\Lambda\rangle_{PC}$.  Our method will be to partition $H_p$ by critical values.  That is, a partition set will be the collection of members of $H_p$ which have a given list of critical values.  We then define a notion of critical values for members of $PC_a$, and partition $PC_a$ by these critical values.  We then show that for any finite list of critical values, $\{v_1,\ldots,v_{n-1}\}$, $\Pi$ maps the partition set of $H_p$ corresponding to this list of critical values onto the partition set of $PC_a$ corresponding to the same list of critical values.  Having shown this, we will conclude that $\Pi(H_p)=PC_a$, and thus $\Pi(H_a)=PC_a$.

We begin by building up some notation for dealing with the critical values we will be working with.

\begin{definition}
For $n$ a positive integer, define $V_n\subset\mathbb{C}^n$ by $V_n\colonequals\{v=(v^{(1)},\ldots,v^{(n)})\in\mathbb{C}^n:0\leq|v^{(1)}|\leq\cdots\leq|v^{(n)}|<1\}$.  Define also $U_n\subset{V_n}$ by $V_n\colonequals\{v=(v^{(1)},\ldots,v^{(n)})\in\mathbb{C}^n:0<|v^{(1)}|<\cdots<|v^{(n)}|<1\}$.  Finally, define $V\colonequals\ds\bigcup_{n=1}^{\infty}V_n$ and $U\colonequals\ds\bigcup_{n=1}^{\infty}U_n$.
\end{definition}

We now define the partition of $H_p$ which we will use.

\begin{definition}
For $n\geq2$ an integer, and $v\in{V_{n-1}}$, let ${H_{p,v}}'$ denote the subset of members $(f,G)\in{H_p}'$ such that the critical values of $f$ are exactly $v^{(1)},v^{(2)},\ldots,v^{(n-1)}$.  Furthermore, define $H_{p,v}\colonequals{H_{p,v}}'/\sim$.  Let $|H_{p,v}|$ denote the number of elements of $H_{p,v}$.
\end{definition}

We will now work out a notion of critical values for a member of $PC_a$.  In essence, the critical values of a $\langle\xi\rangle_{PC}\in PC_a$ are the critical values of any member of $\Pi^{-1}(\langle\xi\rangle_{PC})$.  Since we do not yet know that this set is non-empty, we will have to define the critical values of $\langle\xi\rangle_{PC}$ directly from $\langle\xi\rangle_{PC}$.  We begin with some definitions having to do with graphs.

\begin{definition}
For a meromorphic level curve type graph $\Lambda$ embedded in $\mathbb{C}$, and $w$ a vertex of $\Lambda$, we let $m(w)$ denote the number of edges of $\Lambda$ incident to $z$.  Furthermore, we say that $w$ is a vertex of $\Lambda$ with multiplicity $\frac{m(w)}{2}-1$.
\end{definition}

Recall that if $\Lambda$ is a meromorphic level curve type graph, and $w$ is a vertex of $\Lambda$, then $m(w)$ is even, and greater than $2$.  Note also that if $(f,G)$ is a generalized finite Blaschke ratio, and $w\in{G}$ is a zero of $f'$ with multiplicity $k$, then $f$ is $(k+1)$-to-$1$ in a neighborhood of $w$.  Therefore there are $2(k+1)$ edges of $\Lambda_w$ which are incident to $w$.  Thus the multiplicity of $w$ as a vertex of $\Lambda_w$ is exactly $\frac{2(k+1)}{2}-1=k$.

\begin{definition}
Let $\langle\Lambda\rangle_{PC}\in PC_a$ be given.  If $\langle{w}\rangle_{PC}$ is one of the level $0$ members of $PC_a$ used to form $\langle\Lambda\rangle_{PC}$, then we say that $0$ is a critical value of $\langle\Lambda\rangle_{PC}$ with multiplicity $Z(\langle{w}\rangle_P)-1$.  Suppose that $\langle\lambda\rangle_P$ is a member of $P$ used to build $\langle\Lambda\rangle_{PC}$.  If $w$ is a vertex of $\lambda$, we say that $H(\langle\lambda\rangle_P)e^{ia(w)}$ is a critical value of $\langle\Lambda\rangle_{PC}$ of multiplicity equal to the multiplicity of $w$ as a vertex of $\lambda$.
\end{definition}

With this notion of critical values of a member of $PC_a$ built up, we may now partition $PC_a$ as follows.

\begin{definition}
For $v=(v^{(1)},\ldots,v^{(n)})\in{V_{n}}$, define $PC_{a,v}$ to be the collection of members of $PC_a$ whose critical values listed according to multiplicity are $v^{(1)},\ldots,v^{(n)}$.  Let $|PC_{a,v}|$ denote the number of elements of $PC_{a,v}$.
\end{definition}

From the definition of critical values of a member of $PC_a$, it should be clear that $\Pi(H_{p,v})\subset{PC_{a,v}}$.  To show that $\Pi:H_p\to PC_a$ is surjective, we show that $\Pi:H_{p,v}\to PC_{a,v}$ is surjective for each $v\in{V}$.  In Subsection~\ref{subsect: Proof of partial surjectivity.}, we prove this result for each $v$ in the dense subset $U$ of $V$.

\subsection{PARTIAL SURJECTIVITY LEMMA}\label{subsect: Proof of partial surjectivity.}

\begin{lemma}\label{lemma: Partial surjectivity.}
For any $v\in U$, $\Pi:H_{p,v}\to PC_{a,v}$ is surjective.
\end{lemma}

\begin{proof}

Fix some positive integer $n\geq2$ and some $v_0=({v_0}^{(1)},\ldots,{v_0}^{(n-1)})\in U_{n-1}$.  Since $\Pi$ is injective, it suffices to show that $|H_{a,v_0}|\geq|PC_{a,v_0}|$.  We will in fact show that $|H_{p,v_0}|=|PC_{a,v_0}|$, which immediately implies the desired result.

In a paper by Beardon, Carne, and Ng, \cite{BCN} it was shown that for $v\in V_{n-1}$, if $n=2$ then $|H_{p,v}|=1$, and if $n\geq3$ then $H_{p,v}$ contains exactly $n^{n-3}$ elements according to multiplicity, where multiplicity arises through a use of Bezout's theorem.  As an easy corollary to what was shown in that paper, one may prove that for $v\in{U_{n-1}}$, if $n=2$ then $|H_{p,v}|=1$, and if $n\geq3$, then $H_{p,v}$ contains exactly $n^{n-3}$ {\em distinct} members (that is, multiplicity plays no role).

We now wish to show that $PC_{a,v_0}$ has exactly $1$ member if $n=2$, and $n^{n-3}$ distinct members if $n\geq3$.  Individual arguments must be made for all values of $n\leq5$.  For the sake of brevity we will make only the general argument covering all values of $n\geq6$, however the arguments covering the small values of $n$ only require slight alterations of the general counting argument.

\begin{case}\label{case: n geq 6.}
$n\geq 6$.
\end{case}

Since $0<|{v_0}^{(1)}|<\cdots<|{v_0}^{(n-1)}|$, if $\langle\Lambda\rangle_{PC}\in{PC_{a,v_0}}$, and $\langle\lambda\rangle_P$ is a member of $P$ used in the construction of $\langle\Lambda\rangle_{PC}$, then $\lambda$ contains only a single vertex counting multiplicity.  (Since if there were two vertices in $\lambda$, each would give rise to a critical value, and these critical values would have the same modulus.)  If $w$ is the vertex of $\lambda$, then $H(\langle\lambda\rangle_P)e^{ia(w)}$ is a critical value of $\langle\Lambda\rangle_{PC}$ with multiplicity $1$, so by definition of multiplicity, the number of edges of $\lambda$ which meet at $w$ is $2*(1+1)=4$.  There is only one analytic level curve type graph which has a single vertex at which exactly $4$ edges meet, namely the "figure eight" graph.  (Recall that an edge is counted twice if both ends meet at the vertex.)  Thus each graph used in the construction of $\langle\Lambda\rangle_{PC}$ is this figure eight graph.

We will use an induction argument to count the number of members of $PC_{a,v_0}$.  In order to do this, it will be helpful to have an ordering on the members of $P$ used to construct a given member of $PC_{a}$, which we define now.

\begin{definition}
Fix some $\langle\Lambda\rangle_{PC}\in{PC_a}$, and let $\langle\xi_1\rangle_{PC},\ldots,\langle\xi_n\rangle_{PC}$ with $n\geq2$ be the members of $PC_a$ which are used in constructing $\langle\Lambda\rangle_{PC}$.  Then we say $\xi_i\prec\Lambda$ with respect to $\langle\Lambda\rangle_{PC}$ for each $i\in\{1,2,\ldots,n\}$.  Furthermore, if some $\langle\xi_i\rangle_{PC}$ has been associated to some bounded face of some $\xi_j$ while constructing $\langle\Lambda\rangle_{PC}$ for some $i,j$, then we say $\xi_i\prec\xi_j$ with respect to $\langle\Lambda\rangle_{PC}$ (this "with respect to $\langle\Lambda\rangle_{PC}$" will usually be suppressed when the member $\langle\Lambda\rangle_{PC}$ in question is clear).  We extend this to be a transitive relation.  That is, if $\xi_{i_1}\prec\xi_{i_2}\prec\cdots\prec\xi_{i_k}$ for some $2\leq{k}\leq{n}$, then we say $\xi_{i_1}\prec\xi_{i_k}$.
\end{definition}

We will also apply the $\prec$-ordering to the bounded faces of the graphs used to construct the members of $PC_a$ as follows.

\begin{definition}
Let $\langle\Lambda\rangle_{PC}$ be a member of $PC_a$, and let $D$ denote one of the bounded faces of $\Lambda$.  If $\langle\xi\rangle_{PC}$ is the member of $PC_a$ associated to $D$ in the construction of $\langle\Lambda\rangle_{PC}$, then we say $\xi\prec{D}$.  We extend this as follows.  If $\langle\xi_1\rangle_{PC},\ldots,\langle\xi_k\rangle_{PC}$ were used in the construction of $\langle\Lambda\rangle_{PC}$, and $\xi_1\prec\cdots\prec\xi_k\prec{D}$, then we say $\xi_1\prec{D}$.
\end{definition}

\begin{note}
Let $\langle\Lambda\rangle_{PC}$ be any member of $PC_a$, let $\langle\lambda\rangle_{PC}$ be any member of $PC_a$ used in the construction of $\langle\Lambda\rangle_{PC}$, and let $D$ be any bounded face of $\lambda$.  An easy induction argument gives that the number of level zero members $\langle{w}\rangle_{PC}$ of $PC$ such that $w\prec{D}$ is exactly $z(D)$ (where these single point members of $P$ are counted according to multiplicity), and that the number of critical values of $\langle\Lambda\rangle_{PC}$ which come from members of $PC_a$ associated to $D$ is exactly $z(D)-1$.
\end{note}

\begin{definition}
Let $\langle\Lambda\rangle_{PC}$ be any member of $PC_{a,v_0}$. As noted before, since $v_0\in U_{n-1}$, $\Lambda$ must be the "figure eight" graph.  Let $D_1$ denote the bounded face of $\Lambda$ from which fewer critical values come.  Let $D_2$ denote the other one.  That is, the naming is done so that $z(D_1)\leq{z(D_2)}$.  (If both bounded faces of $\Lambda$ give rise to the same number of critical values then this naming is arbitrary.)
\end{definition}

$\langle\Lambda\rangle_{PC}$ has $n-1$ total critical values, one of which comes from the vertex of $\Lambda$, so $n-2$ of them come from the two regions $D_1$ and $D_2$.  This together with the fact that $z(D_1)\leq{z(D_2)}$ immediately gives that the possible values for $z(D_1)-1$ to take are exactly $\{k\in\mathbb{Z}:0\leq{k}\leq\frac{n-2}{2}\}$.

Our general way of counting the number of elements in $PC_{a,v_0}$ when $n\geq6$ will be to partition $PC_{a,v_0}$ by the value $z(D_1)-1$ takes.  For a given value of $z(D_1)-1$, we find out how many ways the $z(D_1)-1$ different critical values of $\langle\Lambda\rangle_{PC}$ which come from vertices of graphs in $D_1$ may be chosen from the $n-2$ critical values available to come from $D_1$ and $D_2$ (which is of course $\binom{n-2}{z(D_1)-1}$).  For that choice of critical values coming from $D_1$, we count the number of members of $PC_a$ which may be associated to $D_1$ and the number which may be associated to $D_2$ (a natural induction step).  We then count the number of choices of the gradient maps $g_{D_1}$ and $g_{D_2}$ (which are $z(D_1)$ and $z(D_2)$ respectively, except in the case where $z(D_1)-1=1$ as we will see).  We then multiply these numbers to find the number of members of $PC_{a,v_0}$ with the given value of $z(D_1)-1$.  Finally, we take the sum over all possible choices of $z(D_1)-1$.

Assume first that $n$ is odd.  Then $z(D_1)-1$ can take any value in the set

\[
\left\{0,1,\ldots,\frac{(n-2)-1}{2}=\frac{n-3}{2}\right\}.
\]

We begin by calculating the number of members of $PC_{a,v_0}$ for which $z(D_1)-1=0$.  Here we use the fact that there is exactly one member of $PC_a$ which has no critical values, namely the level $0$ member $\langle w\rangle_{PC}$ such that $Z(\langle w\rangle_P)=1$.  Since $Z(D_1)=1$, there is exactly one choice of gradient map (ie the choice which maps the single distinguished point in $\partial D_1$ to $w$).  By the induction step, there is $(n-2+1)^{(n-2+1)-3}$ different members of $PC_a$ which have the requisite $n-2$ critical values, and may thus be assigned to $D_2$, and there are $Z(D_2)=n-2+1$ choices of the gradient map $g_{D_2}$.  Thus we have that the number of members of $PC_{a,v_0}$ for which $z(D_1)-1=0$ is

\[
\binom{n-2}{0}*1*1*(n-2+1)^{(n-2+1)-3}*(n-2+1)=\binom{n-2}{0}(n-1)^{n-3}.
\]

We now calculate the number of members of $PC_{a,v_0}$ for which $z(D_1)-1=1$.  First, there is $\binom{n-2}{1}$ ways of choosing the critical value which will come from $D_1$.  By the induction step, there is exactly one member of $PC_a$ which has a given single critical value.  $Z(D_1)=2$, so at first it appears that there are two choices of $g_{D_1}$, however the member of $PC_a$ which has a given single critical value has symmetry such that if the two edges are interchanged, one obtains the same member of $PC_a$.  Since members of $PC_a$ are formed modulo orientation preserving homeomorphism, we conclude that the two choices of $g_{D_1}$ are actually equivalent.  Thus there is actually only one choice of $g_{D_1}$.  $z(D_2)=n-3$, so by the induction step, there are $(n-3+1)^{(n-3+1)-3}$ different members of $PC_a$ which have the requisite $n-3$ critical values which come from $D_2$, and may thus be assigned to $D_2$.  Finally, there are $n-2$ choices of the gradient map $g_{D_2}$.  Thus we have that the number of members of $PC_{a,v_0}$ for which $z(D_1)-1=1$ is

\[
\binom{n-2}{1}*1*1*(n-3+1)^{(n-3+1)-3}*(n-2)=\binom{n-2}{1}(n-2)^{n-4}.
\]

Using the same reasoning, if $2\leq{i}\leq\frac{n-3}{2}$, then the number of members of $PC_{a,v_0}$ for which $z(D_1)-1=i$ is

\[
\binom{n-2}{i}*(i+1)^{i+1-3}*(i+1)*(n-2-i+1)^{(n-2-i+1)-3}*(n-2-i+1).
\]

Simplifying this, we conclude that the number of members of $PC_{a,v_0}$ for which $z(D_1)-1=i$ is

\[
\binom{n-2}{i}(i+1)^{i-1}(n-i-1)^{n-i-3}.
\]

Hence we get that

\[
|PC_{a,v_0}|=\binom{n-2}{0}(n-1)^{n-3}+\binom{n-2}{1}(n-2)^{n-4}+\\
\displaystyle\sum_{i=2}^{\frac{n-3}{2}}\binom{n-2}{i}(i+1)^{i-1}(n-i-1)^{n-i-3}.
\]

However,

\[
\binom{n-2}{0}(n-1)^{n-3}=\binom{n-2}{0}(0+1)^{0-1}(n-0-1)^{n-0-3},
\]

and

\[
\binom{n-2}{1}(n-2)^{n-4}=\binom{n-2}{1}(1+1)^{1-1}(n-1-1)^{n-1-3},
\]

so we may include these terms in the sum.  That is,

\[
\left|PC_{a,v_0}\right|=\displaystyle\sum_{i=0}^{\frac{n-3}{2}}\binom{n-2}{i}(i+1)^{i-1}(n-i-1)^{n-i-3}.
\]

By performing the substitution $m=n-2$ and using a bit of arithmetic manipulation, we obtain

\[
|PC_{a,v_0}|=\displaystyle\sum_{i=0}^{m-1}\frac{1}{2}\binom{m}{i}(i+1)^{i-1}(m-i+1)^{m-i-1}.
\]

This finite series appears in the text~\cite[page 73]{Ro}, from which we may conclude that $|PC_{a,v_0}|=(m+2)^{m-1}=n^{n-3}$.  Thus we have the desired result when $n$ is odd.

If $n$ is even, an almost identical argument gives the desired result.  The only difference is that the number of critical values which come from $D_1$ (namely $z(D_1)-1$) will be one of the numbers $\left\{0,1,\ldots,\frac{n-2}{2}\right\}$.  When $z(D_1)-1=\frac{n-2}{2}$, we have that $z(D_1)-1=z(D_2)-1$, so the choice of $D_1$ is arbitrary.  That is, we are overcounting by a factor of $2$.  Therefore in the sum we use to calculate $|PC_{a,v_0}|$, we need to include a factor of $\frac{1}{2}$ for this term.  The rest of the argument is essentially the same, and we reach the same conclusion, that $|PC_{a,v_0}|=n^{n-3}$.

We now have that, in either case, $|PC_{a,v_0}|=n^{n-3}=|H_{p,v_0}|\leq|H_{a,v_0}|$, and $\Pi:H_{a,v_0}\to{PC_{a,v_0}}$ is injective, so we conclude that $\Pi:H_{a,v_0}\to{PC_{a,v_0}}$ is also surjective.

\end{proof}

\subsection{POSSIBILITIES FOR EXTENDING THE RESULT OF LEMMA~\ref{lemma: Partial surjectivity.}}\label{subsect: Continuity proof difficulties.}

In the last subsection, we proved that the map $\Pi:H_{a,v}\to PC_{a,v}$ is surjective for all $v$ in a generic subset of $V$, namely for all $v\in U$.  Certainly any point in $V$ is a limit point of members of $U$, and $H_a=\displaystyle\bigcup_{v\in V}H_{a,v}$, so if we knew that $\Pi:H_a\to PC_a$ were continuous, we might be able to use this fact to extend the result of Lemma~\ref{lemma: Partial surjectivity.} to the full strength of Theorem~\ref{thm: Pi:H_a to PC_a is bijective.}.

However to discuss continuity of $\Pi$ would require that a tractable topology be found for $PC_a$.  While a topology for $PC_a$ may be developed, it is at present very unweildy, and in fact the proof which we will give for Theorem~\ref{thm: Pi:H_a to PC_a is bijective.} in Section~\ref{sect: Pi is surjective: The general case.} is in a sense a weakened version of what would be necessary to show that $\Pi$ is continuous.

Another idea would be to adapt the counting method we employed in the proof of Lemma~\ref{lemma: Partial surjectivity.}.  For $n\geq3$, and $v\in V_{n-1}$, $H_{p,v}$ has $n^{n-3}$ elements by~\cite{BCN}, but for $v\notin U$, this calculation now includes a notion of multiplicity which occurs from a use of Bezout's Theorem.  The steps for a counting proof of the full strength of Theorem~\ref{thm: Pi:H_a to PC_a is bijective.} would be as follows.

\begin{enumerate}
\item
Develop a notion of multiplicity for a member of $PC_a$,  (This seems very doable.)

\item
Show that counting according to the multiplicity from the last item, $PC_{a,v}$ contains $n^{n-3}$ members.  (This seems harder, but still perhaps doable.)

\item\label{item: Counting argument step 3.}
Show that the multiplicity which is accorded to a member $(f,G)$ of $H_a$ in the use of Bezout's theorem from~\cite{BCN} is equal to the multiplicity assigned to $\Pi(f,G)$ as a member of $PC_a$.  (This connection seems very difficult.)
\end{enumerate}

The desired result follows directly from the above steps however, as Step~\ref{item: Counting argument step 3.} remains elusive, we proceed to a direct proof in Section~\ref{sect: Pi is surjective: The general case.}
\section{$\Pi:H_a\to PC_a$ IS A BIJECTION: THE GENERAL CASE}\label{sect: Pi is surjective: The general case.}%

Our goal in this section is to complete the proof of the following theorem.

\begin{theorem}\label{thm: Pi:H_a to PC_a is bijective.}
$\Pi:H_a\to PC_a$ is a bijection.
\end{theorem}

We do this by extending the result of Lemma~\ref{lemma: Partial surjectivity.} to all $v\in V$.  That is, we show that $\Pi:H_{a,v}\to PC_{a,v}$ is surjective for each $v\in V$.  In order to make the ideas flow more smoothly, several of the technical arguments have been made in the form of lemmas and included in the appendix.  First several definitions.

\begin{definition}
For $v\in{V_{n-1}}$, we say that $v$ is typical if $v\in{U_{n-1}}$, in which case we say $v$ has atypicallity degree $0$ (so $0<|v^{(1)}|<\cdots<|v^{(n-1)}|$).  We say that $v$ has atypicallity degree $1$ if $0=|v^{(1)}|<|v^{(2)}|<|v^{(3)}|<\cdots<|v^{(n-1)}|$.  Finally, we say that $v$ has atypicallity degree $k$ for $2\leq k\leq{n-1}$ if $0\leq|v^{(1)}|\leq\cdots\leq|v^{(k-1)}|=|v^{(k)}|<|v^{(k+1)}|<\cdots<|v^{(n-1)}|$.
\end{definition}

\begin{definition}
Define a map $\Theta:\mathbb{C}^{n-1}\to\mathbb{C}^{n-1}$ by, for $u=(u^{(1)},\ldots,u^{(n-1)})\in\mathbb{C}^{n-1}$, $\Theta(u)=(p_u(u^{(1)}),\ldots,p_u(u^{(n-1)}))$.
\end{definition}

This map $\Theta$ was studied in~\cite{BCN}, where Beardon, Carne, and Ng established several of its properties, and used these properties to prove among other things the counting result which we have already made use of, that $|H_{p,v}|$ equals $1$ for $n=2$ and $n^{n-3}$ for $n>2$.  $\Theta$ may be thought of as taking a list of prescribed critical points $u$ to a corresponding list of critical values via the normalized polynomial $p_u$.

\begin{definition}
For any $i\geq0$, let $\mathcal{J}(i)$ denote the statement "For any $\langle\lambda\rangle_{PC}\in{PC_a}$ whose vector of critical values $v\in{V_{n-1}}$ has atypicallity degree less than or equal to $i$, there is a $u\in\Theta^{-1}(v)$ such that $(p_u,G_{p_u})\in{H_p}$ and $\Pi(p_u,G_{p_u})=\langle\lambda\rangle_{PC}$."
\end{definition}

Lemma~\ref{lemma: Partial surjectivity.} may now be restated as saying that $\mathcal{J}(0)$ holds.  We will show by induction on $i$ that $\mathcal{J}(i)$ holds for all $i\geq0$.

The idea behind the proof is as follows.  Fix some $M\geq1$ and assume inductively that $\mathcal{J}(i)$ holds for $i<M$.  Fix some $N-1\geq M$, and some $v\in V_{N-1}$ with atypicallity degree $M$.  Fix some $\langle\Lambda\rangle_{PC}\in PC_{a,v}$.  We will alter $\langle\Lambda\rangle_{PC}$ slightly to form a new member $\langle\widehat{\Lambda}\rangle_{PC}$ of $PC$ whose vector of critical values $\widehat{v}$ is very close to $v$, and has atypicallity degree strictly less than $M$.  By the induction assumption there is some $\widehat{u}\in V$ such that $\Pi(p_{\widehat{u}},G_{p_{\widehat{u}}})=\langle\widehat{\Lambda}\rangle_{PC}$.

By constructing $\langle\widehat{\Lambda}\rangle_{PC}$ in such a way as to make $\widehat{v}$ close to $v$, we may find a $u\in V$ such that $(p_u,G_{p_u})$ has critical values $v$, and moreover by making $\widehat{v}$ sufficiently close to $v$, we will deduce from the properties of the map $\Theta$ that we can force $u$ to be arbitrarily close to $\widehat{u}$.

Since the function $p_u$ depends continuously on $u$, by making $u$ sufficiently close to $\widehat{u}$ we can ensure that $|p_u-p_{\widehat{u}}|$ is arbitrarily small uniformly on $G_{p_u}$.  But the level curves of $p_u$ in turn depend continuously on $u$, so by ensuring that $u$ is sufficiently close to $\widehat{u}$, it will follow that the critical level curves of $p_u$ are very close to the critical level curves of $p_{\widehat{u}}$.

Finally, if the critical level curves of $p_u$ are sufficiently close to the critical level curves of $p_{\widehat{u}}$, we will show that the critical level curves of $p_u$ form exactly the graphs which are found in $\langle\Lambda\rangle_{PC}$, and thus we will conclude that $\Pi(p_u,G_{p_u})=\langle\Lambda\rangle_{PC}$.

To make rigorous the successive "sufficiently close's" above, we will need to choose several different constants, each of which depends on the earlier ones, and each of which must fulfill several requirements.  The proof is divided into sub-sections to keep clear the immediate goals on which we are working.

\begin{proof}[Proof of Theorem~\ref{thm: Pi:H_a to PC_a is bijective.}]

We fix some $M\geq1$ and assume inductively that $\mathcal{J}(i)$ holds for $i<M$.  We fix some $N-1\geq M$, and some $v_1\in V_{N-1}$ with atypicallity degree $M$.  We fix some $\langle\Lambda\rangle_{PC}\in PC_{a,v_1}$.

We will need several definitions to determine how small $|v_1-\widehat{v_1}|$ must be to ensure that we can proceed with the steps outlined above.

\begin{definition}
For $x^{(1)},x^{(2)}\in\mathbb{C}$, define 

\[
	d_{\arg}(x^{(1)},x^{(2)})\colonequals
	\begin{cases}
	|\arg(x^{(1)})-\arg(x^{(2)})|,&\text{ if }x^{(1)}\neq0\neq{x^{(2)}}\text{ and }\arg(x^{(1)})\neq\arg(x^{(2)})\\
	2\pi,&\text{ if }x^{(1)}=0\text{ or }x^{(2)}=0\text{ or }\arg(x^{(1)})=\arg(x^{(2)})	
	\end{cases},
\]

where the choice of $\arg(x^{(1)})$ and $\arg(x^{(2)})$ in the definition above is made so as to minimize $d_{\arg}(x^{(1)},x^{(2)})$.  For $x=(x^{(1)},\ldots,x^{(m)})\in\mathbb{C}^m$ with $m\geq2$, define

\[
d_{\arg}(x)\colonequals\min(d_{\arg}(x^{(i)},x^{(j)}):1\leq{i,j}\leq{m}).
\]
\end{definition}

\begin{definition}
For $a,b\in\mathbb{R}$ with $a<b$, and for $I:a=i_0<i_1<\cdots<i_n=b$ a partition of $[a,b]$, define $|I|\colonequals\max(i_k-i_{k-1}:1\leq{k}\leq{n})$.  Let $\gamma:[a,b]\to\mathbb{C}$ be a path and let $f$ be a function analytic and non-zero on the image of $\gamma$.  Define $\Delta_{\arg}(f,\gamma,I)\colonequals\displaystyle\sum_{k=1}^n\arg(f(\gamma(i_k)))-\arg(f(\gamma(i_{k-1})))$, where the choice of the arguments in each summand is made so as to minimize the magnitude of the summand.  Since $f$ is non-zero on $\gamma$, it is not hard to show that the limit as $|I|\to0$ of $\Delta_{\arg}(f,\gamma,I)$ exists (and is finite).  Let $\Delta_{\arg}(f,\gamma)$ denote this limit.  We call $\Delta_{\arg}(f,\gamma)$ the change of $\arg(f)$ along $\gamma$.  Define $|\Delta_{\arg}|(f,\gamma,I)\colonequals\displaystyle\sum_{k=1}^n|\arg(f(\gamma_k))-\arg(f(\gamma(i_{k-1})))|$, where again the choice of arguments is made so as to minimize the summands.  Again it is not hard to show that the limit as $|I|\to0$ of $|\Delta_{\arg}|(f,\gamma,I)$ exists (although possibly infinite).  We let $|\Delta_{\arg}|(f,\gamma)$ denote this limit, and we call $|\Delta_{\arg}|(f,\gamma)$ the total variation of $\arg(f)$ along $\gamma$.
\end{definition}

\begin{definition}
For $x=(x^{(1)},\ldots,x^{(m)})\in\mathbb{C}^m$, define

\[
	\text{minmod}(x)\colonequals
	\begin{cases}
	0,&\text{ if }x^{(1)}=\cdots=x^{(m)}=0\\
	\min(|x^{(i)}|:1\leq{i}\leq{m},x^{(i)}\neq0),&\text{otherwise}	
	\end{cases}.
\]

\end{definition}

\begin{definition}
For any $m\geq2$ and $x=(x^{(1)},\ldots,x^{(m)})\in\mathbb{C}^m$, define

\[
	\text{mindiff}(x)\colonequals
	\begin{cases}
	0,&\text{ if }x^{(1)}=\cdots=x^{(m)}\\
	\min(|x^{(i)}-x^{(j)}|:x^{(i)}\neq{x^{(j)}},1\leq{i},j\leq{m}),&\text{otherwise}	
	\end{cases}.
\]

\end{definition}

\begin{definition}
Let $\gamma:[\alpha,\beta]\to\mathbb{C}$ be a path, and let $f$ be a function analytic and non-zero on $\gamma$.  We say that $\gamma$ is parameterized according to $\arg(f)$ if for each $t\in[\alpha,\beta]$, $\arg(f(\gamma(t)))=t$.
\end{definition}

We will now determine how close $\widehat{v_1}$ must be to $v_1$.  We will choose several constants along the way, culminating in a choice of $\nu_1>0$ which will be how close we need $\widehat{v_1}$ to be to $v_1$, and which will govern our construction of $\langle\widehat{\Lambda}\rangle_{PC}$.

\subsection{CHOICE OF CONSTANTS: HOW CLOSE $\widehat{v_1}$ MUST BE TO $v_1$}

\textbf{Note to the Reader}: The purpose of the items in the choice of the constants to follow will likely be completely opaque at this time.  Therefore it may be advisible to skip this section and refer back to it as needed throughout the proof.

We begin by choosing a $\delta_1>0$ small enough that the following hold.

\begin{enumerate}
\item\label{const: delta_1 item 1.}
Let $u=(u^{(1)},\ldots,u^{(N-1)})$ be any point in $\Theta^{-1}(v_1)$.  Fix some $i\in\{1,\ldots,N-1\}$ such that ${v_1}^{(i)}\neq0$, and choose any $\widehat{u}\in{B_{\delta_1}}(u)$.  Let $L$ be a line segment contained in $B_{4\delta_1}(u^{(i)})$.  Then $|\Delta_{\arg}|(p_{\widehat{u}},L)<\frac{d_{\arg}(v_1)}{4}$.  (This may be done by the finiteness of $\Theta^{-1}(v_1)$~\cite{BCN} and the compactness of both $\text{cl}(B_{\delta_1}(u))$ and $\text{cl}(B_{4\delta_1}(u^{(i)}))$ for each $u\in\Theta^{-1}(v_1)$ and for each $i\in\{1,\ldots,N-1\}$.)

\item\label{const: delta_1 item 3.}
For any $u\in\Theta^{-1}(v_1)$, let $D$ denote either all of $G_{p_u}$, or a bounded face of one of the critical level curves $\lambda$ of $p_u$ such that $D$ contains a critical point of $p_u$ whose corresponding critical value is non-zero.  Let $\lambda_D$ be the critical level curve of $p_u$ in $D$ which is maximal with respect to $D$ (ie such that each critical point of $p_u$ in $D$ is either in $\lambda_D$ or in one of the bounded faces of $\lambda_D$).  Let $m$ denote the number of distinct edges in $\lambda_D$.  Let $E^{(1)},\ldots,E^{(m)}$ be some enumeration of the edges of $\lambda_D$.  For each $i\in\{1,\ldots,m\}$, choose some point $z^{(i)}$ in $E^{(i)}$ such that $\arg(p_u(z^{(i)}))$ is greater than $\frac{d_{\arg}(v_1)}{4}$ away from each of $\{\arg({v_1}^{(1)}),\ldots,\arg({v_1}^{(N-1)})\}$.  Let $y^{(i)}$ be the point in $\partial{D}$ which is connected to $z^{(i)}$ by a section of a gradient line of $p_u$, and let $\sigma^{(i)}:[0,1]\to\mathbb{C}$ parameterize this section of gradient line which connects $z^{(i)}$ and $y^{(i)}$.  Since $\sigma^{(i)}$ is a portion of a gradient line of $p_u$, $\arg(p_u(y^{(i)}))=\arg(p_u(z^{(i)}))$, so $y^{(i)}$ is not a critical point of $p_u$.  Since there are only finitely many such choices of $u$, $\lambda$, and $D$, we may construct such a collection of paths for each such choice of $u,\lambda,D$, and choose $\delta_1$ so that for each such $u$, $\lambda$, and $D$, if $i\in\{1,\ldots,m\}$ (here $m$ depends on the choice of $u$, $\lambda$, and $D$) and $t\in[0,1]$, there is no $j\in\{1,\ldots,m\}\setminus\{i\}$ and $s\in[0,1]$ such that $\sigma^{(j)}(s)$ is within $2\delta_1$ of $\sigma^{(i)}(t)$, and no critical point of $p_u$ is within $2\delta_1$ of $\sigma^{(i)}(t)$, and there is no edge of any critical level curve of $p_u$ other than the ones containing $\sigma^{(i)}(0)$ and $\sigma^{(i)}(1)$ within $2\delta_1$ of $\sigma^{(i)}(t)$.

\item\label{const: delta_1 item 4.}
For each $u\in\Theta^{-1}(v_1)$ , no critical level curve of $p_u$ is within $2\delta_1$ of $\partial{G_{p_u}}$ and no critical level curve of $p_u$ is within $3\delta_1$ of any other critical level curve of $p_u$.  (This may be done by the finiteness of $\Theta^{-1}(v_1)$.)

\item\label{const: delta_1 item 5.}
For each $u\in\Theta^{-1}(v_1)$, and each $k\in\{1,\ldots,N-1\}$, there is no point in the punctured ball $B_{2\delta^{(1)}}(u^{(k)})\setminus\{u^{(k)}\}$ at which $p_u$ takes the value ${v_1}^{(k)}$.  (This may be done by the finiteness of $\Theta^{-1}(v_1)$.)

\item\label{const: delta_1 item 6.}
For each $u\in\Theta^{-1}(v_1)$, if $|\widehat{u}-u|<\delta_1$, then for each $k\in\{1,\ldots,N-1\}$ such that ${v_1}^{(k)}\neq0$, if $|z-u^{(k)}|<2\delta_1$, then $|p_{\widehat{u}}(z)|>\frac{\text{minmod}(v_1)}{2}$.  (This follows from the fact that $p_u$ depends continuously on $u$.)

\item\label{const: delta_1 item 7.}
$\delta_1<\frac{\text{minmod}(v_1)d_{\arg}(v_1)}{8\pi^2}$.

\item\label{const: delta_1 item 8.}
$\delta_1<\frac{\text{mindiff}(v_1)}{4}$.

\item\label{const: delta_1 item 9.}
For each $u\in\Theta^{-1}(v_1)$, if $x_1,x_2\in{G_{p_u}}$ are both in critical level curves of $p_u$, and $\arg(p_u(x_1))=\arg(p_u(x_2))=0$, then either $x_1=x_2$ or $|x_1-x_2|>2\delta_1$.  (This may be done since there are only finitely many $u\in\Theta^{-1}(v_1)$ and only finitely many such $x_1$ and $x_2$ for each such $u$.)
\end{enumerate}

We now choose $\delta_2\in(0,\delta_1)$ small enough so that each of the following holds.

\begin{enumerate}

\item\label{const: delta_2 item 2.}
For each $u\in\Theta^{-1}(v_1)$, for each $k\in\{1,\ldots,N-1\}$, for each $z\in{B_{3\delta_2}}(u^{(k)})$, we have $|p_u(u^{(k)})-p_u(z)|<\delta_1$.  (This may be done because $\Theta^{-1}(v_1)$ is finite and each $p_u$ is continuous.)

\item\label{const: delta_2 item 3.}
By Lemma~\ref{lemma:If p_{widehat{u}} has a gradient line, then p_u has a gradient line.}, we may choose $\delta_2>0$ and $\rho_1>0$ so that the following holds.  Let $u\in\Theta^{-1}(v_1)$ be given.  Let $\widehat{u}$ be any point in $B_{\rho_1}(u)$ and let $\widehat{x_1},\widehat{x_2}\in{G_{p_{\widehat{u}}}}$ be given such that $\arg(p_{\widehat{u}}(\widehat{x_1}))=\arg(p_{\widehat{u}}(\widehat{x_2}))=0$, and such that there is a path $\widehat{\sigma}:[0,1]\to G_{p_{\widehat{u}}}$ such that $\widehat{\sigma}(0)=\widehat{x_1}$ and $\widehat{\sigma}(1)=\widehat{x_2}$ and $\arg(p_{\widehat{u}}(\widehat{\sigma}(r)))=0$ for all $r\in[0,1]$.  Then if $x_1,x_2\in{G_{p_{\widehat{u}}}}$ are such that $\arg(p_u(x_1))=\arg(p_u(x_2))=0$ and $|\widehat{x_1}-x_1|<\delta_2$ and $|\widehat{x_2}-x_2|<\delta_2$, then there is a path $\sigma:[0,1]\to{G_{p_u}}$ such that $\sigma(0)=x_1$, $\sigma(1)=x_2$, and for all $r\in[0,1]$, $\arg(p_u(\sigma(r)))=0$ and $|\widehat{\sigma}(r)-\sigma(r)|<\delta_1$.  Moreover, if $|p_{\widehat{u}}|$ is strictly increasing or strictly decreasing along $\widehat{\sigma}$, then $\sigma$ may be chosen so that $|p_u|$ is strictly increasing or strictly decreasing along $\sigma$ respectively.

\item\label{const: delta_2 item 4.}
$\delta_2<\dfrac{\text{mindiff}(0,|{v_1}^{(1)}|,\ldots,|{v_1}^{(N-1)}|)}{100}$.

\end{enumerate}

In Item~\ref{const: delta_2 item 3.} above we chose a $\rho_1\in(0,\delta_2)$.  We now require that $\rho_1>0$ be chosen smaller if necessary so that the following holds.

\begin{enumerate}

\item\label{const: rho_1 item 2.}
We will use this item to refer to the restriction on $\rho_1$ described in Item~\ref{const: delta_2 item 3.} for the choice of $\delta_2$ above.

\item\label{const: rho_1 item 3.}
Let $u=(u^{(1)},\ldots,u^{(N-1)})\in\Theta^{-1}(v_1)$ be chosen.  For $\widehat{u}\in{B_{\rho_1}}(u)$ define $\widehat{v}=(\widehat{v^{(1)}},\ldots,\widehat{v^{(n-1)}})\colonequals\Theta(\widehat{u})$.  Suppose that $\arg(\widehat{v^{(k)}})=\arg(v^{(k)})$ for each $k\in\{1,\ldots,N-1\}$.  For some $k\in\{1,\ldots,N-1\}$ with $|v^{(k)}|\neq0$, let $\widehat{\lambda}$ denote the level curve of $p_{\widehat{u}}$ which contains $\widehat{u^{(k)}}$.  Then the following holds.  Let $\widehat{E}$ denote some edge of $\widehat{\lambda}$ which is incident to $\widehat{u^{(k)}}$.  Let $\alpha$ denote a choice of $\arg(p_{\widehat{u}}(u^{(k)}))$, and let $\Delta$ denote the total change in $\arg(p_{\widehat{u}})$ along $\widehat{E}$ beginning at $\widehat{u^{(k)}}$.  Let $\widehat{\gamma}:[\alpha,\alpha+\Delta]\to\mathbb{C}$ be a parameterization of $\widehat{E}$ with respect to $\arg(p_{\widehat{u}})$ beginning at $\widehat{u^{(k)}}$.  Then if we let $\lambda$ denote the critical level curve of $p_u$ containing $u^{(k)}$, there is a path $\gamma:[\alpha,\alpha+\Delta]\to\lambda$ such that $\gamma(\alpha)=u^{(k)}$, and for each $r\in[\alpha,\alpha+\Delta]$, $\arg(p_u(\gamma(r)))=r$ and $|\gamma(r)-\widehat{\gamma}(r)|<\delta_2$.  (This may be done by Lemma~\ref{lemma:If p_{widehat{u}} has a level curve edge, then p_u has a level curve edge.}.)


\item\label{const: rho_1 item 5.}
For each $u\in\Theta^{-1}(v_1)$, if $|\widehat{u}-u|<\rho_1$, and $z\in{G_{p_{\widehat{u}}}}$, and $|z-z'|<3\rho_1$, then $|p_{\widehat{u}}(z')-p_u(z)|<\dfrac{\delta_2}{2}$.  (This may be done because $\Theta^{-1}(v_1)$ is finite and because $p_{\widehat{u}}$ depends continuously on $\widehat{u}$.)

\item\label{const: rho_1 item 6.}
For each $u=(u^{(1)},\ldots,u^{(N-1)})\in\Theta^{-1}(v_1)$, $\rho_1<\dfrac{\text{mindiff}(u^{(1)},\ldots,u^{(N-1)})}{2}$.  (This may be done by the finiteness of $\Theta^{-1}(v_1)$.)

\end{enumerate}

Finally, choose some $\nu_1\in(0,\rho_1)$ small enough so that the following holds.

\begin{enumerate}

\item\label{const: nu_1 item 2.}
If $\widehat{v_1}\in{V_{N-1}}$ satisfies $|v_1-\widehat{v_1}|<\nu_1$, and $\widehat{u}\in\Theta^{-1}(\widehat{v_1})$, then there is some $u\in\Theta^{-1}(v_1)$ such that $|u-\widehat{u}|<\frac{\rho_1}{4}$.  (This may be done by Lemma~\ref{lemma:If hat{v} is close to v, then any hat{u} is close to some u.}.)

\item\label{const: nu_1 item 4.}
$\nu_1<1-|{v_1}^{(N-1)}|$.
\end{enumerate}

This $\nu_1$ just found will be how close $\widehat{v_1}$ must be to $v_1$ to make the argument described above.  We now proceed to construct a critical level curve configuration $\langle\widehat{\Lambda}\rangle_{PC}\in{PC}$ with critical values $\widehat{v_1}\in{V_{N-1}}$ satisfying $|v_1-\widehat{v_1}|<\nu_1$, and such that $\widehat{v_1}$ has atypicallity degree strictly less than $M$.

\subsection{CONSTRUCTION OF $\langle\widehat{\Lambda}\rangle_{PC}$}\label{subsect: Construction of widehat{Lambda}.}

We begin with some notation.

\begin{definition}
For $\langle\xi\rangle_{PC}\in{PC}$ and $\epsilon>0$, we define $E_{\langle\xi\rangle_{PC},\epsilon}$ to be the collection of members $\langle\psi\rangle_P\in{P}$ used to construct $\langle\xi\rangle_{PC}$ such that $H(\langle\psi\rangle_P)=\epsilon$.
\end{definition}

Recall that $v_1$ has atypicallity degree $M$, so

\[|{v_1}^{(1)}|\leq\cdots\leq|{v_1}^{(M-1)}|=|{v_1}^{(M)}|<\cdots<|{v_1}^{(N-1)}|.
\]

We will construct $\langle\widehat{\Lambda}\rangle_{PC}\in{PC}$ differently depending on into which of the following three cases $v_1$ falls.

\begin{itemize}
\item
$|{v_1}^{(M)}|=0$.

\item
$|{v_1}^{(M)}|>0$ and for each $\langle\lambda\rangle_P\in{E_{\langle\Lambda\rangle_{PC},|{v_1}^{(M)}|}}$, $\langle\lambda\rangle_P$ only contains a single vertex (counting multiplicity).

\item
$|{v_1}^{(M)}|>0$ and there is some member of $E_{\langle\Lambda\rangle_{PC},|{v_1}^{(M)}|}$ which contains more than one vertex (counting multiplicity).
\end{itemize}

\begin{case}\label{case:|{v_1}^{(M)}|=0.}
$|{v_1}^{(M)}|=0$.
\end{case}

Recall that all single point members of $P$ are identical except for the value that $Z(\cdot)$ takes.  Therefore we make the following definition.

\begin{definition}
For each non-zero integer $k$, let $\langle{w_k}\rangle_P$ denote the single point member of $P$ such that $Z(\langle{w_k}\rangle_P)=k$.
\end{definition}

Since $0$ is a critical value of $\langle\Lambda\rangle_{PC}$, there is some level $0$ member $\langle{w_k}\rangle_{PC}\in{PC}$ used in the construction of $\langle\Lambda\rangle_{PC}$ such that $k\geq2$.  That is, in the construction of $\langle\Lambda\rangle_{PC}$, $\langle{w_k}\rangle_{PC}$ was associated to a face of some member of $P$.  Let $\langle\psi\rangle_P$ denote this member of $P$, and let $D$ denote the face of $\psi$ to which $\langle{w_k}\rangle_{PC}$ was associated.  We will define $\langle\widehat{\lambda}\rangle_{PC}$, another member of $PC$, to replace $\langle{w_k}\rangle_{PC}$ as we construct $\langle\widehat{\Lambda}\rangle_{PC}$, and in every other respect we will construct $\langle\widehat{\Lambda}\rangle_{PC}$ in the same manner as $\langle\Lambda\rangle_{PC}$.

Let $\widehat{\lambda}$ denote the "figure eight" graph.  Let $x$ denote the vertex of $\widehat{\lambda}$.  Define $H(\langle\widehat{\lambda}\rangle_P)\colonequals\frac{\nu_1}{2}$, and $a(x)\colonequals0$.  Let $\widehat{D^{(1)}}$ denote one of the bounded faces of $\widehat{\lambda}$, and $\widehat{D^{(2)}}$ the other.  Distinguish $x$ and distinguish $k-2$ distinct points other than $x$ in the boundary of $\widehat{D^{(1)}}$.

With this auxiliary data we have formed a member of $P$, namely $\langle\widehat{\lambda}\rangle_P$.  To $\widehat{D^{(1)}}$ we associate $\langle{w_{k-1}}\rangle_{PC}$, and define $g_{\widehat{D^{(1)}}}$ by mapping each distinguished point in $\partial\widehat{D^{(1)}}$ to $w_{k-1}$.  We associate $\langle{w_1}\rangle_{PC}$ to $\widehat{D^{(2)}}$, and define $g_{\widehat{D^{(2)}}}$ to map the single distinguished point in $\partial\widehat{D^{(2)}}$ (namely $x$) to $w_1$.  The resulting object is a member of $PC$, namely $\langle\widehat{\lambda}\rangle_{PC}$.

We wish to construct $\langle\widehat{\Lambda}\rangle_{PC}$ in exactly the same manner as $\langle\Lambda\rangle_{PC}$, except by replacing $\langle{w_k}\rangle_{PC}$ with $\langle\widehat{\lambda}\rangle_{PC}$.  Note that since $\nu_1<\text{minmod}(v_1)$ (by Item~\ref{const: delta_1 item 7.} in the choice of $\delta_1$), and $Z(\langle\widehat{\lambda}\rangle_P)=k=Z(\langle{w_k}\rangle_P)$ by the construction of $\langle\widehat{\lambda}\rangle_{PC}$, this replacement does not violate the rules for construction of members of $PC_a$.  The only thing remaining to do in the construction of $\langle\widehat{\Lambda}\rangle_{PC}$ is specify $g_D$.  Let $w^{(1)}$ be any fixed distinguished point in $\partial{D}$.  Then define $g_D(w^{(1)})\colonequals{x}$, and if $w$ is the $i^{th}$ distinguished point in $\partial{D}$ (for some $i\in\{1,\ldots,k-1\}$) in the positive direction after $w^{(1)}$, define $g_{D}(w)$ to be the $i^{th}$ distinguished point in $\partial\widehat{D^{(1)}}$ in the positive direction after $x$ (where the $(k-1)^{\text{st}}$ distinguished point in $\partial\widehat{D^{(1)}}$ after $x$ is interpreted as being $x$ itself).  Then proceeding with the construction in every other way the same as with $\langle\Lambda\rangle_{PC}$, we obtain a member of $PC$, namely $\langle\widehat{\Lambda}\rangle_{PC}$.

Note that the critical values of $\langle\widehat{\Lambda}\rangle_{PC}$ will be exactly

\[
\widehat{v_1}\colonequals(0,\ldots,0,\frac{\nu_1}{2},{v_1}^{(M+1)},\ldots,{v_1}^{(N-1)}),
\]

(with $M-1$ copies of $0$), while

\[
v_1=(0,\ldots,0,{v_1}^{(M+1)},\ldots,{v_1}^{(N-1)}),
\]

so $|v_1-\widehat{v_1}|=\frac{\nu_1}{2}<\nu_1$.  Note also that since $\frac{\nu_1}{2}<\text{minmod}(v_1)<|{v_1}^{(M+1)}|$, $\widehat{v_1}$ has atypicallity degree $M-1<M$.

\begin{case}\label{case:v_1^M>0 single zero.}
$|{v_1}^{(M)}|>0$ and for some $\langle\lambda\rangle_P\in{E_{\langle\Lambda\rangle_{PC},|{v_1}^{(M)}|}}$, $\lambda$ only contains a single vertex (counting multiplicity).
\end{case}

Let $\langle\lambda\rangle_P$ be some fixed member of $E_{\langle\Lambda\rangle_{PC},|{v_1}^{(M)}|}$ such that $\lambda$ contains only a single vertex counting multiplicity.  Let $\langle\widehat{\lambda}\rangle_{PC}$ be identical to $\langle\lambda\rangle_{PC}$, except that we define $H(\langle\widehat{\lambda}\rangle_P)\colonequals(1+\frac{\nu_1}{2})H(\langle\lambda\rangle_P)=(1+\frac{\nu_1}{2})|{v_1}^{(M)}|$.  Let $\langle\psi\rangle_P$ denote the member of $P$ such that $\langle\lambda\rangle_{PC}$ is assigned to some face $D$ of $\psi$ during the construction of $\langle\Lambda\rangle_{PC}$.  By Item~\ref{const: delta_2 item 4.} in the choice of $\delta_2$, $H(\langle\widehat{\lambda}\rangle_P)<H(\langle\psi\rangle_P)$ (since $\nu_1<\delta_2$), so we may replace $\langle\lambda\rangle_{PC}$ with $\langle\widehat{\lambda}\rangle_{PC}$ in the construction of $\langle\widehat{\Lambda}\rangle_{PC}$ without violating the rules of construction for members of $PC_a$.  In every other respect we construct $\langle\widehat{\Lambda}\rangle_{PC}$ in a manner identical to the construction of $\langle\Lambda\rangle_{PC}$.

Note that with this construction, the critical values of $\langle\widehat{\Lambda}\rangle_{PC}$ are exactly $\widehat{v_1}\colonequals({v_1}^{(1)},\ldots,{v_1}^{(M-1)},(1+\frac{\nu_1}{2}){v_1}^{(M)},{v_1}^{(M+1)},\ldots,{v_1}^{(N-1)})$, so $|v_1-\widehat{v_1}|=|\frac{\nu_1}{2}{v_1}^{(M)}|<\nu_1$.

The fact that, by Item~\ref{const: delta_2 item 4.} in the choice of $\delta_2$, $\nu_1<\text{mindiff}(0,|{v_1}^{(1)}|,\ldots,|{v_1}^{(N-1)}|)$ implies that $\widehat{v_1}$ has atypicallity degree strictly less than $M$.

\begin{case}\label{case:v_1^M>0 multiple zero.}
$|{v_1}^{(M)}|>0$ and each member of $E_{\langle\Lambda\rangle_{PC},|{v_1}^{(M)}|}$ contains more than one vertex (counting multiplicity).
\end{case}

Let $\langle\lambda\rangle_{PC}$ denote one of the members of $PC$ used in constructing $\langle\Lambda\rangle_{PC}$ such that $H(\langle\lambda\rangle_P)=|{v_1}^{(M)}|$.  (Possibly $\langle\lambda\rangle_{PC}=\langle\Lambda\rangle_{PC}$.)  By Lemma~\ref{lemma:Analytic graphs have face with one edge boundary.}, we may find some bounded face $F^{(1)}$ of $\lambda$ such that the boundary of $F^{(1)}$ consists of a single edge $E^{(1)}$ of $\lambda$.  Let $z$ denote the vertex at which $E^{(1)}$ has its endpoints.  We make the following definitions.

\begin{itemize}
\item
We define $\lambda\setminus{E^{(1)}}$ to be the analytic level curve type graph which arises from $\lambda$ when the edge $E^{(1)}$ is removed.  If $z$ has multiplicity $1$ as a vertex of $\lambda$, then in $\lambda\setminus E^{(1)}$ we just join the edges whose endpoints are at $z$, and no longer view $z$ as a vertex in $\lambda\setminus E^{(1)}$.

\item
We define $\langle\lambda\setminus{E^{(1)}}\rangle_P$ to be the member of $P$ which arises from $\lambda\setminus{E^{(1)}}$, and inherits all of its auxiliary data from $\langle\lambda\rangle_P$.

\item
We define $\langle\lambda\setminus{E^{(1)}}\rangle_{PC}$ to be the member of $PC_a$ which arises from $\lambda\setminus{E^{(1)}}$, and inherits all of its auxiliary data from $\langle\lambda\rangle_{PC}$.  For example, if $D$ is a bounded face of $\lambda$ other than $F^{(1)}$, and $\langle\xi_D\rangle_{PC}$ is the member of $PC_a$ associated to $D$ in the construction of $\langle\lambda\rangle_{PC}$, then we associate $\langle\xi_D\rangle_{PC}$ to $D$ in the construction of $\langle\lambda\setminus{E^{(1)}}\rangle_{PC}$, and we carry over $g_D$ to $\langle\lambda\setminus{E^{(1)}}\rangle_{PC}$ as well.
\end{itemize}

Let $F^{(2)},\ldots,F^{(h)}$ be an enumeration of the remaining bounded faces of $\lambda$.  Note that $Z(\langle\lambda\setminus{E^{(1)}}\rangle_P)=\displaystyle\sum_{k=2}^hz(F^{(k)})$.

Let $\widehat{\lambda}$ denote the "figure eight" graph, and let $\widehat{D^{(1)}},\widehat{D^{(2)}}$ denote its two faces.  Let $x$ denote the vertex in $\widehat{\lambda}$.  From $\widehat{\lambda}$ we will form a member of $P$, and eventually a member of $PC$ which will replace $\langle\lambda\rangle_{PC}$ in the construction of $\langle\Lambda\rangle_{PC}$.  Define $H(\langle\widehat{\lambda}\rangle_P)\colonequals(1+\frac{\nu_1}{2})H(\langle\lambda\rangle_P)$.  Define $z(\widehat{D^{(1)}})\colonequals{z(F^{(1)})}$ and $z(\widehat{D^{(2)}})\colonequals{Z(\langle\lambda\setminus{E^{(1)}}\rangle_P)}$.  Distinguish $z(D^{(i)})$ points in $\partial{D^{(i)}}$ for $i=1,2$, distinguishing $x$ if and only if $z$ is distinguished as a vertex in $\langle\lambda\rangle_P$.  Define $a(x)\colonequals{a(z)}$ where the value of $a(z)$ comes from $\langle\lambda\rangle_P$.  With this data, we have a member of $P$, namely $\langle\widehat{\lambda}\rangle_P$.

Let $\langle\xi_{F^{(1)}}\rangle_{PC}$ be the member of $PC$ which was associated to $F^{(1)}$ in the construction of $\langle\lambda\rangle_{PC}$.  Then we associate $\langle\xi_{F^{(1)}}\rangle_{PC}$ to $\widehat{D^{(1)}}$, and $\langle\lambda\setminus{E^{(1)}}\rangle_{PC}$ to $\widehat{D^{(2)}}$.  (Note that from now on we will use $\widehat{\lambda\setminus E^{(1)}}$ to refer to the graph $\lambda\setminus E^{(1)}$, as we are using it in the construction of $\langle\widehat{\Lambda}\rangle_{PC}$.)  We now wish to define $g_{\widehat{D^{(1)}}}$ and $g_{\widehat{D^{(2)}}}$.

Let $y^{(1)},\ldots,y^{(z(\widehat{D^{(1)}}))}$ be the distinguished points in $\partial{F^{(1)}}$ listed in increasing order, with $y^{(1)}=z$ if $z$ is distinguished in $\langle\lambda\rangle_P$, and $y^{(1)}$ the first distinguished point after $z$ in $\partial{F^{(1)}}$ otherwise.  Let $x^{(1)},\ldots,x^{({z(\widehat{D^{(1)}})})}$ be the distinguished points in $\partial{\widehat{D^{(1)}}}$ listed in increasing order, with $x^{(1)}=x$ if $x$ is distinguished, and $x^{(1)}$ the first distinguished point in the positive direction from $x$ in $\partial{\widehat{D^{(1)}}}$ otherwise.  Then for $1\leq{i}\leq{z}(\widehat{D^{(1)}})$, we define $g_{\widehat{D^{(1)}}}(x^{(i)})\colonequals{g_{F^{(1)}}}(y^{(i)})$.  

In order to define $g_{\widehat{D^{(2)}}}$, we define an enumeration of the distinguished points in $\widehat{\lambda\setminus E^{(1)}}$.  Let $E^{(2)}$ denote the edge in $\lambda$ which is directly after $E^{(1)}$ as $\lambda$ is traversed with positive orientation.  Define $y^{(1)}$ to be $z$ if $z$ is a distinguished point in $\widehat{\lambda\setminus E^{(1)}}$.  Otherwise define $y^{(1)}$ to be the first distinguished point after $z$ in $\widehat{\lambda\setminus E^{(1)}}$ as $\widehat{\lambda\setminus E^{(1)}}$ is traversed with a positive orientation beginning with $E^{(2)}$.  Continue traversing $\widehat{\lambda\setminus E^{(1)}}$ with a positive orientation, and let $y^{(2)},\ldots,y^{(z(\widehat{D^{(2)}}))}$ be the distinguished points after $y^{(1)}$ of $\widehat{\lambda\setminus E^{(1)}}$ as they appear as $\widehat{\lambda\setminus E^{(1)}}$ is traversed one full time beginning with $E^{(2)}$.  Note that a distinguished point will appear in this list exactly $n+1$ times if it is a vertex of $\widehat{\lambda\setminus E^{(1)}}$ with multiplicity $n$.  Now let $x^{(1)},\ldots,x^{(z(\widehat{D^{(2)}}))}$ be an enumeration of the distinguished points in $\partial{\widehat{D^{(2)}}}$ as they appear in increasing order beginning with $x$ if $x$ is a distinguished point, and beginning with the first distinguished point after $x$ otherwise.  Finally we define $g_{\widehat{D^{(2)}}}(x^{(i)})\colonequals{y^{(i)}}$ for each $i$.  With this data we have a member of $PC$, namely $\langle\widehat{\lambda}\rangle_{PC}$.

We construct $\langle\widehat{\Lambda}\rangle_{PC}$ in every respect the same as $\langle\Lambda\rangle_{PC}$, except that $\langle\lambda\rangle_{PC}$ will be replaced in this construction with $\langle\widehat{\lambda}\rangle_{PC}$.  If $\langle\lambda\rangle_{PC}=\langle\Lambda\rangle_{PC}$, then we are done, and we define $\langle\widehat{\Lambda}\rangle_{PC}\colonequals\langle\widehat{\lambda}\rangle_{PC}$.  If $\langle\lambda\rangle_{PC}\neq\langle\Lambda\rangle_{PC}$, then $\langle\lambda\rangle_{PC}$ was associated to some face $D$ of $\langle\psi\rangle_P$ a member of $P$ during the construction of $\langle\Lambda\rangle_{PC}$.  By identical reasoning as in Case~\ref{case:v_1^M>0 single zero.}, $H(\langle\widehat{\lambda}\rangle_P)<H(\langle\psi\rangle_P)$, so it is legal to assign $\langle\widehat{\lambda}\rangle_{PC}$ to $D$ as we construct $\langle\widehat{\Lambda}\rangle_{PC}$.  $\langle\widehat{\Lambda}\rangle_{PC}$ may inherit all of its data from $\langle\Lambda\rangle_{PC}$ except the gradient function $g_D$ (and of course $\langle\lambda\rangle_{PC}$, which we have exchanged for $\langle\widehat{\lambda}\rangle_{PC}$).  $g_D$ denotes the gradient map for $D$ in $\langle\Lambda\rangle_{PC}$ (which maps the distinguished points in $\partial{D}$ to the distinguished points in $\lambda$).  $\widehat{g_D}$ will denote the gradient map for $D$ in $\langle\widehat{\Lambda}\rangle_{PC}$ (which we are about to define and which will map the distinguished points in $\partial{D}$ to the distinguished points in $\widehat{\lambda}$).  To construct $\widehat{g_D}$, we have two possible cases, first that $z$ is a distinguished point in $\lambda$, and, in fact, the only distinct distinguished point in $\lambda$, and second that there are distinguished points in $\lambda$ which are distinct from $z$.

\begin{subcase}
$z$ is a distinguished point in $\lambda$, and $z$ is the only distinct distinguished point in $\lambda$.
\end{subcase}

Since $z$ is the only distinguished point in $\lambda$, all the distinguished points in $\widehat{\lambda}$ are contained in $\partial\widehat{D^{(2)}}$.  Let $x^{(1)},\ldots,x^{(z(D))}$ be an enumeration of the distinguished points in $\partial{D}$ listed in increasing order.  Let $y^{(1)},\ldots,y^{(z(D))}$ be an enumeration of the distinguished points in $\widehat{\lambda}$ listed in the order in which they appear as $\widehat{\lambda}$ is traversed, beginning and ending with $x$.  Then we define $\widehat{g_D}(x^{(i)})\colonequals{y^{(i)}}$ for each $i$.

\begin{subcase}
There are distinguished points in $\lambda$ which are distinct from $z$.
\end{subcase}

We define enumerations of the distinguished points of $\lambda$ and of $\widehat{\lambda}$ which we will then use to define $\widehat{g_D}$.

Let $y^{(1)},\ldots,y^{(Z(\langle\lambda\rangle_P))}$ be an enumeration of the distinguished points in $\lambda$ in the order in which they appear as $\lambda$ is traversed with positive orientation one full time beginning with the edge $E^{(1)}$.  $y^{(1)}=z$ if $z$ is a distinguished point in $\lambda$, and $y^{(1)}$ is the first distinguished point in $\partial F^{(1)}$ after $z$ otherwise.

Let $\widehat{y^{(1)}},\ldots,\widehat{y^{(Z(\langle\widehat{\lambda}\rangle_P))}}$ be an enumeration of the distinguished points in $\widehat{\lambda}$ in the order in which they appear as $\widehat{\lambda}$ is traversed with positive orientation one full time beginning with the boundary of $\widehat{D^{(1)}}$, with $\widehat{y^{(1)}}=x$ if $x$ is a distinguished point in $\widehat{\lambda}$, and $\widehat{y^{(1)}}$ is the first distinguished point in $\partial\widehat{D^{(1)}}$ after $x$ otherwise.  (Recall that $Z(\langle\widehat{\lambda}\rangle_P)=Z(\langle\lambda\rangle_P)=z(D)$.)

Let $z^{(1)},\ldots,z^{(z(D))}$ be any enumeration of the distinguished points in $\partial{D}$ in the order in which they appear around $\partial D$ such that $g_D(z^{(i)})=y^{(i)}$ for each $i\in\{1,\ldots,z(D)\}$.  Now for $i\in\{1,\ldots,z(D)\}$, we define $\widehat{g_D}(z^{(i)})\colonequals\widehat{y^{(i)}}$.  With this definition, we have all the data needed for a member of $PC$, namely $\langle\widehat{\Lambda}\rangle_{PC}$.

Notice, then, that by the construction of $\langle\widehat{\Lambda}\rangle_{PC}$, the critical values of $\langle\widehat{\Lambda}\rangle_{PC}$ are exactly

\[
\widehat{v_1}\colonequals\left({v_1}^{(1)},\ldots,{v_1}^{(M-1)},\left(1+\frac{\nu_1}{2}\right){v_1}^{(M)},{v_1}^{(M+1)},\ldots,{v_1}^{(N-1)}\right)\in{V_{N-1}}.
\]

Similarly as in Case~\ref{case:v_1^M>0 single zero.}, $\widehat{v_1}$ has atypicallity degree less than $M$, and

\[
|v_1-\widehat{v_1}|=\frac{\nu_1}{2}|{v_1}^{(M)}|<\frac{\nu_1}{2}<\nu_1.
\]

We will call the method of construction of $\langle\widehat{\lambda}\rangle_{PC}$ found in Case~\ref{case:v_1^M>0 multiple zero.} the "scattering method", as $\langle\widehat{\lambda}\rangle_{PC}$ is constructed by "scattering" one of the vertices of $\lambda$.  Figure~\ref{fig: Depiction of Scattering Method.} gives a visual depiction of the scattering method of construction of $\langle\widehat{\Lambda}\rangle_{PC}$.

\begin{figure}[H]
	\centering
		\includegraphics[width=.8\textwidth]{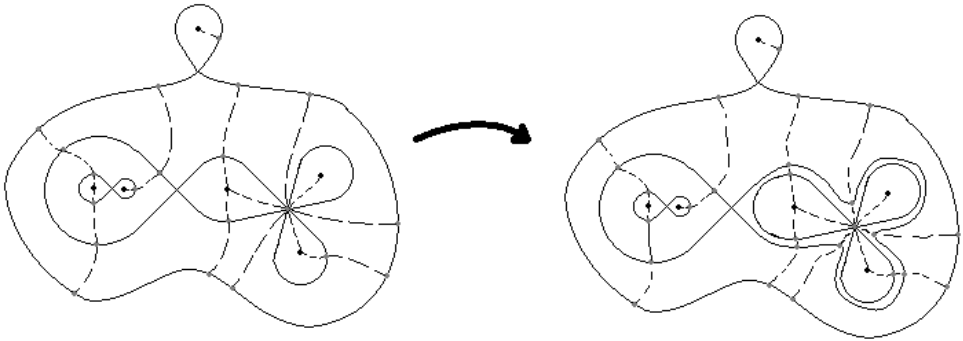}
	\caption{Depiction of Scattering Method Construction of $\langle\widehat{\Lambda}\rangle_{PC}$}
	\label{fig: Depiction of Scattering Method.}
\end{figure}

As a result of Cases~\ref{case:|{v_1}^{(M)}|=0.}, \ref{case:v_1^M>0 single zero.}, and \ref{case:v_1^M>0 multiple zero.}, we now have a member of $PC$, $\langle\widehat{\Lambda}\rangle_{PC}$, with critical values $\widehat{v_1}$ such that $|v_1-\widehat{v_1}|<\nu_1$ and the atypicallity degree of $\widehat{v_1}$ is strictly less than $M$.  Note also that by construction, for each $k\in\{1,\ldots,N-1\}$, $\arg(\widehat{{v_1}^{(k)}})=\arg({v_1}^{(k)})$, and in the special case that $M=N-1$, by Item~\ref{const: nu_1 item 4.} in the choice of $\nu_1$, $\widehat{v_1}$ is still a member of $V$.  In general if $\langle\lambda\rangle_{PC}$ is any member of $PC$ used in the construction of $\langle\Lambda\rangle_{PC}$, we let $\langle\widehat{\lambda}\rangle_{PC}$ denote the member of $PC$ which takes the place of $\langle\lambda\rangle_{PC}$ in the construction of $\langle\widehat{\Lambda}\rangle_{PC}$ (whether or not this is the same member of $PC$ as $\langle\lambda\rangle_{PC}$).

\subsection{CHOICE OF $\widehat{u_1}$ AND $u_1$, AND DEFINITION OF $\langle\widetilde{\Lambda}\rangle_{PC}$}

By the induction assumption there is some $\widehat{u_1}=(\widehat{{u_1}^{(1)}},\ldots,\widehat{{u_1}^{(N-1)}})\in\Theta^{-1}(\widehat{v_1})$ such that $\Pi(p_{\widehat{u_1}},G_{p_{\widehat{u_1}}})=\langle\widehat{\Lambda}\rangle_{PC}$.  By Item~\ref{const: nu_1 item 2.} in the choice of $\nu_1$, there is a $u_1=({u_1}^{(1)},\ldots,{u_1}^{(N-1)})\in\mathbb{C}^{N-1}$ such that $\Theta(u_1)=v_1$ and $|u_1-\widehat{u_1}|<\dfrac{\rho_1}{4}$.  Define $\langle\widetilde{\Lambda}\rangle_{PC}\colonequals\Pi(p_{u_1},G_{p_{u_1}})$.  Our goal is to show that $\langle\Lambda\rangle_{PC}=\langle\widetilde{\Lambda}\rangle_{PC}$.

We will begin this process in Sub-section~\ref{subsect: Choice of widehat{x} and widetilde{x}.} by selecting vertices of $\langle\widehat{\Lambda}\rangle_{PC}$ and $\langle\widetilde{\Lambda}\rangle_{PC}$ to correspond to any given vertex in $\langle\Lambda\rangle_{PC}$.  In Sub-sections~\ref{subsect: Choice of widehat{E}.}~\&~\ref{subsect: Choice of widetilde{D}.} we will similarly choose edges and graph faces respectively of the graphs used in the construction of $\langle\widehat{\Lambda}\rangle_{PC}$ to correspond to given edges and graph faces of $\langle\Lambda\rangle_{PC}$.

\subsection{CHOICE OF A VERTEX $\widehat{x}$ IN $\langle\widehat{\Lambda}\rangle_{PC}$ AND A VERTEX $\widetilde{x}$ IN $\langle\widetilde{\Lambda}\rangle_{PC}$ TO CORRESPOND TO EACH VERTEX $x$ IN $\langle\Lambda\rangle_{PC}$}\label{subsect: Choice of widehat{x} and widetilde{x}.}

Let $x$ denote one of the vertices of one of the members of $P$ used to construct $\langle\Lambda\rangle_{PC}$.  We wish to find a vertex $\widetilde{x}$ in one of the members of $P$ used to construct $\langle\widetilde{\Lambda}\rangle_{PC}$ (ie one of the critical points of $p_u$) which corresponds naturally to $x$.  In order to do this, we will first select a vertex $\widehat{x}$ of one of the members of $P$ used to construct $\langle\widehat{\Lambda}\rangle_{PC}$, which naturally arises from $x$.  Let $\langle\xi\rangle_P$ denote the member of $P$ which contains $x$.  Our selections first  of $\widehat{x}$ and next of $\widetilde{x}$ depends on which case was used to construct $\langle\widehat{\xi}\rangle_{PC}$.  We begin by selecting $\widehat{x}$ by cases.

\begin{case}
$\xi$ was unchanged in the construction of $\langle\widehat{\Lambda}\rangle_{PC}$.
\end{case}

In this case we may just define $\widehat{x}$ to be the vertex $x$ viewed as a vertex in $\widehat{\xi}$ rather than in $\xi$.

\begin{case}\label{case: Choice of widetilde{x}, xi was changed.}
$\xi$  was changed in the construction of $\langle\widehat{\Lambda}\rangle_{PC}$.
\end{case}

If $\langle\xi\rangle_P$ is a single point member of $P$, then as described in Case~\ref{case:|{v_1}^{(M)}|=0.}, the member of $PC$ which replaces $\langle\xi\rangle_{PC}$ is a graph member $\langle\widehat{\lambda}\rangle_{PC}$ of $PC$, where $\widehat{\lambda}$ consists of a "figure eight graph", with a single point member of $PC$ assigned to each face.  In this case we let $\widehat{x}$ denote the vertex of $\widehat{\lambda}$.

Suppose now that $\langle\xi\rangle_P$ is a graph member of $P$, and thus $\langle\widehat{\xi}\rangle_{PC}$ was constructed using the scattering method (as described in Case~\ref{case:v_1^M>0 multiple zero.}).  We will again let $F^{(1)}$ denote the face of $\xi$ which has a single edge $E^{(1)}$ of $\xi$ as its boundary, and which we split off to form $\langle\widehat{\xi}\rangle_{PC}$.  Let $x^{(1)}$ denote the vertex of $\xi$ incident to $E^{(1)}$.  If $x=x^{(1)}$, then we let $\widehat{x}$ denote the vertex in $\widehat{\xi}$.  If $x\neq x^{(1)}$, then we let $\widehat{x}$ denote the vertex $x$ viewed as a vertex in $\widehat{\xi\setminus E^{(1)}}$ (recall that $\widehat{\langle\xi\setminus E^{(1)}}\rangle_{PC}$ was assigned to one of the faces of $\widehat{\xi}$).

Now that $\widehat{x}$ has been defined using one of the sub-cases above, we wish to define a vertex $\widetilde{x}$ in one of the members of $P$ used to construct $\langle\widetilde{\Lambda}\rangle_{PC}$.  Let $i\in\{1,\ldots,N-1\}$ be one of the indices such that $\widehat{{u_1}^{(i)}}$ is the critical point of $p_{\widehat{u_1}}$ which gives rise to $\widehat{x}$ in $\Pi(p_{\widehat{u_1}},G_{p_{\widehat{u_1}}})$.  Then we wish to define $\widetilde{x}$ to be ${u_1}^{(i)}$.  To show that this definition is well defined, let  $i,j\in\{1,\ldots,N-1\}$ be given such that $\widehat{{u_1}^{(i)}}=\widehat{{u_1}^{(j)}}$.  It suffices to show that ${u_1}^{(i)}={u_1}^{(j)}$.

By choice of $\widehat{u_1}$, $|u_1-\widehat{u_1}|<\rho_1$, so since $\widehat{{u_1}^{(i)}}=\widehat{{u_1}^{(j)}}$,

\[
|{u_1}^{(i)}-{u_1}^{(j)}|\leq|{u_1}^{(i)}-\widehat{{u_1}^{(i)}}|+|\widehat{{u_1}^{(j)}}-{u_1}^{(j)}|<2\rho_1.
\]

But by Item~\ref{const: rho_1 item 6.} in the choice of $\rho_1$, $\rho_1<\dfrac{\text{mindiff}({u_1}^{(1)},\ldots,{u_1}^{(N-1)})}{2}$, and therefore ${u_1}^{(i)}={u_1}^{(j)}$.  Thus the choice of $\widetilde{x}$ is well defined.

In the case where $\langle\widehat{\Lambda}\rangle_{PC}$ is constructed using the scattering method, we would like to define an additional point in one of the graphs used to construct $\langle\widehat{\Lambda}\rangle_{PC}$, which may not appear as a $\widehat{x}$ for any vertex $x$ of a graph used to construct $\langle\Lambda\rangle_{PC}$.  Let $\langle\lambda\rangle_{PC}$ denote the member of $PC$ used in the construction of $\langle\Lambda\rangle_{PC}$ which is replaced in the scattering method construction of $\langle\widehat{\Lambda}\rangle_{PC}$.

If $x^{(1)}$ is still a vertex of $\widehat{\lambda\setminus E^{(1)}}$, we let $\widehat{x^{(0)}}$ denote this vertex, and let $\widetilde{x^{(0)}}$ be defined in a similar manner as the $\widetilde{x}$ are defined above.  Suppose now that $x^{(1)}$ is not a vertex of $\widehat{\lambda\setminus E^{(1)}}$.  First a definition.

\begin{definition}
Let $\langle\xi\rangle_P$ be a graph member of $P$, and let $E$ be an edge of $\xi$.  Let $x$ denote the initial point of $E$ and let $y$ denote the final point of $E$ with respect to positive orientation around $\xi$.  We wish to define a quantity which we will call the change in argument along $E$ (with respect to $\langle\xi\rangle_P$).  Define $r_1\colonequals{a(x)}$ and define $r_2\colonequals{a(y)}$.  If $a(y)=0$, then we instead define $r_2\colonequals2\pi$.  Then we define the change in argument along $E$ with respect to $\langle\xi\rangle_P$ to be $r_2-r_1+2\pi{n}$, where $n$ denotes the number of distinguished points in $E$ which are not end points of $E$.  Note that if $\langle\xi\rangle_P$ arises as a critical level curve of some analytic function $f$, then the change in argument along $E$ with respect to $\langle\xi\rangle_P$ is the same as the change in $\arg(f)$ along $E$.
\end{definition}

Let $E^{(2)}$ denote the next edge in $\lambda$ after $E^{(1)}$ as $\lambda$ is traversed with positive orientation.  Since $\langle\widehat{\lambda}\rangle_{PC}$ was formed using the scattering method, $\lambda$ must possess more than two bounded faces, so $E^{(2)}$ does not have both of its end points at $x^{(1)}$.  Let $K$ denote the number of distinct vertices in $\lambda$, and let $x^{(2)},\ldots,x^{(K)}$ be an enumeration of these vertices.  Let $j\in\{1,\ldots,K\}$ be the index of the end point of $E^{(2)}$ other than $x^{(1)}$.

Let $\Delta$ denote the change in argument along $E^{(2)}$ where $E^{(2)}$ is traversed from $x^{(j)}$ to $x^{(1)}$.  Let $\widehat{H}$ denote the edge of $\widehat{\lambda\setminus{E^{(1)}}}$ which contains the point $x^{(1)}$ (which is no longer a vertex of $\widehat{\lambda\setminus{E^{(1)}}}$).  Then let $\widehat{x^{(0)}}$ denote the point in $\widehat{H}$ such that the change in $\arg(p_{\widehat{u_1}})$ along the portion of $\widehat{H}$ beginning at $\widehat{x^{(j)}}$ and ending at $\widehat{x^{(0)}}$ is exactly $\Delta$.  In this case, $\widehat{x^{(0)}}$ is meant to represent the point $x^{(1)}$ in $\widehat{\lambda\setminus E^{(1)}}$, which is no longer a vertex in $\widehat{\lambda\setminus E^{(1)}}$, and we do not make any definition for $\widetilde{x^{(0)}}$.

\subsection{CHOICE OF AN EDGE $\widehat{E}$ IN $\langle\widehat{\Lambda}\rangle_{PC}$ TO CORRESPOND TO AN EDGE $E$ IN $\langle\Lambda\rangle_{PC}$}\label{subsect: Choice of widehat{E}.}

Let $\langle\lambda\rangle_P$ denote one of the members of $P$ used in the construction of $\langle\Lambda\rangle_{PC}$.  For any edge $E$ in $\lambda$, we will now select an edge, or a portion of an edge, of a graph used to construct $\langle\widehat{\Lambda}\rangle_{PC}$ (which we will call $\widehat{E}$) which corresponds naturally to $E$.

\begin{case}
$\lambda$ was unchanged in the construction of $\langle\widehat{\Lambda}\rangle_{PC}$.
\end{case}

In this case we just let $\widehat{E}$ denote the edge $E$, where $\lambda$ is now viewed as a graph used in the construction of $\langle\widehat{\Lambda}\rangle_{PC}$.

\begin{case}\label{case: Choice of widehat{E}, lambda was changed.}
$\lambda$ was changed in the construction of $\langle\widehat{\Lambda}\rangle_{PC}$.
\end{case}

In this case $\langle\widehat{\lambda}\rangle_{PC}$ was constructed using the scattering method.  We use the same notation as in Case~\ref{case: Choice of widetilde{x}, xi was changed.} (ie that of $F^{(1)}$, $E^{(1)}$, and $x^{(1)},\ldots,x^{(K)}$).  Let $\widehat{D^{(1)}}$ denote the face of $\widehat{\lambda}$ to which we assigned $\langle\xi_{\widehat{F^{(1)}}}\rangle_{PC}$.  Let $\widehat{D^{(2)}}$ denote the other face of $\widehat{\lambda}$, namely the one to which $\langle\widehat{\lambda\setminus{E^{(1)}}}\rangle_{PC}$ was assigned.  We have three possible sub-cases.

\begin{subcase}
$E=E^{(1)}$.
\end{subcase}

In this sub-case we define $\widehat{E}$ to be the edge of $\widehat{\lambda}$ which forms the boundary of $\widehat{D^{(1)}}$.

\begin{subcase}
$E\neq E^{(1)}$, and both of the end points of $E$ are still vertices in $\widehat{\lambda\setminus E^{(1)}}$.
\end{subcase}

Since both end points of $E$ are still vertices in $\widehat{\lambda\setminus E^{(1)}}$, $E$ is still an edge of $\widehat{\lambda\setminus E^{(1)}}$, so we may just define $\widehat{E}$ to be the edge $E$ viewed as an edge in $\widehat{\lambda\setminus E^{(1)}}$.

\begin{subcase}
$E\neq E^{(1)}$, and one of the end points of $E$ is not a vertex in $\widehat{\lambda\setminus E^{(1)}}$. 
\end{subcase}

The only vertex in $\lambda$ which might not be a vertex of $\widehat{\lambda\setminus E^{(1)}}$ is $x^{(1)}$.  Recall that $x^{(2)},\ldots,x^{(K)}$ is an enumeration of the other vertices in $\lambda$.  Since $\langle\widehat{\lambda}\rangle_{PC}$ was constructed using the scattering method, $\lambda$ must have more than two faces, and thus not both end points of $E$ are at $x^{(1)}$.  Let $j\in\{2,\ldots,K\}$ be the index of the end point of $E$ other than $x^{(1)}$.  Recall that $\widehat{x^{(1)}}$ is the vertex in $\widehat{\lambda}$, and $\widehat{x^{(0)}}$ is the point $x^{(1)}$ viewed still as a point in $\widehat{\lambda\setminus E^{(1)}}$, though it is no longer a vertex.  Let $\widehat{H}$ denote the edge of $\widehat{\lambda\setminus{E^{(1)}}}$ which contains the point $\widehat{x^{(0)}}$.  We define $\widehat{E}$ to be the portion of $\widehat{H}$ beginning at $\widehat{x^{(j)}}$ and ending at $\widehat{x^{(0)}}$.

Note that in each case, by the construction of $\langle\widehat{\Lambda}\rangle_{PC}$, the change in argument along $\widehat{E}$ with respect to $\langle\widehat{\lambda}\rangle_P$ is the same as the change in argument along $E$ with respect to $\langle\lambda\rangle_P$.

\subsection{CHOICE OF A GRAPH FACE $\widehat{D}$ IN $\langle\widehat{\Lambda}\rangle_{PC}$ TO CORRESPOND TO A GRAPH FACE $D$ IN $\langle\Lambda\rangle_{PC}$}\label{subsect: Choice of widetilde{D}.}

Again let $\langle\lambda\rangle_P$ be one of the members of $P$ used to construct $\langle\Lambda\rangle_{PC}$, and let $D$ be a face of $\lambda$.  We now pick a face $\widehat{D}$ of one of the graphs used in the construction of $\langle\widehat{\Lambda}\rangle_{PC}$ which corresponds to $D$.

Let $\langle\widehat{\lambda}\rangle_{PC}$ denote the member of $PC$ which replaced $\langle\lambda\rangle_{PC}$ in the construction of $\langle\widehat{\Lambda}\rangle_{PC}$.

\begin{case}
$\lambda$ was unchanged in the construction of $\langle\widehat{\Lambda}\rangle_{PC}$.
\end{case}

Then let $\widehat{D}$ denote the face $D$ of $\lambda$ viewed as a face of $\widehat{\lambda}$.

\begin{case}
$\lambda$ was changed in the construction of $\langle\widehat{\Lambda}\rangle_{PC}$.
\end{case}

In this case $\langle\widehat{\lambda}\rangle_{PC}$ was constructed using the scattering method.  We use the same notation as in Case~\ref{case: Choice of widehat{E}, lambda was changed.}.  If $D=F^{(1)}$ we define $\widehat{D}$ to be $\widehat{D^{(1)}}$.  If $D\neq{F^{(1)}}$, we define $\widehat{D}$ to be the face of $\widehat{\lambda\setminus{E^{(1)}}}$ which corresponds to $D$ in the construction of $\langle\widehat{\lambda\setminus{E^{(1)}}}\rangle_{PC}$.

We now let $\langle\lambda_D\rangle_{PC}$ and $\langle\widehat{\lambda_{\widehat{D}}}\rangle_{PC}$ be the members of $PC$ which are assigned to $D$ and $\widehat{D}$ in the construction of $\langle\Lambda\rangle_{PC}$ and $\langle\widehat{\Lambda}\rangle_{PC}$ respectively.  Notice that as a result of any of the three cases above, by the construction of $\langle\widehat{\Lambda}\rangle_{PC}$ $\partial D$ and $\partial \widehat{D}$ contain the same number of distinguished points.

\subsection{THE PLAN TO SHOW RECURSIVELY THAT $\langle\widetilde{\Lambda}\rangle_{PC}=\langle\Lambda\rangle_{PC}$}

We will show that $\langle\Lambda\rangle_{PC}=\langle\widetilde{\Lambda}\rangle_{PC}$ recursively, working "outside in", by doing the following steps.

\begin{enumerate}
\item\label{item:Step 1.}
Show that $\langle\Lambda\rangle_P=\langle\widetilde{\Lambda}\rangle_P$ by showing that the correspondence established in Sub-section~\ref{subsect: Choice of widehat{x} and widetilde{x}.} between the vertices of $\langle\Lambda\rangle_P$ and the vertices of $\langle\widetilde{\Lambda}\rangle_P$ respects the data contained in a member of $P$.  (That is, if $k\geq1$ is the number of vertices in $\Lambda$, and $u^{(1)},\ldots,u^{(k)}$ is an enumeration of these vertices, then $\widetilde{u^{(1)}},\ldots,\widetilde{u^{(k)}}$ is an enumeration of the vertices in $\widetilde{\Lambda}$ such that the following holds.  For each $i\in\{1,\ldots,k\}$, $a(u^{(i)})=a(\widetilde{u^{(i)}})$.  For each $i,j\in\{1,\ldots,k\}$, $\{u^{(i)}u^{(j)}\}$ is an edge in $\Lambda$ if and only if $\{\widetilde{u^{(i)}}\widetilde{u^{(j)}}\}$ is an edge in $\widetilde{\Lambda}$ and, moreover, if $\{u^{(i)}u^{(j)}\}$ is an edge in $\Lambda$, then $\{u^{(i)}u^{(j)}\}$ and $\{\widetilde{u^{(i)}}\widetilde{u^{(j)}}\}$ contain the same number of distinguished points (as edges in $\langle\Lambda\rangle_P$ and $\langle\widetilde{\Lambda}\rangle_P$ respectively).  Finally, if $n\geq2$ is the number of edges in $\Lambda$, and $\{u^{(i_1)}u^{(i_2)}\},\ldots,\{u^{(i_n)}u^{(i_{n+1})}\}$ is a list of edges of $\Lambda$ as they appear in order around $\Lambda$, then $\{\widetilde{u^{(i_1)}}\widetilde{u^{(i_2)}}\},\ldots,\{\widetilde{u^{(i_n)}}\widetilde{u^{(i_{n+1})}}\}$ is the order in which the edges of $\widetilde{\Lambda}$ appear around $\widetilde{\Lambda}$.  Note that this will immediately provide a well-defined correspondence between the bounded faces $D$ of $\Lambda$ and the bounded faces $\widetilde{D}$ of $\widetilde{\Lambda}$, and between the distinguished points $x$ of $\langle\Lambda\rangle_P$ and the distinguished points $\widetilde{x}$ of $\langle\widetilde{\Lambda}\rangle_P$.)  We will then conclude that $\langle\Lambda\rangle_P=\langle\widetilde{\Lambda}\rangle_P$.

\item\label{item:Step 2.}
For each bounded face $D$ of $\Lambda$, let $\langle\lambda_D\rangle_{PC}$ and $\langle\widetilde{\lambda_{\widetilde{D}}}\rangle_{PC}$ denote the members of $PC$ assigned to $D$ and $\widetilde{D}$ during the construction of $\langle\Lambda\rangle_{PC}$ and $\langle\widetilde{\Lambda}\rangle_{PC}$ respectively.  Show that $\langle\lambda_D\rangle_P=\langle\widetilde{\lambda_{\widetilde{D}}}\rangle_P$ as in the previous step, and show that the correspondence between the vertices of $\lambda_D$ and the vertices of $\widetilde{\lambda_{\widetilde{D}}}$ additionally respects the gradient maps $g_D$ and $g_{\widetilde{D}}$.  That is, if $x$ is one of the distinguished points of $\Lambda$ in $\partial{D}$, and $\widetilde{x}$ is the corresponding distinguished point of $\widetilde{\Lambda}$ in $\partial\widetilde{D}$, then show that $g_{\widetilde{D}}(\widetilde{x})$ is the distinguished point in $\widetilde{\lambda_{\widetilde{D}}}$ which corresponds to the distinguished point $g_D(x)$ in $\lambda_D$.

\item
For each bounded face $D$ of $\Lambda$, iterate Step~\ref{item:Step 2.} for each bounded face of $\Lambda_D$.  Continue recursively.

\end{enumerate}

Since $\langle\Lambda\rangle_{PC},\langle\widetilde{\Lambda}\rangle_{PC}$ are constructed using finitely many members of $P$, this process will terminate after finitely many steps.  When this process terminates, we will have shown that $\langle\Lambda\rangle_{PC}$ and $\langle\widetilde{\Lambda}\rangle_{PC}$ have all the same data, and are therefore equal.  Notice that the base case (Step~\ref{item:Step 1.} and Step~\ref{item:Step 2.} as written) is just a simpler case of the recursive step (which is Step~\ref{item:Step 1.} and Step~\ref{item:Step 2.} with any $\langle\lambda\rangle_{PC}$ used in the construction of $\langle\Lambda\rangle_{PC}$ inserted in the place of $\langle\Lambda\rangle_{PC}$).  Therefore we just do the recursive step.

Suppose that Step~\ref{item:Step 2.} has just been completed for some $\langle\lambda\rangle_P$ used to construct $\langle\Lambda\rangle_{PC}$, with corresponding $\langle\widetilde{\lambda}\rangle_P$ used to construct $\langle\widetilde{\Lambda}\rangle_{PC}$.  Let $D$ be one of the bounded faces of $\lambda$, and let $\widetilde{D}$ be the corresponding bounded face of $\widetilde{\lambda}$.  Let $\langle\lambda_D\rangle_{PC}$ and $\langle\widetilde{\lambda_{\widetilde{D}}}\rangle_{PC}$ be the members of $PC$ assigned to $D$ and $\widetilde{D}$ in the construction of $\langle\Lambda\rangle_{PC}$ and $\langle\widetilde{\Lambda}\rangle_{PC}$ respectively.  We wish to show that $\langle\lambda_D\rangle_P=\langle\widetilde{\lambda_{\widetilde{D}}}\rangle_P$, and that $g_D$ and $g_{\widetilde{D}}$ respect the correspondence between the distinguished points of $\langle\lambda\rangle_P$ and $\langle\lambda_D\rangle_P$ on the one hand, and those of $\langle\widetilde{\lambda}\rangle_P$ and $\langle\widetilde{\lambda_{\widetilde{D}}}\rangle_P$ on the other.

In the following argument, we make the assumption that $\langle\lambda_D\rangle_P$ is a graph member of $P$.  The case where $\langle\lambda_D\rangle_P$ is a single point member of $P$ requires a much simplified version of the same argument, so we omit it.  We will also assume that $\langle\widehat{\lambda_{\widehat{D}}}\rangle_{PC}$ was formed using the scattering method (and therefore $\langle\widehat{\lambda}\rangle_P$ is the same as $\langle\lambda\rangle_P$).  If $\langle\widehat{\lambda_{\widehat{D}}}\rangle_{PC}$ was not formed using the scattering method, the argument is substantially the same, though somewhat simpler, so we omit it as well.

\subsection{BUILDING THE CORRESPONDENCE BETWEEN THE VERTICES $x$ OF $\langle\Lambda\rangle_{PC}$ AND THE VERTICES $\widetilde{x}$ OF $\langle\widetilde{\Lambda}\rangle_{PC}$}

Still using the same notation as in Case~\ref{case: Choice of widehat{E}, lambda was changed.}, but now applied to $\lambda_D$, we will first show that for each $l\in\{1,\ldots,K\}$ (and for $l=0$ if $\widetilde{x^{(0)}}$ is defined), $\widetilde{x^{(l)}}$ is contained in $\widetilde{\lambda_{\widetilde{D}}}$.  In order to do this we must first show that each $\widetilde{x^{(l)}}$ is contained in $\widetilde{D}$.

\subsubsection{SHOW THAT IF $x$ IS A VERTEX IN $D$ THEN $\widetilde{x}$ IS IN $\widetilde{D}$}

Let $\widehat{\gamma}$ be a path which parameterizes $\partial\widehat{D}$ according to $\arg(p_{\widehat{u_1}})$.  Item~\ref{const: rho_1 item 3.} in the choice of $\rho_1$ guarantees that there is a corresponding path $\widetilde{\gamma}$ through a level curve of $p_{u_1}$ on which $|p_{u_1}|$ takes the same value as $|p_{\widehat{u_1}}|$ takes on $\widehat{\gamma}$, which is parameterized according to $\arg(p_{u_1})$, and such that $|\widetilde{\gamma}-\widehat{\gamma}|<\delta_2$.  We can then use the choice of $u_1$ and Item~\ref{const: delta_1 item 5.} in the choice of $\delta_1$ to show that $\widetilde{\gamma}$ parameterizes $\partial\widetilde{D}$.

Item~\ref{const: delta_1 item 4.} in the choice of $\delta_1$ and the fact that $|\widetilde{\gamma}-\widehat{\gamma}|<\delta_2$ now imply that if any critical point $\widetilde{x}$ of $p_{u_1}$ is in $\widetilde{D}$, then the corresponding critical point $\widehat{x}$ of $p_{\widehat{u_1}}$ is in $\widehat{D}$.  Due to the recursive assumption on $\langle\widetilde{\lambda}\rangle_P$ and $\langle\lambda\rangle_P$, $\widehat{D}$ contains the same number of critical points of $p_{\widehat{u_1}}$ as $\widetilde{D}$ contains of $p_{u_1}$.  Therefore $\widetilde{x}\in\widetilde{D}$ if and only if $\widehat{x}\in\widehat{D}$.  We conclude that if $x$ is a vertex of $\langle\Lambda\rangle_{PC}$ in $D$, then by construction of $\langle\widehat{\Lambda}\rangle_{PC}$, $\widehat{x}\in\widehat{D}$, and thus $\widetilde{x}\in\widetilde{D}$.

\subsubsection{SHOW THAT IF $x$ IS A VERTEX IN $\lambda_D$ THEN $\widetilde{x}$ IS A VERTEX IN $\widetilde{\lambda_{\widetilde{D}}}$}

We now wish to show that for each $l\in\{1,\ldots,K\}$ (and for $l=0$ if $\widetilde{x^{(0)}}$ is defined), $\widetilde{x^{(l)}}\in\widetilde{\lambda_{\widetilde{D}}}$.  For each such $l$, let $t_l$ be one of the indices in $\{1,\ldots,N-1\}$ such that ${u_1}^{(t_l)}=\widetilde{x^{(l)}}$.  If $l\in\{2,\ldots,K\}$ (or $l=0$ if $\widetilde{x^{(0)}}$ is defined), then $\widetilde{x^{(l)}}\neq{u_1}^{(M)}$, so

\[
|p_{u_1}(\widetilde{x^{(l)}})|=|p_{u_1}({u_1}^{(t_l)})|=|{v_1}^{(t_l)}|=|\widehat{{v_1}^{(t_l)}}|=H(\langle\widehat{\lambda_D\setminus{E^{(1)}}}\rangle_P)=H(\langle\lambda_D\rangle_P).
\]

For $l=1$, we may let $t_l=M$, so

\[
|p_{u_1}(\widetilde{x^{(l)}})|=|p_{u_1}({u_1}^{(M)})|=|{v_1}^{(M)}|=\left|\frac{1}{1+\frac{\nu_1}{2}}\widehat{{v_1}^{(t_l)}}\right|=\frac{1}{1+\frac{\nu_1}{2}}H(\langle\widehat{\lambda_{\widehat{D}}}\rangle_P)=H(\langle\lambda_D\rangle_P).
\]

As a result of Theorem~\ref{thm: Two level curves imply crit. level curve.}, if there are distinct level curves of $p_{u_1}$ in $\widetilde{D}$ on which $|p_{u_1}|=H(\langle\lambda_D\rangle_P)$, then there is some critical point $z$ in $\widetilde{D}$ such that $|p_{u_1}(z)|>H(\langle\lambda_D\rangle_P)$.  Suppose by way of contradiction that this is the case.  Then $|p_{u_1}(z)|\geq H(\langle\lambda_D\rangle_P)+\text{mindiff}(|{v_1}^{(1)}|,\ldots,|{v_1}^{(N-1)}|)$ by definition of $\text{mindiff}$.  Choose some $t\in\{1,\ldots,N-1\}$, such that ${u_1}^{(t)}=z$.  Then $\widehat{{u_1}^{(t)}}\in\widehat{D}$ and, since $t_1=M$, $t\neq{M}$, so $|\widehat{{v_1}^{(t)}}|=|{v_1}^{(t)}|\geq H(\langle\lambda_D\rangle_P)+\text{mindiff}(v_1)>H(\langle\widehat{\lambda_{\widehat{D}}}\rangle_P)$ by Item~\ref{const: delta_1 item 8.} in the choice of $\delta_1$ (since $\delta_1<\nu_1$).  Thus, by the maximum modulus theorem, $\widehat{{u_1}^{(t)}}$ is not in one of the bounded faces of $\widehat{\lambda_{\widehat{D}}}$, which is a contradiction of the definition of $\langle\widehat{\lambda_{\widehat{D}}}\rangle_{PC}$.  Thus we conclude that each point $\widetilde{x^{(l)}}$ is in the unique critical level curve of $p_{u_1}$ in $\widetilde{D}$ on which $|p_{u_1}|$ takes its largest value, namely $\widetilde{\lambda_{\widetilde{D}}}$.

\subsection{CHOICE OF AN EDGE $\widetilde{E}$ of $\langle\widetilde{\Lambda}\rangle_{PC}$ TO CORRESPOND TO A GIVEN EDGE $E$ OF $\langle\Lambda\rangle_{PC}$}\label{subsect: Choice of widetilde{E}.}

Let $L\geq2$ denote the number of edges in $\lambda_D$, and let $E^{(2)},\ldots,E^{(L)}$ be an enumeration of the edges of $\lambda_D$ other than $E^{(1)}$ in the order in which they appear around $\lambda_D$ after $E^{(1)}$ with positive orientation.  Choose some $k_0\in\{1,\ldots,L\}$.  We will find an edge $\widetilde{E^{(k_0)}}$ of $\widetilde{\lambda_{\widetilde{D}}}$ which corresponds to the edge $E^{(k_0)}$ in $\lambda_D$.

\begin{case}\label{case: find edge, x^{(1)} is not an end point of E^{(k_0)}.}
$x^{(1)}$ is not an end point of $E^{(k_0)}$.
\end{case}

Let $i,j\in\{1,\ldots,L\}$ be the indices such that $x^{(i)}$ is the initial point of $E^{(k_0)}$ and $x^{(j)}$ is the final point of $E^{(k_0)}$.  Since $x^{(1)}$ is not an end point of $E^{(k_0)}$, $\widehat{E^{(k_1)}}$ is an edge in $\widehat{\lambda_D\setminus E^{(1)}}$, having end points $\widehat{x^{(i)}}$ and $\widehat{x^{(j)}}$.  Recall that, as we are viewing $\langle\widehat{\Lambda}\rangle_{PC}$ as embedded in $\mathbb{C}$ as the critical level curves of $p_{\widehat{u_1}}$, we have $\widehat{x^{(i)}}=\widehat{{u_1}^{(t_i)}}$ and $\widehat{x^{(j)}}=\widehat{{u_1}^{(t_j)}}$.  In addition, $a(x^{(i)})=a(\widehat{x^{(i)}})$ and $a(x^{(j)})=a(\widehat{x^{(j)}})$ by the construction of $\langle\widehat{\Lambda}\rangle_{PC}$, so

\[
a(x^{(i)})=a(\widehat{x^{(i)}})=\arg(p_{\widehat{u_1}}(\widehat{{u_1}^{(t_i)}}))\text{ and }a(x^{(j)})=a(\widehat{x^{(j)}})=\arg(p_{\widehat{u_1}}(\widehat{{u_1}^{(t_j)}})).
\]

Assume that $a(x^{(i)})=0$ (otherwise make the appropriate minor changes throughout the argument).  Let $\Delta>0$ denote the change in $\arg(p_{\widehat{u_1}})$ along $\widehat{E^{(k_0)}}$, and let $\widehat{\gamma}$ be a parameterization of $\widehat{E^{(k_0)}}$ according to $\arg(p_{\widehat{u_1}})$ beginning at $\widehat{x^{(i)}}$.  By Item~\ref{const: rho_1 item 3.} in the choice of $\rho_1$, we may find a path $\widetilde{\gamma}:[0,\Delta]\to\widetilde{\lambda_{\widetilde{D}}}$ such that $\widetilde{\gamma}(0)={u_1}^{(t_i)}$, and for each $r\in[0,\Delta]$, $\arg(p_{u_1}(\widetilde{\gamma}(r)))=r$, and $|\widetilde{\gamma}(r)-\widehat{\gamma}(r)|<\delta_2$.

Now $\widehat{\gamma}(\Delta)=\widehat{{u_1}^{(t_j)}}$, so $\Delta=\arg(\widehat{{v_1}^{(t_j)}})$ mod $2\pi$.  Therefore

\[
p_{u_1}(\widetilde{\gamma}(\Delta))=|{v_1}^{(t_j)}|e^{i\arg({v_1}^{(t_j)})}={v_1}^{(t_j)}.
\]

Moreover,

\[
|\widetilde{\gamma}(\Delta)-{u_1}^{(t_j)}|\leq|\widetilde{\gamma}(\Delta)-\widehat{\gamma}(\Delta)|+|\widehat{\gamma}(\Delta)-\widehat{{u_1}^{(t_j)}}|+|\widehat{{u_1}^{(t_j)}}-{u_1}^{(t_j)}|.
\]

However $\widehat{{u_1}^{(t_j)}}=\widehat{\gamma}(\Delta)$, so we have

\[
|\widetilde{\gamma}(\Delta)-{u_1}^{(t_j)}|\leq|\widetilde{\gamma}(\Delta)-\widehat{\gamma}(\Delta)|+|\widehat{{u_1}^{(t_j)}}-{u_1}^{(t_j)}|<\delta_2+\rho_1<2\delta_2.
\]

By Item~\ref{const: delta_1 item 6.} in the choice of $\delta_1$, there is no point in $B_{2\delta_2}({u_1}^{(t_j)})\setminus\{{u_1}^{(t_j)}\}$ at which $p_{u_1}$ takes the value ${v_1}^{(t_j)}$, so we conclude that $\widetilde{\gamma}(\Delta)={u_1}^{({t_j})}$.  Therefore we have that $\widetilde{\gamma}$ is a path from ${u_1}^{(t_i)}$ to ${u_1}^{(t_j)}$ through $\widetilde{\lambda_{\widetilde{D}}}$.  

Since $\widetilde{\gamma}$ and $\widehat{\gamma}$ stay within $\delta_2$ of each other, and $\widehat{\gamma}$ does not intersect any critical point of $p_{\widehat{u_1}}$ other than at its end points, a geometric argument may be made (using, for example, Items~\ref{const: delta_1 item 1.}~\&~\ref{const: delta_1 item 7.} in the choice of $\delta_1$ and Items~\ref{const: delta_2 item 2.}~\&~\ref{const: delta_2 item 4.} in the choice of $\delta_2$) on the respective images of $p_{u_1}$ and $p_{\widehat{u_1}}$ on $\widetilde{\gamma}$ and $\widehat{\gamma}$ to the effect that $\widetilde{\gamma}$ does not contain any critical point of $p_{u_1}$ except at its end points.  That is, $\widetilde{\gamma}$ parameterizes a single edge of $\widetilde{\lambda_{\widetilde{D}}}$ from ${u_1}^{(t_i)}$ to ${u_1}^{(t_j)}$ (thus from $\widetilde{x^{(i)}}$ to $\widetilde{x^{(j)}}$).  We let $\widetilde{E^{(k_0)}}$ denote this edge.

\begin{case}
$x^{(1)}$ is an end point of $E^{(k_0)}$.
\end{case}

In this case, there are three sub-cases to consider, which correspond whether or not $E^{(k_0)}=E^{(1)}$ (the edge which was split off of $\lambda_D$ in the scattering method construction of $\langle\widehat{\lambda_{\widehat{D}}}\rangle_{PC}$) and, if not, whether or not $\widetilde{x^{(0)}}$ is defined.

\begin{subcase}\label{subcase:k_0=1.}
$k_0=1$ (recall that $E^{(1)}=\partial{F^{(1)}}$).
\end{subcase}

The definition of $\widetilde{E^{(k_0)}}$ in this sub-case is essentially the same as in Sub-case~\ref{case: find edge, x^{(1)} is not an end point of E^{(k_0)}.}, since both end points of $\widehat{E^{(k_0)}}$ are critical points of $p_{\widehat{u_1}}$ and $\widehat{E^{(k_0)}}$ is a full edge of one of the single critical level curves of $p_{\widehat{u_1}}$.

\begin{subcase}\label{subcase:k_0neq1, x^{(1)} still a vertex.}
$k_0\neq1$, but $x^{(1)}$ is still a vertex of $\lambda_D\setminus{E^{(1)}}$.
\end{subcase}

Since $k_0\neq1$, $E^{(k_0)}$ does not form $\partial{F^{(1)}}$.  As noted earlier, the fact that $\langle\widehat{\lambda_{\widehat{D}}}\rangle_{PC}$ was formed using the scattering method implies that at least one end point of $E^{(k_0)}$ is not at $x^{(1)}$.  Let $i\in\{2,\ldots,K\}$ be the index so that $x^{(i)}$ is the other end point of $E^{(k_0)}$.  In this sub-case, again both end points of $\widehat{E^{(k_0)}}$ are critical points of $p_{\widehat{u_1}}$, and $\widehat{E^{(k_0)}}$ is a full edge of one of the single critical level curves of $p_{\widehat{u_1}}$, so the method of construction of $\widetilde{E^{(k_0)}}$ is again identical to that in Sub-case~\ref{case: find edge, x^{(1)} is not an end point of E^{(k_0)}.}, however in this sub-case, $\widehat{E^{(k_0)}}$ is an edge in $\widehat{\lambda_D\setminus E^{(1)}}$, with end points $\widehat{x^{(i)}}$ and $\widehat{x^{(0)}}$.  Let $\widehat{\gamma}$ be a parameterization of this edge according to $\arg(p_{\widehat{u_1}})$.

As in Sub-case~\ref{case: find edge, x^{(1)} is not an end point of E^{(k_0)}.}, we conclude that there is an edge $\widetilde{E^{(k_0)}}$ in $\widetilde{\lambda_{\widetilde{D}}}$ with end points $\widetilde{x^{(i)}}$ and $\widetilde{x^{(0)}}$, and a parameterization $\widetilde{\gamma}$ of $\widetilde{E^{(k_0)}}$ according to $\arg(p_{u_1})$ such that for each $r$, $|\widehat{\gamma}(r)-\widetilde{\gamma}(r)|<\delta_2$.  Since the end points of $E^{(k_0)}$ are $x^{(i)}$ and $x^{(1)}$, we wish to show that $\widetilde{x^{(0)}}=\widetilde{x^{(1)}}$.

Define $\Gamma_1$ to be the straight line path from $\widetilde{x^{(1)}}$ to $\widehat{x^{(1)}}$.  Let $\Gamma_2$ denote the portion of the gradient line of $p_{\widehat{u_1}}$ which connects $\widehat{x^{(1)}}$ with $\widehat{x^{(0)}}$.  Let $\Gamma_3$ denote the straight line path from $\widehat{x^{(0)}}$ to $\widetilde{x^{(0)}}$.  Let $\Gamma$ denote the concatenation of these three paths.  Using Item~\ref{const: rho_1 item 5.} in the choice of $\rho_1$ and Item~\ref{const: delta_2 item 4.} in the choice of $\delta_2$, along with Theorem~\ref{thm: Two level curves imply crit. level curve.}, one can make a basic argument based on the image of $p_{u_1}$ on $\Gamma$ that $\Gamma$ can not intersect any critical level curve of $p_{u_1}$ other than $\widetilde{\lambda_{\widetilde{D}}}$.  Therefore $\Gamma$ may be projected along gradient lines of $p_{u_1}$ to a path through $\widetilde{\lambda_{\widetilde{D}}}$ from $\widetilde{x^{(1)}}$ to $\widetilde{x^{(0)}}$.  Since $\Gamma_1\subset B_{\rho_1}(\widetilde{x^{(1)}})$, Item~\ref{const: delta_2 item 2.} in the choice of $\delta_2$ and Item~\ref{const: delta_1 item 7.} in the choice of $\delta_1$, along with some elementary trigonometry, now imply that $\Delta_{\arg}(p_{u_1},\Gamma_1)\leq \dfrac{d_{\arg}(v_1)}{4}$.  Similarly $\Delta_{\arg}(p_u,\Gamma_3)\leq \dfrac{d_{\arg}(v_1)}{4}$.   Item~\ref{const: rho_1 item 5.} in the choice of $\rho_1$ implies that for each $z\in\Gamma_2$, $|p_{\widehat{u_1}}(z)|\geq H(\langle\lambda_D\rangle_P)$.  Since $\arg(p_{\widehat{u_1}})$ is constant on $\Gamma_2$, Item~\ref{const: delta_1 item 7.} in the choice of $\delta_1$ and Item~\ref{const: delta_2 item 4.} in the choice of $\delta_2$, along with some elementary trigonometry, now show that $\Delta_{\arg}(p_{u_1},\Gamma_2)<\dfrac{d_{\arg}(v_1)}{4}$.  Therefore $\Delta_{\arg}(p_{u_1},\Gamma)<d_{\arg}(v_1)$, and therefore $\Gamma$ (that is, the projected version of $\Gamma$ which resides in $\widetilde{\lambda_{\widetilde{D}}}$) cannot traverse $\widetilde{\lambda_{\widetilde{D}}}$ from $\widetilde{x^{(1)}}$ to a distinct vertex of $\widetilde{\lambda_{\widetilde{D}}}$, so we conclude that $\widetilde{x^{(0)}}=\widetilde{x^{(1)}}$.

\begin{subcase}
$k_0\neq1$ and $x^{(1)}$ is not a vertex of $\lambda_D\setminus{E^{(1)}}$.
\end{subcase}

Assume during the following argument that $x^{(i)}$ is the initial point of $E^{(k_0)}$ (otherwise make the appropriate minor changes, such as reversing orientations of paths, etc.).  Let $\Delta>0$ denote change in argument along $E^{(k_0)}$ from $x^{(i)}$ to $x^{(1)}$.  Assume that $a(x^{(i)})=0$ (otherwise make the appropriate minor changes).  Let $\widehat{H}$ denote the edge of $\widehat{\lambda_D\setminus{E^{(1)}}}$ which contains $\widehat{x^{(0)}}$ (the point which takes the place of $x^{(1)}$ in $\widehat{\lambda_D\setminus E^{(1)}}$).  Recall that $\widehat{E^{(k_0)}}$ is the portion of $\widehat{H}$ with end points $\widehat{x^{(i)}}$ and $\widehat{x^{(0)}}$.  Let $\widehat{\gamma}$ be a parameterization of $\widehat{E^{(k_0)}}$ with respect to $\arg(p_{\widehat{u_1}})$.

Again by Item~\ref{const: rho_1 item 3.} in the choice of $\rho_1$, there is a path $\widetilde{\gamma}:[0,\Delta]\to\widetilde{\lambda_{\widetilde{D}}}$ such that $\widetilde{\gamma}(0)={u_1}^{(t_{i})}$ and for each $r\in[0,\Delta]$, $\arg(p_{u_1}(\widetilde{\gamma}(r)))=r$ and $|\widetilde{\gamma}(r)-\widehat{\gamma}(r)|\leq\delta_1$.  The argument for this sub-case is very similar to the argument for Case~\ref{case: find edge, x^{(1)} is not an end point of E^{(k_0)}.}.  The major difference is in the method by which we show that $\widetilde{\gamma}(\Delta)=\widetilde{{u_1}^{(t_1)}}$.  Our method here is similar to that by which we showed that $\widetilde{x^{(0)}}=\widetilde{x^{(1)}}$ in Sub-case~\ref{subcase:k_0neq1, x^{(1)} still a vertex.}.

Since the image of $\widetilde{\gamma}$ is contained in $\widetilde{\lambda_{\widetilde{D}}}$, we conclude that $|p_{u_1}(\widetilde{\gamma}(\Delta))|=|p_{u_1}(\widetilde{x^{(1)}})|$.  By definition of $\widetilde{\gamma}$, $\arg(p_{u_1}(\widetilde{\gamma}(\Delta)))=\Delta$, and by definition of $\Delta$, $\arg(p_{u_1}(\widetilde{x^{(1)}}))=a(\widetilde{x^{(1)}})=\Delta$ (since we are assuming that $\arg(p_{u_1}(\widetilde{x^{(i)}}))=0$).  Therefore $p_{u_1}(\widetilde{\gamma}(\Delta))=p_{u_1}(\widetilde{x^{(1)}})$.  Let $\Gamma_1$ denote the straight line path from $\widetilde{x^{(1)}}$ to $\widehat{x^{(1)}}$.  Let $\Gamma_2$ denote the portion of the gradient line of $p_{\widehat{u_1}}$ which connects $\widehat{x^{(1)}}$ to $\widehat{x^{(0)}}$.  Let $\Gamma_3$ denote the straight line path from $\widehat{x^{(0)}}$ (which equals $\widehat{\gamma}(\Delta)$) to $\widetilde{\gamma}(\Delta)$.  Let $\Gamma$ denote the path from $\widetilde{x^{(1)}}$ to $\widetilde{\gamma}(\Delta)$ obtained by concatenating the three paths $\Gamma_1$, $\Gamma_2$, and $\Gamma_3$.

Again, using a similar argument as that used in Sub-case~\ref{subcase:k_0neq1, x^{(1)} still a vertex.} to show that $\widetilde{x^{(0)}}=\widetilde{x^{(1)}}$, it now follows that ${u_1}^{(t_1)}=\widetilde{x^{(1)}}=\widetilde{\gamma}(\Delta)$.  The rest of the argument for this sub-case is essentially the same as for Case~\ref{case: find edge, x^{(1)} is not an end point of E^{(k_0)}.}, so we omit it.

The conclusion which we draw is as before.  Namely, there is a path $\widehat{\gamma}$ which parameterizes $\widehat{E^{(k_0)}}$ according to $\arg(p_{\widehat{u_1}})$ and an edge $\widetilde{E^{(k_0)}}$ of $\widetilde{\lambda_{\widetilde{D}}}$ from $\widetilde{x^{(i)}}$ to $\widetilde{x^{(1)}}$ with a parameterization  $\widetilde{\gamma}$ which parameterizes $\widetilde{E^{(k_0)}}$ according to $\arg(p_{u_1})$ such that for each $r$, $|\widehat{\gamma}(r)-\widetilde{\gamma}(r)|<\delta_2$.  Now for each $l\in\{1,\ldots,L\}$, let $\widetilde{\gamma^{(l)}}$ denote the path which parameterizes $\widetilde{E^{(l)}}$.  Let $\widehat{\gamma^{(l)}}$ be the path which parameterizes $\widehat{E^{(l)}}$.

\subsection{SHOW THAT THE ORDER OF THE EDGES AROUND $\widetilde{\lambda_{\widetilde{D}}}$ IS THE SAME AS THE ORDER OF THE CORRESPONDING EDGES AROUND $\lambda_D$}

We now wish to show that the edges $\widetilde{E^{(1)}},\ldots,\widetilde{E^{(L)}}$ appear in the same order around $\widetilde{\lambda_{\widetilde{D}}}$ as the order in which the edges to which they correspond appear around $\lambda_D$.  Recall that by Item~\ref{const: delta_1 item 3.} in the choice of $\delta_1$, there is a choice of points $z^{(1)},\ldots,z^{(L)}\in\widetilde{\lambda_{\widetilde{D}}}$ such that for each $l\in\{1,\ldots,L\}$, the following holds.

\begin{itemize}
\item
$z^{(l)}$ is in $\widetilde{E^{(l)}}$  but is not an end point of $\widetilde{E^{(l)}}$.

\item
$\arg(p_{u_1}(z^{(l)}))$ is more than $\frac{d_{\arg}(v_1)}{4}$ away from each of $\{\arg({v_1}^{(1)}),\ldots,\arg({v_1}^{(N-1)})\}$.

\item
$z^{(l)}$ is more than $2\delta_1$ away from each critical point of $p_{u_1}$.

\end{itemize}

For $l\in\{1,\ldots,L\}$, let $\Delta>0$ denote the change in $\arg(p_{u_1})$ along $\widetilde{E^{(1)}}$, and assume that $\arg(p_{u_1})=0$ at the initial point of $\widetilde{E^{(l)}}$ (otherwise make the appropriate minor changes in the following argument).  Let $\widetilde{\gamma^{(l)}},\widehat{\gamma^{(l)}}:[0,\Delta]\to\mathbb{C}$ be the paths which parameterize $\widetilde{E^{(l)}}$ and $\widehat{E^{(l)}}$ with respect to $\arg(p_{u_1})$ and $p_{\widehat{u_1}}$ respectively (as described in Sub-section~\ref{subsect: Choice of widetilde{E}.}).  For $r\in[0,\Delta]$ such that $z^{(l)}=\widetilde{\gamma^{(l)}}(r)$, define $\widehat{z^{(l)}}\colonequals\widehat{\gamma^{(l)}}(r)$, the point in $\widehat{E^{(l)}}$ which corresponds to $z^{(l)}$.

Now for each $l\in\{1,\ldots,L\}$, let $\sigma^{(l)}:[0,1]\to\mathbb{C}$ be a parameterization of the portion of the gradient line of $p_{u_1}$ which connects $z^{(l)}$ to a point in $\partial\widetilde{D}$.  Let $y^{(l)}$ denote this point in $\partial\widetilde{D}$.  Recall that by Item~\ref{const: delta_1 item 3.} in the choice of $\delta_1$, for any $j,k\in\{1,\ldots,L\}$ with $j\neq{k}$, and for any $s,t\in[0,1]$, $|\sigma^{(j)}(s)-\sigma^{(k)}(t)|>2\delta_1$, and there is no edge of a critical level curve of $p_{u_1}$ other than the edges which contain $\sigma^{(j)}(0)$ and $\sigma^{(j)}(1)$ within $2\delta_1$ of $\sigma^{(j)}(s)$.

Since the region in $\widetilde{D}$ exterior to $\widetilde{\lambda_{\widetilde{D}}}$ does not contain any critical points of $p_{u_1}$, the gradient lines of $p_{u_1}$ do not cross in this region, and therefore the order of appearance of the edges $\widetilde{E^{(l)}}$ of $\widetilde{\lambda_{\widetilde{D}}}$ around $\widetilde{\lambda_{\widetilde{D}}}$ is the same as the order in which the corresponding points $y^{(l)}$ appear around $\partial\widetilde{D}$.  Thus we show that these points have the desired order.

Define $i_1\colonequals1$ and choose distinct indices $i_2,\ldots,i_L\in\{2,\ldots,L\}$ so that the order in which the edges of $\widetilde{\lambda_{\widetilde{D}}}$ appear around $\widetilde{\lambda_{\widetilde{D}}}$ is $\widetilde{E^{(i_1)}},\ldots,\widetilde{E^{(i_L)}}$.  In order to show that the edges of $\widetilde{\lambda_{\widetilde{D}}}$ appear in the same order as their corresponding edges in $\lambda_D$, it suffices to show that for each $l\in\{1,\ldots,L\}$, $i_l=l$.  By Item~\ref{const: delta_1 item 4.} in the choice of $\delta_1$, $\widehat{\lambda_{\widehat{D}}}\subset\widetilde{D}$.  Now for each $l\in\{1,\ldots,L\}$, extend $\sigma^{(l)}$ by a straight line from $z^{(l)}$ to $\widehat{z^{(l)}}$.  Let $s>0$ be the smallest number in the domain of this extended path such that $\sigma^{(l)}(s)$ is in $\widehat{E^{(l)}}$, and define $\widehat{\sigma^{(l)}}$ to be the restriction of $\sigma^{(l)}$ to $[0,s]$.

By choice of $\widetilde{\gamma^{(l)}}$ and $\widehat{\gamma^{(l)}}$ (using, for example, Item~\ref{const: delta_1 item 3.} in the choice of $\delta_1$) $\widehat{\sigma^{(l)}}$ does not intersect $\widehat{\sigma^{(j)}}$ for any $j\neq{l}$, and similarly $\widehat{\sigma^{(l)}}$ does not intersect the gradient line which connects $\widehat{x^{(1)}}$ to $\widehat{x^{(0)}}$.  Therefore the points $y^{(i_2)},\ldots,y^{(i_L)}$ (the end points of $\widehat{\sigma^{(i_2)}},\ldots,\widehat{\sigma^{(i_L)}}$ in $\partial\widetilde{D}$) appear in the same order as the corresponding edges $\widehat{E^{(i_2)}},\ldots,\widehat{E^{(i_L)}}$ appear in $\widehat{\lambda_{\widehat{D}}\setminus E^{(1)}}$.  By construction of $\langle\widehat{\Lambda}\rangle_{PC}$, that order is exactly $\widehat{E^{(2)}},\ldots,\widehat{E^{(L)}}$, so we conclude that for each $l\in\{1,\ldots,L\}$, $i_l=l$.

Note also that by the recursive assumption, $z(D)$ (which equals $Z(\lambda_D)$, which in turn equals the sum of the values of $z(\cdot)$ on each bounded face of $\lambda_D$) is equal to $z(\widetilde{D})$ (which similarly equals the sum of the values of $z(\cdot)$ on each bounded face of $\widetilde{\lambda_{\widetilde{D}}}$).  Our work in Sub-section~\ref{subsect: Choice of widetilde{E}.} implies that for each bounded face $F$ of $\lambda_D$, $z(F)=z(\widetilde{F})$.  Therefore we conclude that $\widetilde{\lambda_{\widetilde{D}}}$ has no bounded faces other than those corresponding to bounded faces of $\lambda_D$, and it in turn follows that $\widetilde{\lambda_{\widetilde{D}}}$ has no edges other than those corresponding to the edges in $\lambda_D$.  We conclude finally that $\langle\widetilde{\lambda_{\widetilde{D}}}\rangle_P=\langle\lambda_D\rangle_P$.

\subsection{SHOW RESPECT FOR GRADIENT MAPS}

We now wish to show that the correspondence established between $\lambda_D$ and $\widetilde{\lambda_{\widetilde{D}}}$ respects the gradient maps.  That is, if $y$ is a distinguished point in $\partial D$, and $\widetilde{y}$ is the corresponding distinguished point in $\partial\widetilde{D}$, then $g_{\widetilde{D}}(\widetilde{y})$ is the distinguished point in $\widetilde{\lambda_{\widetilde{D}}}$ which corresponds to the distinguished point $g_D(y)$ in $\lambda_D$.

As before, let $\langle\lambda\rangle_{PC}$ be a member of $PC$ used to construct $\langle\Lambda\rangle_{PC}$.  Let $D$ be a bounded face of $\lambda$, and let $\langle\widehat{\lambda}\rangle_{PC}$, $\widehat{D}$, and $\langle\widehat{\lambda_{\widehat{D}}}\rangle_{PC}$, and $\langle\widetilde{\lambda}\rangle_{PC}$, $\widetilde{D}$, and $\langle\widetilde{\lambda_{\widetilde{D}}}\rangle_{PC}$, be the objects for $\langle\widehat{\Lambda}\rangle_{PC}$ and $\langle\widetilde{\Lambda}\rangle_{PC}$ which correspond to $\langle\lambda\rangle_{PC}$, $D$, and $\langle\lambda_D\rangle_{PC}$ respectively.  Choose some edge $E_1$ of $\lambda$ in $\partial{D}$ and let $\widehat{E_1}$ and $\widetilde{E_1}$ be the corresponding edges in $\partial\widehat{D}$ and $\partial\widetilde{D}$.  Let $x^{(i_1)}$ and $x^{(i_2)}$ be the initial and final points of $E_1$, and let $\widehat{x^{(i_1)}},\widehat{x^{(i_2)}}$ and $\widetilde{x^{(i_1)}},\widetilde{x^{(i_2)}}$ be the corresponding distinguished points for $\langle\widehat{\Lambda}\rangle_{PC}$ and $\langle\widetilde{\Lambda}\rangle_{PC}$ respectively.  Let $\Delta_1$ denote the change in argument along $E_1$.  Assume that $a(x^{(i_1)})=0$, (otherwise make the appropriate minor changes).  Then let $\widehat{\gamma^{(1)}}:[0,\Delta_1]\to\mathbb{C}$ be the path which parameterizes $\widehat{E_1}$ according to $\arg(p_{\widehat{u_1}})$.  Let $\widetilde{\gamma^{(1)}}:[0,\Delta_1]\to\mathbb{C}$ be the path which parameterizes $\widetilde{E_1}$ according to $\arg(p_{\widetilde{u_1}})$.  Let $y$ be a distinguished point in $E_1$.  Let $\widehat{y}$ and $\widetilde{y}$ be the corresponding points in $\widehat{E_1}$ and $\widetilde{E_1}$.  Then by choice of $\widehat{\gamma^{(1)}}$ and $\widetilde{\gamma^{(1)}}$, $|\widehat{y}-\widetilde{y}|<\delta_2$.  Define $z$ to be the distinguished point in $\lambda_D$ such that $g_D(y)=z$.  Let $\widehat{z}$ and $\widetilde{z}$ be the distinguished points corresponding to $z$ for $\langle\widehat{\Lambda}\rangle_{PC}$ and $\langle\widetilde{\Lambda}\rangle_{PC}$.  Then since $g_D(y)=z$, the goal is to show that $g_{\widetilde{D}}(\widetilde{y})=\widetilde{z}$.  Let $E_2$ denote one of the edges of $\lambda_D$ which contains $z$ (if $z$ is a vertex of $\lambda_D$ then it will be contained in more than one edge of $\lambda_D$).  Let $x^{(j_1)}$ and $x^{(j_2)}$ be the initial and final points of $E_2$.  Let $\Delta_2$ denote the change in argument along $E_2$.  Let $\widehat{\gamma^{(2)}},\widetilde{\gamma^{(2)}}:[a(x^{(j_1)}),a(x^{(j_1)})+\Delta_2]\to\mathbb{C}$ be the paths which parameterize $\widehat{E_2}$ and $\widetilde{E_2}$ with respect to $\arg(p_{\widehat{u_1}})$ and $\arg(p_{u_1})$ respectively.  Then by choice of $\widehat{\gamma^{(2)}}$ and $\widetilde{\gamma^{(2)}}$, $|\widehat{z}-\widetilde{z}|<\delta_2$.  We will show the desired result in the case where $\langle\widehat{\lambda_{\widehat{D}}}\rangle_{PC}$ was formed using the scattering method.  The other cases are just simpler versions of this case.

Recall that $\widehat{D^{(1)}}$ denotes the face of $\widehat{\lambda_{\widehat{D}}}$ to which $\langle\xi_{F^{(1)}}\rangle_{PC}$ was assigned in the construction of $\langle\widehat{\Lambda}\rangle_{PC}$, and $\widehat{D^{(2)}}$ denotes the other face of $\widehat{\lambda_{\widehat{D}}}$.  We now will define a path $\widehat{\sigma}$ from $\widehat{y}$ to $\widehat{z}$.

\begin{case}
$z\in\partial{F^{(1)}}$.
\end{case}

In this case $\widehat{z}$ is a distinguished point in $\partial\widehat{F^{(1)}}$.  By definition of $\widehat{y}$ and $\widehat{z}$, $g_{\widehat{D}}(\widehat{y})=\widehat{z}$.  Therefore there is a portion of a gradient line $\widehat{\sigma}:[0,1]\to\mathbb{C}$ of $p_{\widehat{u_1}}$ which connects $\widehat{y}$ and $\widehat{z}$ and such that $\widehat{\sigma}((0,1))$ is contained in the portion of $\widehat{D}$ which is exterior to $\widehat{\lambda_{\widehat{D}}}$.

\begin{case}
$z\notin\partial{F^{(1)}}$.
\end{case}

In this case by the definition of the correspondence already established, $\widehat{z}$ is a point in an edge of $\widehat{\lambda_D\setminus{E^{(1)}}}$.  Recall that $\langle\widehat{\lambda_D\setminus{E^{(1)}}}\rangle_{PC}$ has been assigned to $\widehat{D^{(2)}}$ during the construction of $\langle\widehat{\lambda_{\widehat{D}}}\rangle_{PC}$, and by this construction, $g_{\widehat{D}}(\widehat{y})$ is a point in $\partial\widehat{D^{(2)}}$, and by the construction of $\langle\widehat{\Lambda}\rangle_{PC}$, $g_{\widehat{D^{(2)}}}(g_{\widehat{D}}(\widehat{y}))=\widehat{z}$.  Therefore there is a portion of a gradient line $\widehat{\sigma_1}:[0,1]\to\mathbb{C}$ of $p_{\widehat{u_1}}$ which connects $\widehat{y}$ to $g_{\widehat{D}}(\widehat{y})$, and another portion of a gradient line $\widehat{\sigma_2}:[0,1]\to\mathbb{C}$ of $p_{\widehat{u_1}}$ which connects $g_{\widehat{D}}(\widehat{y})$ to $\widehat{z}$.  Let $\widehat{\sigma}$ denote the concatenation of these two paths.

Therefore we have the desired path $\widehat{\sigma}$.  By Item~\ref{const: delta_2 item 3.} in the choice of $\delta_2$ and Item~\ref{const: rho_1 item 2.} in the choice of $\rho_1$, we conclude that there is a path $\widetilde{\sigma}:[0,1]\to\mathbb{C}$ such that $\widetilde{\sigma}(0)=\widetilde{y}$ and $\widetilde{\sigma}(1)=\widetilde{z}$ and, for all $r\in[0,1]$, $\arg(p_u(\widetilde{\sigma}(r)))=0$ and $|\widetilde{\sigma}(r)-\widehat{\sigma}(r)|<\delta_1$.  Moreover, since $|p_{\widehat{u_1}}|$ is strictly decreasing on $\widehat{\sigma}$, we may assume that $|p_{u_1}|$ is strictly decreasing on $\widetilde{\sigma}$.  Therefore for each $r\in(0,1)$, $|p_{u_1}(\widetilde{\sigma}(r))|\in(|p_{u_1}(\widetilde{z})|,|p_{u_1}(\widetilde{y})|)$.  Therefore also for each $r\in(0,1)$, $\widetilde{\sigma}(r)$ is in the portion of $\widetilde{D}$ which is in the unbounded face of $\widetilde{\lambda_{\widetilde{D}}}$.  By definition of $g_{\widetilde{D}}$, we conclude that $g_{\widetilde{D}}(\widetilde{y})=\widetilde{z}$.

Thus the correspondence established above between the graphs, vertices, and distinguished points of $\langle\Lambda\rangle_{PC}$ and those of $\langle\widetilde{\Lambda}\rangle_{PC}$ respects the gradient maps of $\langle\Lambda\rangle_{PC}$ and $\langle\widetilde{\Lambda}\rangle_{PC}$.

This concludes the recursive step for our proof.  We conclude that $\langle\Lambda\rangle_{PC}$ and $\langle\widetilde{\Lambda}\rangle_{PC}$ share all auxiliary data, and are thus equal.  That is, $\langle\Lambda\rangle_{PC}=\Pi(p_{u_1},G_{p_{u_1}})$.

\end{proof}

\subsection{COROLLARY AND EXAMPLE}

By inspecting the proof of Theorem~\ref{thm: Pi:H_a to PC_a is bijective.} just shown, we immediately have the following corollary.

\begin{corollary}\label{cor: Polynomial equivalence.}
For any $(f,G)\in{H_a}$ there is a polynomial $(p,G_p)$ such that $(p,G_p)\sim(f,G)$.
\end{corollary}

\begin{proof}
As shown in the proof of Theorem~\ref{thm: Pi:H_a to PC_a is bijective.}, we may find a polynomial $p$ such that $\Pi(p,G_p)=\Pi(f,G)$.  By Theorem~\ref{thm:Conformal equivalence iff Pi equivalence.}, it immediately follows that $(p,G_p)\sim(f,G)$.
\end{proof}

An example is in order here.  Unfortunately it may be very difficult in general either to determine the critical level curve configuration of a given generalized finite Blaschke product, or to find a polynomial with a given critical level curve configuration.  Therefore our example is quite simple.

\begin{example}
Consider the function

\[
f(z)=\frac{1}{.6}\left(z^2+\frac{9}{25}\right)e^z.
\]

The shaded region $G$ in Figure~\ref{fig:(f,G).} is one of the components of the set $\{w:|f(w)|<1\}$.  The critical point of $f$ in $G$ is at $z=-.2$ and the corresponding critical value is non-zero, so the vector $v=(f(-.2))$ is in $U_1$.  Therefore by Corollary~\ref{cor: Polynomial equivalence.} there is some polynomial $p$ such that $(p,G_p)\sim(f,G)$.  Consider, for example, the polynomial

\[
p(z)=z^2+f(-.2).
\]

The shaded region in Figure~\ref{fig:(p,D).} is the set $G_p=\{w:|p(w)|<1\}$.  The critical point of $p$ is at $z=0$.  It is easy to see that the critical value which arises from the critical point of $f$ in $G$ is equal to the critical value of $p$.  Since there is only one member of $PC_a$ which has a given single critical value and no other critical values, it follows that $\Pi(f,G)=\Pi(p,G_p)$.  Therefore by Theorem~\ref{thm:Conformal equivalence iff Pi equivalence.}, $(f,G)\sim(p,G_p)$.  That is, there is some conformal function $\phi:G\to{G_p}$ such that $f\equiv{p}\circ\phi$ on $G$.

\begin{figure}[htpb]
\begin{minipage}[b]{0.45\linewidth}
\centering
	\includegraphics[width=\textwidth]{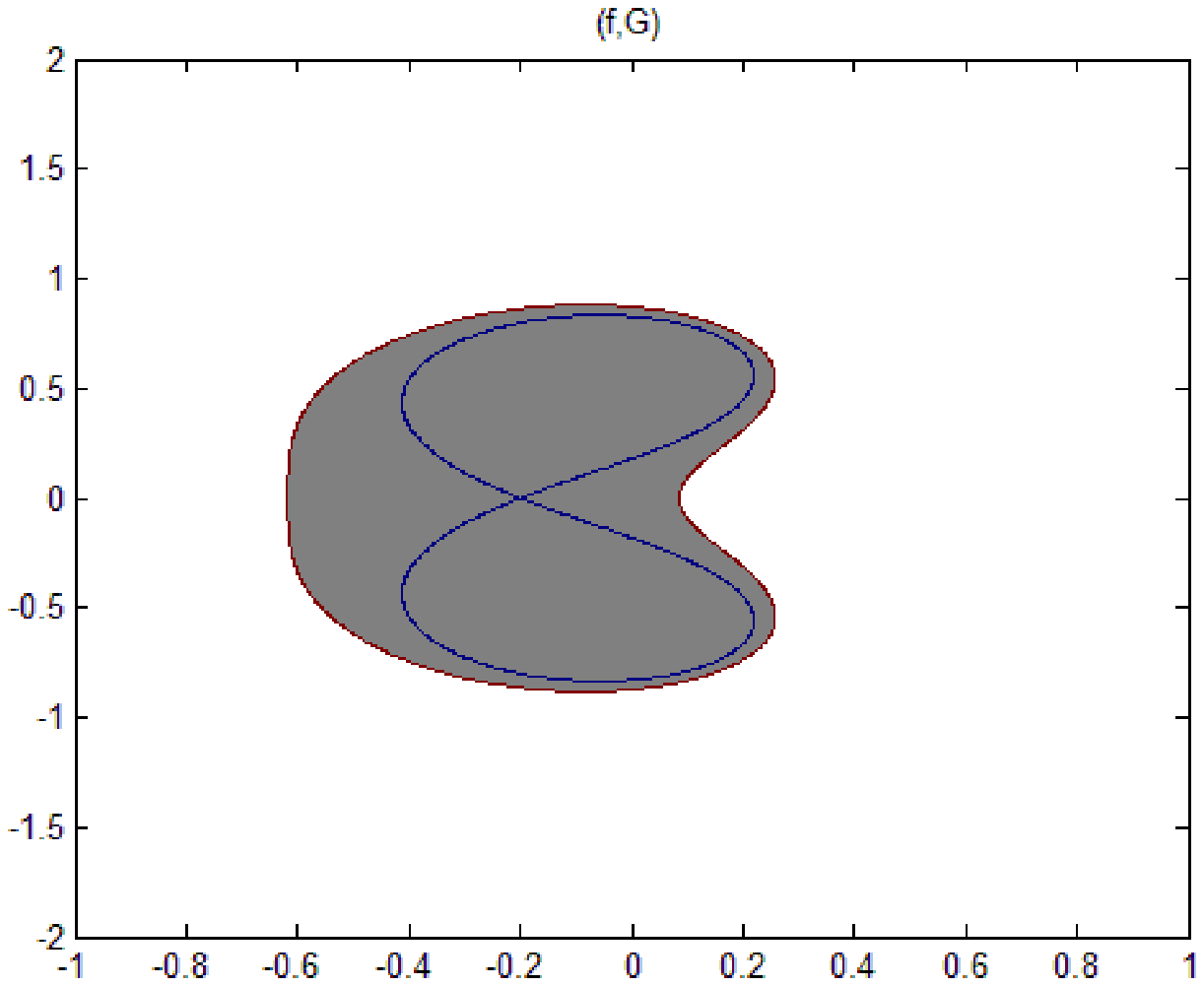}
  \caption{Tract of $f$}
	\label{fig:(f,G).}
\end{minipage}
\hspace{0.5cm}
\begin{minipage}[b]{0.45\linewidth}
\centering
	\includegraphics[width=\textwidth]{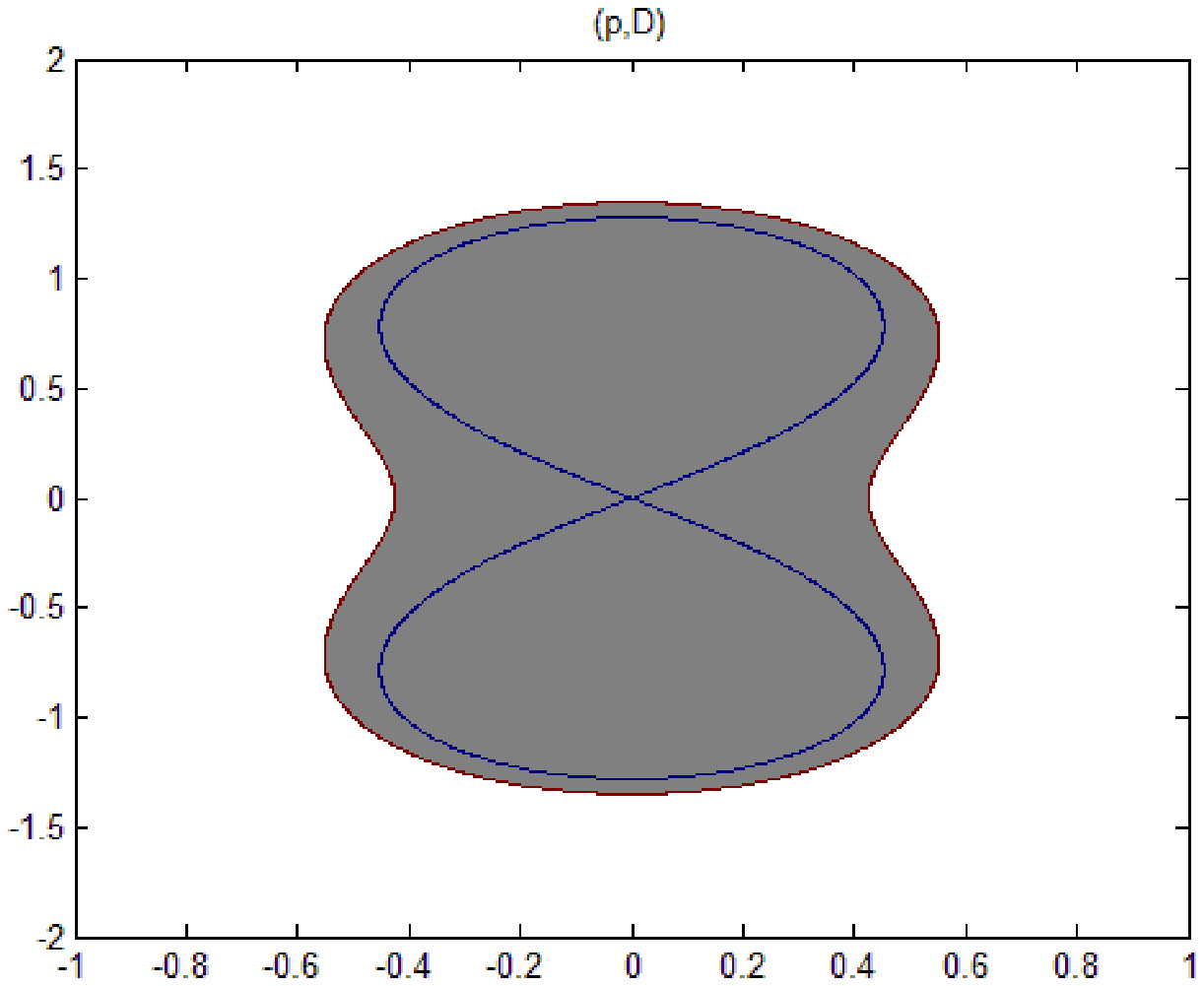}
	\caption{Tract of $p$}
	\label{fig:(p,D).}
\end{minipage}
\end{figure}
\end{example}
\section{DIRECTIONS FOR FUTURE WORK}%

Since Theorem~\ref{thm:Conformal equivalence iff Pi equivalence.} applies to generalized finite Blaschke ratios, not just generalized finite Blaschke products, it is a natural goal to extend Theorem~\ref{thm: Pi:H_a to PC_a is bijective.} to all generalized finite Blaschke ratios (namely, to show that $\Pi:H\to PC$ is a bijection).  This seems likely, but the problem of counting the possible critical level curve configurations of generalized finite Blaschke ratios with a given list of critical values is much more complicated than the similar process for generalized finite Blaschke products (which was one of the major steps in our proof of Theorem~\ref{thm: Pi:H_a to PC_a is bijective.}).

More promising is the extension of Corollary~\ref{cor: Polynomial equivalence.} to analytic functions on compact sets.  That is, in future work we intend to use level sets to show that if $f$ is any function analytic on the closure of $\mathbb{D}$, there is a polynomial $p$ and an injective map $\phi:\mathbb{D}\to G_p$ such that $f\equiv p\circ\phi$ on $\mathbb{D}$.

\appendix %

\section{SEVERAL LEMMATA}\label{app:Several Lemmata.}

\begin{lemma}\label{lemma:Analytic graphs have face with one edge boundary.}
Let $\lambda$ be a finite connected graph embedded in the plane.  Suppose that $\lambda$ has the following properties.

\begin{itemize}
\item
Each vertex of $\lambda$ is incident to more than one bounded face of $\lambda$.

\item
Each edge of $\lambda$ is incident to both a bounded face of $\lambda$ and the unbounded face of $\lambda$.
\end{itemize}

Then some bounded face of $\lambda$ has a single edge of $\lambda$ as its boundary.
\end{lemma}

\begin{proof}
Construct an auxiliary graph $\mathcal{T}$ from $\lambda$ as follows.  Start with $\lambda$ and place a vertex in each bounded face of $\lambda$.  For each bounded face $F$ of $\lambda$, draw an edge from the $F$-vertex to each vertex of $\lambda$ in $\partial F$.  Deleting all the original edges of $\lambda$, we let $\mathcal{T}$ denote the remaining graph, which is connected by the assumption that each edge of $\lambda$ is incident to both the unbounded face of $\lambda$ and a bounded face of $\lambda$.  Any cycle in $\mathcal{T}$ would contradict the same assumption on the edges of $\lambda$, so $\mathcal{T}$ is a tree.  Finally, the assumption that each vertex is incident to more than one bounded face of $\lambda$ implies that each leaf of $\mathcal{T}$ arises from a bounded face of $\lambda$.  Finally, if $v$ is a leaf of $\mathcal{T}$, the corresponding face of $\lambda$ can have only one edge of $\lambda$ in its boundary.
\end{proof}

\begin{lemma}\label{lemma:If hat{v} is close to v, then any hat{u} is close to some u.}
Let $v\in\mathbb{C}^{n-1}$and $\rho>0$  be given.  Then there exists a $\nu>0$ such that if $\widehat{v}\in\mathbb{C}^{n-1}$ and $|v-\widehat{v}|<\nu$, and $\widehat{u}\in\Theta^{-1}(\widehat{v})$, then there is a $u\in\Theta^{-1}(v)$ such that $|u-\widehat{u}|<\rho$.
\end{lemma}

\begin{proof}
Suppose by way of contradiction that the desired result fails.  Thus there exists a sequence $\{v_k\}_{k=1}^{\infty}\subset\mathbb{C}^{n-1}$ such that $v_k\to v$, and for each $k$ we may choose a $u_k\in\Theta^{-1}(v_k)$ such that $|u_k-u|>\rho$ for each $u\in\Theta^{-1}(v)$.

Define $K\colonequals\{v_k\}_{k=1}^\infty\cup\{v\}$.  $K$ is compact, and $\Theta$ is proper, so $\Theta^{-1}(K)$ is compact.  $\{u_k\}_{k=1}^\infty\subset\Theta^{-1}(K)$, so there is a subsequence $\{u_{k_l}\}_{l=1}^\infty$ which converges to some point $u$.  Since $\Theta$ is continuous,

\[
\Theta(u)=\Theta(\lim_{l\to\infty}u_{k_l})=\lim_{l\to\infty}\Theta(u_{k_l})=v.
\]

Thus $\{u_{k_l}\}_{l=1}^\infty$ converges to a point in $\Theta^{-1}(v)$, which is a contradiction of the choice of $\{u_k\}_{k=1}^\infty$.
\end{proof}

\begin{definition}
If $\gamma:[\alpha,\beta]\to\mathbb{C}$  is a path, and $f$ is a function which is analytic and non-zero on the image of $\gamma$, then we say that $\gamma$ is parameterized according to $\arg(f)$ if for each $r\in[\alpha,\beta]$, $\arg(f(\gamma(r)))=r$.
\end{definition}

\begin{lemma}\label{lemma:u and hat{u} close implies p_u and p_{hat{u}} close.}
Let $v\in{V_{n-1}}$, and $\tau>0$ and be given.  Then there exists a $\rho>0$ such that if $u\in\Theta^{-1}(v)$, and $\widehat{u}\in\Theta^{-1}(V_{n-1})$ such that $|u-\widehat{u}|<\rho$, then the following holds.  $G_{p_{\widehat{u}},1}\subset{G_{p_u,2}}$, and $|p_u(z)-p_{\widehat{u}}(z')|<\tau$ for all $z,z'\in{G_{p_u,2}}$ satisfying $|z-z'|<\rho$.
\end{lemma}

\begin{proof}
This follows from the fact that the coefficients of $p_u$ are polynomials in the components of $u$.
\end{proof}

\begin{lemma}\label{lemma:Points close to x have path to x in lambda.}
Let $v\in{V_{n-1}}$, and $\delta^{(1)}>0$ be given.  There exists some $\delta^{(2)}\in(0,\delta^{(1)})$ such that if $u\in\Theta^{-1}(v)$, and $\lambda$ is a critical level curve of $(p_u,G_{p_u})$ (with $|f|\equiv\epsilon>0$ on $\lambda$), and $x\in\lambda$ is a critical point of $p_u$, then if $y\in{B_{\delta^{(2)}}(x)}$ satisfies $|f(y)|=\epsilon$, then there is a path $\sigma$ from $y$ to $x$ which is contained in $\lambda\cap{B_{\delta^{(1)}}(x)}$.  Moreover, we may choose $\sigma$ so that $\arg(p_u)$ is strictly increasing or strictly decreasing along $\sigma$, and parameterized according to $\arg(p_u)$.
\end{lemma}

\begin{proof}
Since $\Theta^{-1}(v)$ is finite (\cite{BCN}), we need only show the result for some fixed $u\in\Theta^{-1}(v)$.  Let $u\in\Theta^{-1}(v)$, and let $\lambda$ be one of the critical level curves of $(p_u,G_{p_u})$, (with $|f|\equiv\epsilon>0$ on $\lambda$).  Let $x\in\lambda$ be a given critical point of $p_u$.  Let $k\in\mathbb{N}$ denote the multiplicity of $x$ as a zero of ${p_u}'$.  Then there is some neighborhood $D\subset{B_{\delta^{(1)}}(x)}$ of $x$ and $S>0$ and conformal map $\phi:D\to{B_S(p_u(0))}$ such that $p_u(z)=\phi(z)^{k+1}+p_u(x)$ for all $z\in{D}$.  Define $f(w)\colonequals{w}^{k+1}+p_u(x)$.  The level curves of $f$ are well understood.  Let $L$ denote the level curve of $f$ which contains $0$.  Then if $w\in{L}$, there is a path in $L$ from $w$ to $0$ which is contained in $B_{|w|}(0)$, along which $\arg(f)$ is either strictly increasing or strictly decreasing.  Choose some $r>0$ such that $B_{r}(x)\subset{D}$.  Let $y\in{B_{r}(x)}$ be any point such that $|p_u(y)|=\epsilon$.  Then $\phi(y)\in L$.  Let $\sigma^{(1)}$ denote the path in $B_{|\phi(y)|}(0)$ from $\phi(y)$ to $0$ along which $\arg(f)$ is strictly increasing or strictly decreasing.  Then if we define $\sigma\colonequals\phi^{-1}\circ\sigma^{(1)}$, $\sigma\subset\phi^{-1}(B_{|\phi(y)|}(0))\subset{D}\subset{B_{\delta^{(1)}}(x)}$, and for each $t\in[0,1]$, $p_u(\sigma(t))=f(\sigma^{(1)}(t))$, so $\arg(p_u)$ is either strictly increasing or strictly decreasing along $\sigma$.

Since $p_u$ has only finitely many critical points in $\lambda$, we may choose $\delta^{(2)}>0$ to be smaller than the $r$ chosen above for each critical point $x$ of $p_u$ in $\lambda$, and this $\delta^{(2)}$ has the desired property.

\end{proof}

\begin{lemma}\label{lemma:Points close to x have gradient path to x.}
Let $v\in{V_{n-1}}$, and $\delta^{(1)}>0$ be given.  There exists some $\delta^{(2)}\in(0,\delta^{(1)})$ such that the following holds.  Let $u\in\Theta^{-1}(v)$ be given.  Let $\lambda$ be a critical level curve of $(p_u,G_{p_u})$ (with $|p_u|\equiv\epsilon>0$ on $\lambda$), and let $x\in\lambda$ be a critical point of $p_u$.  Then if $y\in{B_{\delta^{(2)}}(x)}$ satisfies $\arg(p_u(y))=\arg(p_u(x))=0$, then there is a path $\sigma$ from $y$ to $x$ which is contained in $B_{\delta^{(1)}}(x)$ and such that $\arg(p_u(\sigma(r)))=\arg(p_u(x))$ for all $r$.  Moreover we may choose $\sigma$ so that $|p_u|$ is strictly increasing or strictly decreasing along $\sigma$, and parameterized according to $|p_u|$.
\end{lemma}

\begin{proof}
Essentially the same argument for Lemma~\ref{lemma:Points close to x have path to x in lambda.} works here.
\end{proof}

\begin{lemma}\label{lemma:If g is close to f then they have same values close to each other.}
Given any GFBP $(f,G)$ and $\eta>0$, and any compact set $G'\subset{G}$ which does not contain any critical points of $f$, there exists $\tau>0$ such that if $g$ is analytic on $G$, and $|f(z)-g(z)|<\tau$ for all $z\in{G}$, then the following hold:

\begin{enumerate}
\item\label{item:Points far from critical points of f.}
If $z^{(0)}\in{G'}$, and $w^{(1)}\in{B_{\tau}(f(z^{(0)}))}$, then there is a point $z^{(1)}\in{B_{\eta}(z^{(0)})}$ such that $g(z^{(1)})=w^{(1)}$.  (In particular we may put $w^{(1)}=f(z^{(0)})$.)

\item\label{item:Points far from critical points of g.}
If $z^{(0)}\in{G'}$, and $w^{(1)}\in{B_{\tau}(g(z^{(0)}))}$, then there is a point $z^{(1)}\in{B_{\eta}(z^{(0)})}$ such that $f(z^{(1)})=w^{(1)}$.  (In particular we may put $w^{(1)}=g(z^{(0)})$.)
\end{enumerate}
\end{lemma}

\begin{proof}
This follows from elementary properties of an analytic function of one complex variable, including primarily the maximum modulus principle.
\end{proof}

\begin{definition}
For $u\in\mathbb{C}^{n-1}$, if $\gamma:[0,1]\to\mathbb{C}$ is a path, and $0<a<b<1$, then for $0<\rho<\delta$, we say that $\gamma$ takes an $(\rho,\delta)$ trip on $[a,b]$ with respect to $p_u$ if the following hold.

\begin{itemize}
\item
There is some $\iota>0$ such that for all $r\in(a-\iota,a)\cup(b,b+\iota)$, $\gamma(r)$ is less than $\rho$ away from some critical point of $p_u$ (possibly different critical points of $p_u$ for different values of $r$).

\item
For each $r\in(a,b)$, $\gamma(r)$ is greater than or equal to $\rho$ away from every critical point of $p_u$.

\item
There is some $r\in(a,b)$ such that $\gamma(r)$ is greater than $\delta$ away from every critical point of $p_u$.
\end{itemize}
\end{definition}

\begin{lemma}\label{lemma:If p_{widehat{u}} has a level curve edge, then p_u has a level curve edge.}
Fix some $v=(v^{(1)},\ldots,v^{(n-1)})\in{V_{n-1}}$ not the zero vector, and $\delta^{(1)}>0$.  Then there exists a constant $\rho>0$ such that for each $u\in\Theta^{-1}(v)$, the following holds.  Fix some $\widehat{u}\in{B_{\rho}}(u)$ such that, if we define $\widehat{v}=(\widehat{v^{(1)}},\ldots,\widehat{v^{(n-1)}})\colonequals\Theta(\widehat{u})$, then $\arg(\widehat{v^{(k)}})=\arg(v^{(k)})$ for each $k\in\{1,\ldots,n-1\}$.  For some $k\in\{1,\ldots,n-1\}$ with $|v^{(k)}|\neq0$, let $\widehat{\lambda}$ denote the level curve of $p_{\widehat{u}}$ which contains $\widehat{u^{(k)}}$.  Let $\widehat{E}$ denote some edge of $\widehat{\lambda}$ which is incident to $\widehat{u^{(k)}}$, and let $\widehat{\gamma}$ denote a parameterization of $\widehat{E}$ according to $\arg(p_{\widehat{u}})$ beginning with $\widehat{u^{(k)}}$,  with domain $[\alpha,\beta]$ (where $\alpha<\beta$ if $\arg(p_{\widehat{u}})$ is increasing on $\widehat{E}$, and $\alpha>\beta$ otherwise).  Then if we let $\lambda$ denote the critical level curve of $p_u$ containing $u^{(k)}$, there is a path $\gamma:[\alpha,\beta]\to\lambda$ such that $\gamma(\alpha)=u^{(k)}$ and, for each $t\in[\alpha,\beta]$, $\arg(p_u(\gamma(t)))=t$ and $|\gamma(t)-\widehat{\gamma}(t)|<\delta^{(1)}$.
\end{lemma}

\begin{proof}

We will show that the result of the lemma holds for any fixed $u\in\Theta^{-1}(v)$, which will suffice because $\Theta^{-1}(v)$ is finite by \cite{BCN}.  Broadly speaking, the idea of the proof is that close to any critical point of $p_u$, $\gamma$ can be defined using Lemma~\ref{lemma:Points close to x have path to x in lambda.}, and far from the critical points of $p_u$, $\gamma$ can be defined using Lemma~\ref{lemma:If g is close to f then they have same values close to each other.}.  The notion of a $(\rho,\delta)$ trip defined above is how we will quantify "close" and "far" from the critical points of $p_u$.

Reduce $\delta^{(1)}>0$ if necessary so that $\delta^{(1)}<\dfrac{\text{mindiff}({u_1}^{(1)},\ldots,{u_1}^{(n-1)})}{2}$ and for each $l\in\{1,\ldots,n-1\}$ with $|v^{(l)}|\neq0$, if $|z-u^{(l)}|<\delta^{(1)}$, then $|p_u(z)-v^{(l)}|<\frac{\text{mindiff}(0,|v^{(1)}|,\ldots,|v^{(n-1)}|)}{4}$.  Of course $\frac{\text{mindiff}(0,|v^{(1)}|,\ldots,|v^{(n-1)}|)}{4}\leq\frac{|v^{(l)}|}{4}$, so by geometry, if $|z-u^{(l)}|<\delta^{(1)}$ then $|\arg(p_u(z))-\arg(v^{(l)})|<\frac{\pi}{4}$.  Note that this also implies that if $u^{(l)}$ is a critical point of $p_u$ at which $p_u\neq0$, and $\sigma$ is a path contained in $B_{\delta^{(1)}}(u^{(l)})$, then $\Delta_{\arg}(p_u,\sigma)<\frac{\pi}{2}$.

By Lemma~\ref{lemma:Points close to x have path to x in lambda.}, we may choose $\delta^{(2)}\in(0,\frac{\delta^{(1)}}{4})$ such that the following holds.  If $y\in{B_{\delta^{(2)}}(u^{(l)})}$ for some $l\in\{1,\ldots,n-1\}$ such that $|p_u(y)|=|v^{(l)}|\neq0$, then there is a path $\sigma$ from $y$ to $u^{(l)}$ contained in $B_{\frac{\delta^{(1)}}{2}}(u^{(l)})\cap{E_{p_u,|v^{(l)}|}}$ such that $\arg(p_u)$ is strictly monotonic along $\sigma$.

Since $p_u$ is an open mapping, we may choose some $M>0$ small enough so that for each $k\in\{1,\ldots,n-1\}$, $B_{2M}(v^{(k)})\subset{p_u}(B_{\delta^{(2)}}(u^{(k)}))$.  By Lemma~\ref{lemma:u and hat{u} close implies p_u and p_{hat{u}} close.}, we may choose a $\rho^{(1)}>0$ so that $\rho^{(1)}<\frac{\delta^{(2)}}{2}$, and if $\widehat{u}\in{B_{\rho^{(1)}}(u)}$, then $|p_u(z)-p_{\widehat{u}}(\widehat{z})|<M$ for all $z,\widehat{z}\in{G_{p_u}}$ such that $|z-\widehat{z}|<\rho^{(1)}$.

Let $K$ denote the set of all points $x$ in $G_{p_u}$ such that the following hold.

\begin{itemize}
\item
$x\in{E_{p_u,|v^{(k)}|}}$ for some $k\in\{1,\ldots,n-1\}$.

\item
$|x-u^{(k)}|\geq\frac{\delta^{(2)}}{2}$ for each $k\in\{1,\ldots,n-1\}$.
\end{itemize}

By the compactness of $K$, we may choose an $\eta>0$ such that the following holds.

\begin{itemize}
\item
$\eta<\min(d(\{z\},\partial{G_{p_u}}):z\in{K})$.

\item
$p_u$ is injective on $B_{\eta}(x)$ for each $x\in{K}$.  (Since $K$ does not contain any critical point of $p_u$.)

\item
$\eta<\rho^{(1)}$.  (Thus $|x-u^{(l)}|>\eta$ for each $l\in\{1,\ldots,n-1\}$, since $\rho^{(1)}<\frac{\delta^{(2)}}{2}$.)
\end{itemize}

Define $G'\colonequals\{x\in{G_{p_u}}:d(x,\partial{G_{p_u}})\geq\eta,d(x,u^{(k)})\geq\eta\text{ for each }k\}$.  By Lemma~\ref{lemma:If g is close to f then they have same values close to each other.}, we may choose $\tau>0$ so that $\tau<\min(M,\frac{\text{mindiff}(0,|v^{(1)}|,\ldots,|v^{(n-1)}|)}{2})$, and if $f$ is analytic on $G'$ with $|f-p_u|<\tau$ on $G'$, then for all $x$ in $G'$, the following hold.

\begin{itemize}
\item
For any $w\in{B_{\tau}(p_u(x))}$, there is a $y\in{B_{\eta}(x)}$ with $f(y)=w$.

\item
For any $w\in{B_{\tau}(f(x))}$, there is a $y\in{B_{\eta}(x)}$ with $p_u(y)=w$.
\end{itemize}

By Lemma~\ref{lemma:u and hat{u} close implies p_u and p_{hat{u}} close.} and the continuity of $\Theta$, we may choose $\rho\in(0,\rho^{(1)})$ so that if $\widehat{u}\in{B_{\rho}(u)}$, then $|p_u(z)-p_{\widehat{u}}(\widehat{z})|<\tau$ for all $z,\widehat{z}\in{G_{p_u}}$ such that $|z-\widehat{z}|<\rho$, and for $\widehat{v}=(\widehat{v^{(1)}},\ldots,\widehat{v^{(n-1)}})\colonequals\Theta(\widehat{u})$, $|\widehat{v}-v|<\tau$.  We now show that the statement of the lemma holds for the chosen $\rho$.

Let $\widehat{u}\in{B_{\rho}}(u)$ be chosen.  Fix some $k\in\{1,\ldots,n-1\}$ such that $|v^{(k)}|\neq0$.  Note that since $\tau<\frac{\text{mindiff}(0,|v^{(1)}|,\ldots,|v^{(n-1)}|)}{2}$, and $|v-\widehat{v}|<\tau$, we have $|\widehat{v^{(k)}}|>\frac{\text{mindiff}(0,|v^{(1)}|,\ldots,|v^{(n-1)}|)}{2}$.  Let $\widehat{\lambda}$ denote the level curve of $p_{\widehat{u}}$ which contains $\widehat{u^{(k)}}$.  Let $\widehat{E}$ denote some edge of $\widehat{\lambda}$ which is incident to $\widehat{u^{(k)}}$.  Let $\alpha$ denote some choice of the argument of $p_{\widehat{u}}(\widehat{u^{(k)}})$, and let $\widehat{\gamma}:[\alpha,\beta]\to\widehat{\lambda}$ be a path which parameterizes $\widehat{E}$ according to the argument of $p_{\widehat{u}}$ beginning with $\widehat{u^{(k)}}$.  (Here we assume that $\arg(p_{\widehat{u}})$ is increasing as $\widehat{E}$ is traversed beginning with $\widehat{u^{(k)}}$, and thus $\beta>\alpha$.  Otherwise make the appropriate minor changes.)

We will now define a path $\gamma$ with domain $[\alpha,\beta]$ which has the desired properties.  We first identify the sub-intervals of $[\alpha,\beta]$ on which $\widehat{\gamma}$ takes a $(\rho^{(1)},\delta^{(1)}/2)$ trip (with respect to $p_u$).  On these sub-intervals we will define $\gamma$ in one way, and on the intervening sub-intervals we will define $\gamma$ in another way.  Note that by the definition of a $(\rho^{(1)},\delta^{(1)}/2)$ trip over an interval, if $\widehat{\gamma}$ takes $(\rho^{(1)},\delta^{(1)}/2)$ trips over two sub-intervals $I^{(1)},I^{(2)}\subset[\alpha,\beta]$, then either $I^{(1)}=I^{(2)}$, or $I^{(1)}$ and $I^{(2)}$ are disjoint. Therefore since $\widehat{\gamma}$ is a rectifiable path, and any sub-interval on which $\widehat{\gamma}$ takes a $(\rho^{(1)},\delta^{(1)}/2)$ trip must have length at least $\dfrac{\delta^{(1)}}{2}-\rho^{(1)}$, $\widehat{\gamma}$ takes at most finitely many distinct $(\rho^{(1)},\delta^{(1)}/2)$ trips.

\begin{case}\label{case:widehat{gamma} takes p_u trips.}
$\widehat{\gamma}$ takes a $(\rho^{(1)},\delta^{(1)}/2)$ trip on some sub-interval of $[\alpha,\beta]$.
\end{case}

Let $[r^{(1)},s^{(1)}],\ldots,[r^{(N)},s^{(N)}]\subset[\alpha,\beta]$ be the disjoint subintervals of $[\alpha,\beta]$ over which $\gamma$ takes $(\rho^{(1)},\delta^{(1)}/2)$ trips, ordered so that $s^{(k)}<r^{(k+1)}$ for each $k\in\{1,\ldots,N-1\}$.  Fix for the moment some $j\in\{1,\ldots,N\}$ and some $r\in[r^{(j)},s^{(j)}]$.

For all $t\in[r^{(j)},s^{(j)}]$, define $w(t)\colonequals|v^{(k)}|e^{it}$.  Then since $\tau<\text{mindiff}(0,|v^{(1)}|,\ldots,|v^{(n-1)}|)$, $|w(t)-p_{\widehat{u}}(\widehat{\gamma}(t))|<\tau$, so there is some $y\in{B_{\eta}(\widehat{\gamma}(r))}$ such that $p_u(y)=w(r)$.  Moreover, since $p_u$ is injective in $B_{\eta}(\widehat{\gamma}(r))$, this choice of $y$ is unique.  Define $\gamma(r)=y$.

Since $p_u$ is injective on $B_{\eta}(\widehat{\gamma}(r))$ for each $r\in[r^{(j)},s^{(j)}]$, and $p_u$ is an open mapping, it is easy to show that $\gamma$ is a continuous function, and thus a path  from $\gamma(r^{(j)})$ to $\gamma(s^{(j)})$.  Further, if $r\in[r^{(j)},s^{(j)}]$, $|p_u(\gamma(r))|=|w(r)|=|v^{(k)}|$.  Therefore we conclude that $\gamma|_{[r^{(j)},s^{(j)}]}$ is a path in $E_{p_u,|v^{(k)}|}$, and by construction, for each $r\in[r^{(j)},s^{(j)}]$, $|\widehat{\gamma}(r)-\gamma(r)|<\eta$ and $\arg(p_u(\gamma(r)))=r$.  Having done this for each $j\in\{1,\ldots,N\}$, we now wish to define $\gamma$ on $(s^{(j)},r^{(j+1)})$ for each $j\in\{1,\ldots,N-1\}$.

Again fix for the moment some new $j\in\{1,\ldots,N\}$.  Since there is no sub-interval of $(s^{(j)},r^{(j+1)})$ on which $\widehat{\gamma}$ takes a $(\rho^{(1)},\delta^{(1)}/2)$ trip, $\widehat{\gamma}(r)$ is within $\delta^{(1)}/2$ of some critical point of $p_u$ for each $r\in(s^{(j)},r^{(j+1)})$.  However $\delta^{(1)}<\frac{\text{mindiff}(u)}{2}$, thus there is some unique $l\in\{1,\ldots,n-1\}$ such that for each $r\in(s^{(j)},r^{(j+1)})$, $|\widehat{\gamma}(r)-u^{(l)}|\leq\delta^{(1)}/2$.  Since $\widehat{\gamma}$ takes a $(\rho^{(1)},\delta^{(1)}/2)$ trip over $[r^{(j)},s^{(j)}]$, $|\widehat{\gamma}(s^{(j)})-u^{(l)}|=\rho^{(1)}$.  Therefore

\[
|\gamma(s^{(j)})-u^{(l)}|\leq|\gamma(s^{(j)})-\widehat{\gamma}(s^{(j)})|+|\widehat{\gamma}(s^{(j)})-u^{(l)}|<\eta+\rho^{(1)}<\delta^{(2)}.
\]

In addition to this, $|p_u(\gamma(s^{(j)}))|=|v^{(k)}|$, so by choice of $\delta^{(1)}$,

\[
||v^{(k)}|-|v^{(l)}||=||p_u(\gamma(s^{(j)}))|-|v^{(l)}||<|p_u(\gamma(s^{(j)}))-v^{(l)}|<\text{mindiff}(0,|v^{(1)}|,\ldots,|v^{(n-1)}|).
\]

Therefore we conclude that $|p_u(\gamma(s^{(j)}))|=|v^{(l)}|=|v^{(k)}|$.  Then by choice of $\delta^{(2)}$, there is some path $\sigma^{(1)}$ from $\gamma(s^{(j)})$ to $u^{(l)}$ contained in $B_{\frac{\delta^{(1)}}{2}}(u^{(l)})\cap{E_{p_u,|v^{(l)}|}}$ such that $\arg(p_u)$ is strictly monotonic on $\sigma^{(1)}$, and $\sigma^{(1)}$ is parameterized according to $\arg(p_u)$.  Since $\arg(p_{\widehat{u}})$ is increasing along $\widehat{\gamma}$, $\arg(p_u)$ is increasing along the portions of $\gamma$ which have already been defined.  Let $D$ denote an open region containing $\gamma(s^{(j)})$ on which $p_u$ is injective.  Choose some $t^{(0)}\in(r^{(j)},s^{(j)})$ such that $\gamma(t^{(0)},s^{(j)})\subset{D}$.  If $\arg(p_u)$ is decreasing on $\sigma^{(1)}$, then since $p_u$ is injective on $D$, for each $r\in(s^{(j)},t^{(0)})$, $\sigma^{(1)}(r)=\gamma(r)$.  Furthermore, since $p_u$ is injective in a neighborhood of each point of $\gamma([r^{(j)},s^{(j)}])$, $\sigma^{(1)}$ must continue to trace back along the entire length of $\gamma([r^{(j)},s^{(j)}])$.  This is because both $\sigma^{(1)}$ and $\gamma$ are parameterized according to $\arg(p_u)$, so any branching off of $\sigma^{(1)}$ from $\gamma$ would have to be a critical point of $p_u$.  However $\sigma^{(1)}$ may not trace back along $\gamma([r^{(j)},s^{(j)}])$ because the image of $\sigma^{(1)}$ is contained in $B_{\frac{\delta^{(1)}}{2}}(u^{(l)})$.  Therefore we conclude that $\arg(p_u)$ is increasing on $\sigma^{(1)}$.

By very similar reasoning we may obtain a path $\sigma^{(2)}$ from $u^{(l)}$ to $\gamma(r^{(j+1)})$ contained in $B_{\frac{\delta^{(1)}}{2}}(u^{(l)})\cap{E_{p_u,|v^{(l)}|}}$ parameterized according to $\arg(p_u)$, and along which $\arg(p_u)$ is increasing.  Moreover, the choice of $\delta^{(1)}$ may now be used to show that the concatenation of these two paths may be assumed to have domain $(s^{(j)},r^{(j+1)})$.  Therefore we define $\gamma(r)\colonequals\sigma(r)$ for each $r\in(s^{(j)},r^{(j+1)})$.  With this definition, we have that for each $r\in(s^{(j)},r^{(j+1)})$, $\arg(p_u(\gamma(r)))=r$, and

\[
|\gamma(r)-\widehat{\gamma(r)}|\leq|\gamma(r)-u^{(l)}|+|u^{(l)}-\widehat{\gamma(r)}|<\dfrac{\delta^{(1)}}{2}+\dfrac{\delta^{(1)}}{2}=\delta^{(1)}.
\]

We extend $\gamma$ in this manner to $(s^{(j)},r^{(j+1)})$ for each $j\in\{1,\ldots,N-1\}$.  Moreover, we may extend $\gamma$ using the exactly similar construction to $[\alpha,r^{(1)})$ and $(s^{(N)},\beta]$, and this extended $\gamma$ has all of the desired properties.

\begin{case}
There is no sub-interval of $[\alpha,\beta]$ along which $\widehat{\gamma}$ takes a $(\rho^{(1)},\delta^{(1)}/2)$ trip.
\end{case}

Then either $|\widehat{\gamma}(r)-u^{(k)}|\leq\delta^{(1)}/2$ for all $r\in[\alpha,\beta]$, or there is some $r^{(0)}\in(\alpha,\beta)$ such that for all $r\in[\alpha,r^{(0)}]$, $|\widehat{\gamma}(r)-u^{(k)}|\leq\delta^{(1)}/2$, and for all $r\in(r^{(0)},\beta]$, $\widehat{\gamma}$ is greater than $\rho^{(1)}$ from any critical point of $p_u$.

\begin{subcase}
$|\widehat{\gamma}(r)-u^{(k)}|\leq\delta^{(1)}/2$ for all $r\in[\alpha,\beta]$.
\end{subcase}

In this case, we construct $\gamma$ using the same method as in the second part of Case~\ref{case:widehat{gamma} takes p_u trips.}.

\begin{subcase}
There is some $r^{(0)}\in(\alpha,\beta)$ such that for all $r\in[\alpha,r^{(0)}]$, $|\widehat{\gamma}(r)-u^{(k)}|\leq\delta^{(1)}/2$, and for all $r\in(r^{(0)},\beta]$, $\widehat{\gamma}$ is greater than $\rho^{(1)}$ from any critical point of $p_u$.
\end{subcase}

In this case, we construct $\gamma$ on $[\alpha,r^{(0)})$ using the same method as in the second part of Case~\ref{case:widehat{gamma} takes p_u trips.}, and we construct $\gamma$ on $[r^{(0)},\beta]$ using the same method as in the first part of Case~\ref{case:widehat{gamma} takes p_u trips.}.
\end{proof}

\begin{lemma}\label{lemma:If p_{widehat{u}} has a gradient line, then p_u has a gradient line.}
Fix some $v=(v^{(1)},\ldots,v^{(n-1)})\in{V_{n-1}}$ not the zero vector, and $\delta^{(1)}>0$.  Then there exists constants $\rho,\delta^{(2)}>0$ such that the following hold.  Let $u\in\Theta^{-1}(v)$ be chosen, and fix some $\widehat{u}\in{B_{\rho}}(u)$.  Let $\widehat{x_1},\widehat{x_2}\in{G}_{p_{\widehat{u}}}$ be given such that $\arg(p_u(x_1))=\arg(p_u(x_2))=0$, and such that there is a path $\widehat{\sigma}:[0,1]\to{G_{p_u}}$ such that $\widehat{\sigma(0)}=\widehat{x_1}$ and $\widehat{\sigma(1)}=\widehat{x_2}$ and $\arg(p_{\widehat{u}}(\widehat{\sigma}(r)))=0$ for all $r\in[0,1]$.  Then if $x_1,x_2\in{G_{p_{\widehat{u}}}}$ are such that $\arg(p_u(x_1))=\arg(p_u(x_1))=0$ and $|\widehat{x_1}-x_1|<\delta^{(2)}$ and $|\widehat{x^{(2)}}-x^{(2)}|<\delta^{(2)}$, then there is a path $\sigma:[0,1]\to{G_{p_u}}$ such that $\sigma(0)=x_1$, $\sigma(1)=x_2$, and for all $r\in[0,1]$, $\arg(p_u(\sigma(r)))=0$ and $|\widehat{\sigma}(r)-\sigma(r)|<\delta^{(1)}$.  Moreover, if $|p_{\widehat{u}}|$ is strictly increasing or strictly decreasing on $\widehat{\sigma}$, then we may assume that $|p_u|$ is strictly increasing or strictly deacreasing on $\sigma$ respectively.
\end{lemma}

\begin{proof}
The exact same method of proof used for Lemma~\ref{lemma:If p_{widehat{u}} has a level curve edge, then p_u has a level curve edge.} works here except that instead of invoking Lemma~\ref{lemma:Points close to x have path to x in lambda.} we would invoke the gradient line version Lemma~\ref{lemma:Points close to x have gradient path to x.}.
\end{proof}

\bibliographystyle{plain}
\bibliography{refs}

\end{document}